\documentclass[10pt]{amsart}
      \usepackage{amsmath,amsfonts}
\def\rg{\hbox to 30pt{\rightarrowfill}}
\def\lg{\hbox to 30pt{\leftarrowfill}}

      \parskip\smallskipamount
          \newtheorem{theorem}{Theorem}[section]
      
      \newtheorem{proposition}[theorem]{Proposition}
      \newtheorem{corollary}[theorem]{Corollary}
      \newtheorem{lemma}[theorem]{Lemma}

      \newtheorem{remark}[theorem]{Remark}
      \makeatletter
      \@addtoreset{equation}{section}
      \makeatother

      \newcommand{\BB}{{\mathbb B}}
      \newcommand{\CC}{{\mathbb C}}
      \newcommand{\NN}{{\mathbb N}}
      
      \newcommand{\ZZ}{{\mathbb Z}}
      \newcommand{\DD}{{\mathbb D}}
      
      \newcommand{\FF}{{\mathbb F}}
      \newcommand{\TT}{{\mathbb T}}

      \newcommand{\cA}{{\mathcal A}}
      
      \newcommand{\cC}{{\mathcal C}}
      \newcommand{\cD}{{\mathcal D}}
      \newcommand{\cE}{{\mathcal E}}
      
      \newcommand{\cG}{{\mathcal G}}
      \newcommand{\cH}{{\mathcal H}}
      \newcommand{\cK}{{\mathcal K}}
      
      \newcommand{\cM}{{\mathcal M}}
      \newcommand{\cN}{{\mathcal N}}
      \newcommand{\cQ}{{\mathcal Q}}
      
      \newcommand{\cR}{{\mathcal R}}
      \newcommand{\cS}{{\mathcal S}}

      \newcommand{\cV}{{\mathcal V}}

      \newdimen\expt
      \def\boxit#1{\setbox0\hbox{$\displaystyle{#1}$}
            \hbox{\lower.4\expt
       \hbox{\lower3\expt\hbox{\lower\dp0
            \hbox{\vbox{\hrule height.4\expt
       \hbox{\vrule width.4\expt\hskip3\expt
            \vbox{\vskip3\expt\box0\vskip2\expt}%
       \hskip3\expt\vrule width.4\expt}\hrule height.4\expt}}}}}}
      
\begin{document}
       \pagestyle{myheadings}
      \markboth{ Gelu Popescu}{  Noncommutative polydomains,   Berezin transforms,   and operator model theory    }
      %\pagestyle{plain}
      %\begin{flushright}
       % \it Date of this draft: \today
      %\end{flushright}
      %\bigskip

      \title [    Similarity problems in  noncommutative polydomains ]
      {          Similarity problems in  noncommutative polydomains  }
        \author{Gelu Popescu}
     % \date{\today}
\date{September 20, 2014}
      \thanks{Research supported in part by an NSF grant}

      \address{Department of Mathematics, The University of Texas
      at San Antonio \\ San Antonio, TX 78249, USA}
      \email{\tt gelu.popescu@utsa.edu}

\begin{abstract}
In this paper we consider  several problems of joint similarity  to tuples of bounded linear operators in noncommutative polydomains and varieties associated with sets of noncommutative polynomials. We obtain analogues  of classical results such as Rota's model theorem for operators with spectral radius less than one, Sz.-Nagy characterization of operators similar to isometries (or unitary operators), and the refinement obtained by Foia\c s and by de Branges and Rovnyak for strongly stable contractions. We also provide analogues  of these results in the context of joint similarity of commuting tuples  of positive linear maps on the algebra of bounded linear operators on a separable Hilbert space. An important role in this paper is played by a class of noncommutative cones associated with positive linear maps, the Fourier type representation of their elements, and the constrained noncommutative Berezin transforms associated with these elements. It is shown that there is a intimate relation between the similarity problems and the existence of positive invertible  elements in these noncommutative cones and the corresponding Berezin kernels.

\end{abstract}

      \maketitle

\bigskip

\section*{Introduction}

Throughout this paper, we denote by $B(\cH)$ the algebra of bounded
linear operators on a separable  Hilbert space $\cH$. Let $B(\cH)^{n_1}\times_c\cdots \times_c B(\cH)^{n_k}$ be
   the set of all tuples  ${\bf X}:=({ X}_1,\ldots, { X}_k)$ in the cartesian product $B(\cH)^{n_1}\times\cdots \times B(\cH)^{n_k}$
     with the property that the entries of ${X}_s:=(X_{s,1},\ldots, X_{s,n_s})$  are commuting with the entries of
      ${X}_t:=(X_{t,1},\ldots, X_{t,n_t})$  for any $s,t\in \{1,\ldots, k\}$, $s\neq t$. Note that the operators $X_{s,1},\ldots, X_{s,n_s}$ are not necessarily commuting.  Denote by $\CC\left< Z_{i,j}\right>$  the algebra of all polynomials in noncommutative indeterminates $Z_{i,j}$, $i\in \{1,\ldots,k\}$, $j\in \{1,\ldots, n_i\}$.
 In an attempt to unify the multivariable operator model theory  for  ball-like domains and  commutative polydiscs,  we developed in \cite{Po-Berezin-poly} an operator  model  theory and a theory of free holomorphic functions on  {\it regular polydomains} of the form
$$
{\bf D_q^m}(\cH):=\left\{ {\bf X}=(X_1,\ldots, X_k)\in B(\cH)^{n_1}\times_c\cdots \times_c B(\cH)^{n_k}: \ {\bf \Delta_{q,X}^p}(I)\geq 0 \ \text{ for }\ {\bf 0}\leq {\bf p}\leq {\bf m}\right\},
$$
where ${\bf m}:=(m_1,\ldots, m_k)$ and ${\bf n}:=(n_1,\ldots, n_k)$ are in $\NN^k$ with  $\NN:=\{1,2,\ldots\}$,
 the {\it defect mapping} ${\bf \Delta_{q,X}^p}:B(\cH)\to  B(\cH)$ is defined by
$$
{\bf \Delta_{q,X}^p}:=\left(id -\Phi_{q_1, X_1}\right)^{p_1}\circ \cdots \circ\left(id -\Phi_{q_k, X_k}\right)^{p_k},
$$
and ${\bf q}=(q_1,\ldots, q_k)$ is  a $k$-tuple of positive regular polynomials $q_i\in \CC\left<Z_{i,1},\ldots, Z_{i,n_i}\right>$, i.e. all the coefficients  of $q_i$ are positive, the constant term is zero, and the coefficients of the linear terms $Z_{i,1},\ldots, Z_{i,n_i}$ are different from zero.  If the polynomial $q_i$ has the form $q_i=\sum_{\alpha} a_{i,\alpha} Z_{i,\alpha}$, the  completely positive linear map
$\Phi_{q_i,X_i}:B(\cH)\to B(\cH)$  is defined by setting $\Phi_{q_i,X_i}(Y):=\sum_{\alpha} a_{i,\alpha} X_{i,\alpha} Y X_{i,\alpha} ^*$ for $Y\in B(\cH)$.

 In \cite{Po-Berezin3}, we studied  noncommutative varieties in the polydomain ${\bf D_q^m}(\cH)$, given by
$$
\cV_\cQ(\cH):=\{{\bf X}=\{X_{i,j}\}\in {\bf D_q^m}(\cH):\ g({\bf X})=0 \text{ for all } g\in \cQ\},
$$
where $\cQ$ is a set of polynomials in noncommutative indeterminates
$Z_{i,j}$ which generates a nontrivial ideal in $\CC\left<Z_{i,j}\right>$.
We showed that there is a {\it universal model}
${\bf S}=\{{\bf S}_{i,j}\}$  for the {\it abstract  noncommutative variety}
   $${\bf \cV_\cQ}:=\{{\bf \cV_\cQ}(\cH):\ \cH \text { is a Hilbert space}\}
   $$
   such that $g({\bf S})=0$, $g\in \cQ$, acting on a subspace of a tensor product of full Fock spaces.   We studied  the universal model ${\bf S}$, its joint invariant subspaces and the representations of the universal operator algebras it generates: the {\it variety algebra} $\cA(\cV_\cQ)$, the Hardy algebra $F^\infty(\cV_\cQ)$, and   the $C^*$-algebra $C^*(\cV_\cQ)$. Using noncommutative Berezin transforms associated with each variety, we developed  an operator model theory and dilation theory for large classes of   varieties in  noncommutative polydomains.

In the present paper, we solve   several problems of joint similarity  to tuples of bounded operators in noncommutative  regular polydomains
$
{\bf D_q^m}(\cH)$ and varieties $\cV_\cQ(\cH)$ associated with sets $\cQ$  of noncommutative polynomials. We obtain analogues  of  the classical result of  Rota \cite{R} regarding  the model theorem for operators with spectral radius less than one,  the Sz.-Nagy  \cite{SzN} characterization of operators similar to isometries (or unitary operators), and the refinement obtained by Foia\c s \cite{Fo} and by de Branges and Rovnyak \cite{BR} for strongly stable contractions. We also provide analogues  of these results in the context of joint similarity of commuting tuples  of positive linear maps on the algebra of bounded linear operators on a separable Hilbert space.

If $\varphi:B(\cH)\to
B(\cH)$ is  a linear map we denote by  $\varphi^k$  the $k$ iterate of $\varphi$ with respect to the composition, i.e. $\varphi^k:=\varphi\circ \varphi^{k-1}$ and $\varphi^0:=id$, the identity map on $B(\cH)$.
 For  information  on  positive (resp. completely  positive or bounded maps),  we refer the reader to the excelent books by
  Paulsen \cite{Pa-book} and Pisier \cite{Pi-book}. Let $\Phi=(\varphi_1,\ldots, \varphi_k)$ be a $k$-tuple of
 positive linear maps on $B(\cH)$.
  For each ${\bf p}:=(p_1,\ldots, p_k)\in \ZZ_+^k$, where $\ZZ_+:=\{0,1,\ldots\}$, we define the linear map
  ${\bf \Delta}_{\Phi}^{\bf p}:B(\cH)\to B(\cH)$ by setting
 $$
 {\bf \Delta}_{\Phi}^{(p_1,\ldots,p_k)}={\bf \Delta}_{\Phi}^{\bf p}:= (id-\varphi_1)^{ p_1}\circ\cdots \circ(id-\varphi_k)^{ p_k}.
 $$
 If ${\bf p}:=(p_1,\ldots, p_k)\in \ZZ_+^k$ and  ${\bf s}:=(s_1,\ldots, s_k)\in \ZZ_+^k$, we set ${\bf p}\leq {\bf s}$ iff $p_i\leq s_i$ for all $i\in \{1,\ldots,k\}$.
Given
${\bf m}:=(m_1,\ldots, m_k)\in \NN^k$,   we  define the noncommutative cone
$$
\cC_{\geq}({\bf \Delta}_{\Phi}^{\bf m})^+:=\left\{X\in B(\cH):\  X\geq 0 \text{ and }
{\bf \Delta}_{\Phi}^{\bf p}(X)\geq 0 \text{ \rm  for } 0\leq {\bf p}\leq
 {\bf m}\right\}.
$$
In Section 1, we prove  that if each  positive map $\varphi_i$ is {\it pure}, i.e. $\varphi_i^s(I)\to 0$ weakly as $s\to\infty$, then
${\bf\Delta_\Phi^m}$ is a one-to-one map and each $X\in \cC_{\geq}({\bf \Delta}_{\Phi}^{\bf m})^+$ has a Fourier type representation
$$
X=\sum_{(s_1,\ldots,s_k)\in \ZZ_+^k}\left(\begin{matrix} s_1+m_1-1\\m_1-1\end{matrix}\right)\cdots \left(\begin{matrix} s_k+m_k-1\\m_k-1\end{matrix}\right)\varphi_1^{s_1}\circ \cdots \circ \varphi_k^{s_k}({\bf\Delta_\Phi^m}(X)),
$$
where
 the convergence of the series is in the weak operator topology. An important role in this paper is played by the  class $\cC_{\geq}({\bf \Delta}_{\Phi}^{\bf m})^+ $ of noncommutative cones associated with commuting positive linear maps $\Phi=(\varphi_1,\ldots, \varphi_k)$  and the Fourier type representations of their elements. Basic properties of these noncommutative cones are provided.

Let ${\bf q}:=(q_1,\ldots, q_k)$ be a $k$-tuple of positive regular polynomials $q_i\in \CC[Z_{i,1},\ldots, Z_{i,n_i}]$.
Consider  two tuples of operators ${\bf A}:=({ A}_1,\ldots, {A}_k)\in  B(\cH)^{n_1}\times\cdots \times B(\cH)^{n_k}$, where ${ A}_i:=(A_{i,1},\ldots, A_{i,n_i})\in B(\cH)^{n_i}$, and ${\bf B}:=({ B}_1,\ldots, {B}_k)\in  B(\cK)^{n_1}\times\cdots \times B(\cK)^{n_k}$, where ${ B}_i:=(B_{i,1},\ldots, B_{i,n_i})\in B(\cK)^{n_i}$. We say the ${\bf A}$ is jointly similar to ${\bf B}$ if there exists an invertible operator $Y:\cK\to \cH$
such that
$$
A_{i,j}=Y B_{i,j} Y^{-1}
$$
for all  $i\in \{1,\ldots,k\}$ and  $j\in \{1,\ldots, n_i\}$. We call $Y$ the similarity operator.

In Section 2, using some ideas from \cite{Ber},  \cite{Po-similarity-domains} and \cite{Po-Berezin3}, we introduce a class of {\it generalized constrained noncommutative  Berezin kernels} ${\bf K}_\omega$ associated with certain compatible tuples $\omega:=({\bf f,m, A},R,\cQ)$.
These kernels will play an important role in proving some of the similarity results, namely, that of similarity operators. Due to their explicit forms, we are able to estimate the magnitude of $\|Y\|\|Y^{-1}\|$ and provide von Neumann type inequalities. We introduce the  {\it constrained noncommutative Berezin transform} ${\bf
B}_\omega$
 associated with a compatible  tuple $\omega:=({\bf f,m, A},R,\cQ)$ to be the operator
 ${\bf B}_\omega: B(\cN_\cQ)\to B(\cH)$ given by
 $$
 {\bf B}_\omega[\chi]:={\bf K}_\omega^*[\chi\otimes I_\cR] {\bf K}_\omega,\qquad \chi\in B(\cN_\cQ).
 $$
 where $\cN_\cQ$ is an appropriate subspace of  a tensor product of full Fock spaces and  $\cR:=\overline{ R^{1/2}(\cH)}$.
 We prove  that
  the elements of the noncommutative cone  $C_{\geq}({\bf \Delta_{{\bf \Phi}_{\bf A}}^m})^+$, where
  ${\bf \Phi}_{\bf A}:=(\Phi_{q_1,A_1},\ldots, \Phi_{q_k,A_k})$,
are in one-to-one correspondence
 with the elements of a class
of extended noncommutative Berezin transforms. We will see throughout this paper that there is an intimate relation between the similarity problems and the existence of positive invertible  elements in these noncommutative  cones.

The fact that the unilateral shift on the Hardy space $H^2(\TT)$   plays  the role of {\it universal model} in $B(\cH)$
was discovered by Rota  \cite{R}.
 Rota's model theorem  asserts that any bounded linear operator on a Hilbert space with spectral radius less than one is similar to the adjoint of the unilateral shift of infinite multiplicity restricted to an invariant subspace. An analogue of this result was obtained by Herrero \cite{He} and Voiculescu \cite{Vo} for operators with spectrum   in a certain class of bounded open sets of the complex plane.
 Clark \cite{C} obtained a several variable version of Rota's model theorem for commuting strict contractions, and Ball  \cite{Ba} extended the result to a more general commutative multivariable setting. In the noncommutative multivariable setting, joint similarity to elements in  ball-like domains or their universal models were considered in \cite{Po-models}, \cite{Po-similarity}, \cite{Po-charact2}, and
 \cite{Po-similarity-domains}.

In Section 3, we obtain the
following
  analogue of Rota's  model  theorem  for  similarity to tuples of operators  in the  noncommutative
varieties.
Let $\cQ$ be a set of polynomials in indeterminates $\{Z_{i,j}\}$,  where $i\in \{1,\ldots,k\}$ and  $j\in \{1,\ldots, n_i\}$, and  let
 ${\bf A}:=({ A}_1,\ldots, {A}_k)\in  B(\cH)^{n_1}\times_c\cdots \times_c B(\cH)^{n_k}$, where ${ A}_i:=(A_{i,1},\ldots, A_{i,n_i})\in B(\cH)^{n_i}$ has the property that $q({\bf A})=0$ for any $ q\in \cQ$.
 If
$$
 \sum_{(s_1,\ldots,s_k)\in \ZZ_+^k}\left(\begin{matrix} s_1+m_1-1\\m_1-1\end{matrix}\right)\cdots \left(\begin{matrix} s_k+m_k-1\\m_k-1\end{matrix}\right)\Phi_{q_1,A_1}^{s_1}\circ \cdots \circ \Phi_{q_k, A_k}^{s_k}(I)\leq bI
$$
for some constant  $ b>0$, then there exists an invertible
operator $Y:  \cH\to \cG$ such that
$$
A_{i,j}^*=Y^{-1}[({\bf S}_{i,j}^*\otimes I_\cH)|_\cG]Y
$$
for all  $i\in \{1,\ldots,k\}$ and  $j\in \{1,\ldots, n_i\}$,
where
 $\cG\subseteq \cN_\cQ\otimes \cH$ is an invariant  subspace under  each operator ${\bf S}_{i,j}^*\otimes
 I_\cH$  and  ${\bf S}:=({\bf S}_1,\ldots, {\bf S}_k)$, with ${\bf S}_i:=({\bf S}_{i,1}\ldots, {\bf S}_{i,n_i})$, is the universal model associated with the abstract
noncommutative variety $\cV_Q$. In the particular case when $n_i=m_i=1$, $q_i=Z_i$, and $\cQ=\{0\}$, the universal model ${\bf S}_i$ is the multiplication by the coordinate function $z_i$ on   the Hardy space of the polydisc $H^2(\DD^k)$. As a consequence of the result above
  we obtain an analogue of Foia\c s \cite{Fo} (see also \cite{SzFBK-book}) and de Branges--Rovnyak \cite{BR}
  model theorem for pure tuples of operators in  the noncommutative variety $\cV_\cQ(\cH)$.

Rota \cite{R} also proved  that any bounded linear operator on a Hilbert space with spectral radius less then one is similar to a strict contraction.  In Section 3, we obtain an analogue of this result for noncommutative polydomains (see Theorem \ref{simi2}).
To give the reader some flavor of this result we state  it  in the particular case of polyballs, i.e. $m_i=1$ and $q_i:=Z_{i,1}+\cdots +Z_{i,n_i}$.
Let $\FF_{n_i}^+$ be the  free monoid on $n_i$ generators
$g_{1}^i,\ldots, g_{n_i}^i$ and the identity $g_{0}^i$.  We recall \cite{Po-models} that the joint spectral radius of a row contraction $T=[T_1\cdots T_n]$ is defined by $r(T):=\lim_{k\to \infty}\|\Phi_T^k(I)\|^{1/2k}$, where $\Phi_T(X):=\sum_{i=1}^n T_iXT_i^*$. We say that  $\pi_i:\FF_{n_i}^+\to B(\cH)$ is  a strictly row contractive  representation if its generators form a strict row contraction, i.e.
$\|[\pi_i(g_1^i)\cdots \pi_i(g_{n_i}^i)]\|<1$. We  denote the joint spectral radius of the row operator $[\pi_i(g_1^i)\cdots \pi_i(g_{n_i}^i)]$ by
$$
r(\pi_i):=r( \pi_i(g_1^i),\ldots, \pi_i(g_{n_i}^i))
$$
and call it the joint spectral radius of  $\pi_i$.
We prove that if $\pi_i:\FF_{n_i}^+\to B(\cH)$, $i\in \{1,\ldots,k\}$, are representations with commuting ranges and   $\sigma: \FF_{n_1}^+\times \cdots \times \FF_{n_k}^+\to \cH$  is  the direct product representation defined by
$$
\sigma(\alpha_1,\ldots, \alpha_k)=\pi_1(\alpha_1)\cdots \pi_k(\alpha_k),\qquad (\alpha_1,\ldots, \alpha_k)\in \FF_{n_1}^+\times \cdots \times \FF_{n_k}^+,
$$
then the following statements are equivalent:
\begin{enumerate}
\item[(i)] There is an invertible operator $Y \in B(\cH)$ such that $Y^{-1}\sigma(\cdot) Y$ is the direct product of strictly row contractive representations, i.e. $Y^{-1}\pi_i(\cdot) Y$ is a  strictly row contractive representation for each $i\in \{1,\ldots, k\}$.

\item[(ii)] $r(\pi_i)<1$  for each $i\in \{1,\ldots, k\}$.
\end{enumerate}
In \cite{Pi2}, Pisier proved that  there are commuting operators $T_1$, $T_2$  on a Hilbert space  which are each similar to a contraction, i.e. there are invertible operators $\xi_1, \xi_2$ such that $\xi_j^{-1} T_j \xi_j$ is a contraction for $j=1,2$, but such that $(T_1,T_2)$ is not jointly similar to a pair of contractions, i.e. there is no invertible operator $\xi$ such that $\xi^{-1} T_1 \xi$ and $\xi^{-1} T_1 \xi$ are contractions. The proof of this result uses some ideas from Pisier's remarkable  paper \cite{Pi} (see also \cite{DP}), where  he solves the long-standing Halmos' similarity problem \cite{H1}, \cite{H2}, as well as Paulsen's beautiful  similarity criterion \cite{Pa}.

We remark that, in the  particular case when $n_1=\cdots n_k=1$, the  above-mentioned Rota type result for polyballs  shows that
 a $k$-tuple of commuting operators $(C_1,\ldots, C_k)\in B(\cH)^k$ is jointly similar to a $k$-tuple of commuting strict contractions $(G_1,\ldots, G_k)\in B(\cH)$    if and only if
$$ r(C_i)<1,\qquad  i\in \{1,\ldots,k\},
$$
where $r(C_i)$ denotes the spectral radius of $C_i$. Rephrasing this result, on can see that for tuples of commuting operators similarity of each of them to  a strict contraction is equivalent to joint similarity to  strict contractions. In this case, we deduce the following   inequality
$$
\|[q_{s,t}(C_1,\ldots, C_k)]_{m\times m}\|\leq \sqrt{b} \sup_{|z_i| \leq 1}\|[q_{s,t}(z_1,\ldots, z_k)]_{m\times m}\|
$$
for any  matrix $[q_{s,t}]_{m\times m}$ of polynomials   in $k$ variables and any $m\in \NN$, where
$b=\prod_{i=1}^k \left(\sum_{s_i=0}^\infty \|C_i^{s_i}\|^2\right).
$
We remark that, in the particular case when
$\|C_i\|\leq r<1$ for $i\in \{1,\ldots, k\}$,  we obtain the inequality
$$
\|[q_{s,t}(C_1,\ldots, C_k)]_{m\times m}\|\leq \frac{1}{(1-r^2)^{k/2}} \sup_{|z_i| \leq 1}\|[q_{s,t}(z_1,\ldots, z_k)]_{m\times m}\|,
$$
which seems to be new if $k\geq 3$ and $m\geq 2$. We remark that, when $m=1$, the inequality above is an immediate consequence of Cauchy-Schwartz inequality. When $k=m=1$,  due to a result by Bombieri and Bourgain \cite{BB},
the constant $1/\sqrt{1-r^2}$ is best, in a certain sense,  as $r\to 1$.

In 1947, Sz.-Nagy \cite{SzN} found necessary and sufficient conditions
for an operator to be similar to a unitary operator. As a consequence,  an
operator $T$ is
similar to an isometry if and only if there are constants $a,b>0$ such that
$$
a\|h\|^2\leq  \|T^nh\|^2\leq b\|h\|^2,\qquad  h\in \cH, n\in \NN.
$$
In Section 4,   we obtain an analogue of  Sz.-Nagy's similarity result
 for noncommutative  polydomains (see Theorem \ref{simi}). We shall    mention the corresponding result in the particular case of the polyball.
We say that  $\pi_i:\FF_{n_i}^+\to B(\cH)$ is  a Cuntz representation if its generators form a row operator matrix
$[\pi_i(g_1^i)\cdots \pi_i(g_{n_i}^i)]$  which is a  unitary from the direct sum $\cH^{(n_i)}:=\cH\oplus\cdots \oplus \cH$ to $\cH$.
Let $\pi_i:\FF_{n_i}^+\to B(\cH)$, $i\in \{1,\ldots,k\}$, be representations with commuting ranges and let $\sigma: \FF_{n_1}^+\times \cdots \times \FF_{n_k}^+\to \cH$ be the direct product representation.
Then there is an invertible operator $Y \in B(\cH)$ such that $Y^{-1}\sigma(\cdot) Y$ is the direct product of   Cuntz representations, i.e. $Y^{-1}\pi_i(\cdot) Y$ is a Cuntz type representation for each $i\in \{1,\ldots, k\}$, if and only if
the generators of each representation $\pi_i$  form a one-to-one row operator matrix $[\pi_i(g_1^i)\cdots \pi_i(g_{n_i}^i)]$ and
 there exist constants $0<c\leq d$ such
$$
c\|h\|^2\leq \|\sigma(\alpha_1,\ldots, \alpha_k)h\|^2\leq d\|h\|^2,\qquad h\in \cH,
$$
for any $(\alpha_1,\ldots, \alpha_k)\in \FF_{n_1}^+\times \cdots \times \FF_{n_k}^+$.

 In the particular case when $n_1=\cdots= n_k=1$, we also prove  that  a $k$-tuple of commuting operators $(C_1,\ldots, C_k)\in B(\cH)^k$ is jointly similar to a $k$-tuple of commuting isometries $(V_1,\ldots, V_k)\in B(\cH)^k$ if and only if there are constants $0<c\leq d$ such that
$$
c\|h\|^2\leq \|C_1^{s_1}\cdots C_k^{s_k}h\|^2\leq d\|h\|^2, \quad h\in \cH,
$$
for any $s_1,\ldots, s_k\in \ZZ^+$. Moreover, there is an  invertible operator $\xi:\cH\to \cH$  such that $V_i=\xi C_i \xi^{-1}$  for $i\in \{1,\ldots, k\}$ and $\xi$ is in the von Neumann algebra generated by $C_1,\ldots, C_n$ and the identity.
As a consequence, we deduce the well-known result  of Dixmier (see  \cite{Di}, \cite{Da}) that any uniformly bounded representation $u:\ZZ^k\to B(\cH)$ is similar to a unitary representation.

In Section 5, we provide analogues  of all the similarity results presented in the previous sections in the context of joint similarity of commuting tuples  of positive linear maps on the algebra of bounded linear operators on a separable Hilbert space.

We remark that all the similarity results regarding the noncommutative polydomain ${\bf D_q^m}(\cH)$
and the noncommutative variety $\cV_\cQ(\cH)$
are presented in the more general  setting where  the $k$-tuple ${\bf q}=(q_1,\ldots, q_k)$  of positive regular polynomials is repaced by  a $k$-tuple ${\bf f}:=(f_1,\ldots, f_k)$  of positive regular free holomorphic functions.

\bigskip

\section{Noncommutative cones associated with  positive linear maps}

In this section, we provide basic properties for certain  noncommutative sets associated with commuting positive linear maps  and  obtain a Fourier type representation for  their elements. These results are needed in the next sections.

  Let $\Phi=(\varphi_1,\ldots, \varphi_k)$ be a $k$-tuple of
 positive linear maps on $B(\cH)$.
  For each ${\bf p}:=(p_1,\ldots, p_k)\in \ZZ_+^k$, where $\ZZ_+:=\{0,1,\ldots\}$, we define the linear map
  ${\bf \Delta}_{\Phi}^{\bf p}:B(\cH)\to B(\cH)$ by setting
 $$
 {\bf \Delta}_{\Phi}^{(p_1,\ldots,p_k)}={\bf \Delta}_{\Phi}^{\bf p}:= (id-\varphi_1)^{ p_1}\circ\cdots \circ(id-\varphi_k)^{ p_k}.
 $$
  Given  $A,B\in B(\cH)$ two self-adjoint  operators, we say that
$A<B$ if $B-A$ is positive and invertible, i.e., there exists a
constant $\gamma>0$ such that $\left<(B-A)h,h\right>\geq
\gamma\|h\|^2$ for any $h\in \cH$.
Let
${\bf m}:=(m_1,\ldots, m_k)\in \NN^k$  and define the following sets:
\begin{equation*}
\begin{split}
\cC_{\geq}({\bf \Delta}_{\Phi}^{\bf m})^{sa}&:=\left\{X\in B(\cH):\  X=X^* \text{ and }
{\bf \Delta}_{\Phi}^{\bf p}(X)\geq 0 \text{ \rm  for } 0\leq {\bf p}\leq {\bf m}, {\bf p}\neq 0\right\},\\
\cC_{=}({\bf \Delta}_{\Phi}^{\bf m})&:=\left\{X\in B(\cH):\
{\bf \Delta}_{\Phi}^{\bf p}(X)= 0 \text{ \rm  for } 0\leq {\bf p}\leq {\bf m}, {\bf p}\neq 0\right\},\\
\cC_{>}({\bf \Delta}_{\Phi}^{\bf m})^{sa}&:=\left\{X\in B(\cH):\  X=X^* \text{ and }
{\bf \Delta}_{\Phi}^{\bf p}(X)> 0 \text{ \rm  for } 0\leq {\bf p}\leq {\bf m}, {\bf p}\neq 0\right\}.
\end{split}
\end{equation*}
The definitions for the sets $\cC_{\geq}({\bf \Delta}_{\Phi}^{\bf m})^+$, $\cC_{=}({\bf \Delta}_{\Phi}^{\bf m})^{sa}$, $\cC_{=}({\bf \Delta}_{\Phi}^{\bf m})^+$, and $\cC_{>}({\bf \Delta}_{\Phi}^{\bf m})^+$ are clear. We also introduce the set
$\cC_{\geq}^{pure}({\bf \Delta}_{\Phi}^{\bf m})^{+}$  of all $X\in \cC_{\geq}({\bf \Delta}_{\Phi}^{\bf m})^{+}$ with the property that, for each $i\in \{1,\ldots, k\}$, $\varphi_i^s(X)\to 0$ weakly as $s\to \infty$.

A linear map $\varphi:B(\cH)\to
B(\cH)$ is called {\it power bounded} if there exists a constant $M>0$
such that $\|\varphi^k\|\leq M$ for any $k\in \NN$, where $\varphi^k$ is the $k$ iterate of $\varphi$ with respect to the composition. We say that a $k$-tuple $\Phi=(\varphi_1,\ldots, \varphi_k)$  of  linear maps on $B(\cH)$
is {\it commuting} if
 $\varphi_i \circ\varphi_j=\varphi_j \circ\varphi_i$ for  $i,j\in \{1,\ldots,k\}$.
A positive linear map $\varphi$ on $B(\cH)$ is called {\it pure} if $\varphi^p(I)\to 0$ weakly as $p\to\infty$.

\begin{proposition} \label{Delta-ineq} Let
$\Phi=(\varphi_1,\ldots, \varphi_k)$ be a $k$-tuple of commuting  positive linear maps on $B(\cH)$ and let
${\bf m}\in \NN^k$.
\begin{enumerate}
\item[(i)]
     If   $Y\in \cC_{\geq}({\bf \Delta}_{\Phi}^{\bf m})^{sa}$ and  $0\neq{\bf q}\in \ZZ_+^k$ is such that  ${\bf q}\leq {\bf m}$, then
$$
0\leq{\bf \Delta}_{\Phi}^{\bf m}(Y)\leq {\bf \Delta}_{\Phi}^{\bf q}(Y).
$$
If, in addition, $Y\geq 0$ and  ${\bf \Delta}_{\Phi}^{\bf m}(Y)>0$, then
$$Y\in \cC_{>}({\bf \Delta}_{\Phi}^{\bf m})^+ \quad \text{ and } \quad Y>0.$$
\item[(ii)] If each $\varphi_i$ is pure and  $Y\in B(\cH)$ is a self-adjoint operator   with
 ${\bf \Delta}_{\Phi}^{\bf m}(Y)\geq 0$,
then   $$Y\in \cC_{\geq}^{pure}({\bf \Delta}_{\Phi}^{\bf m})^+.$$
\item[(iii)] If  $\Psi=(\psi_1,\ldots, \psi_k)$  is a $k$-tuple of  commuting  positive linear maps on $B(\cH)$ such that $\psi_i\leq \varphi_i$ and $\psi_i\circ\varphi_j=\varphi_j\circ\psi_i$ for any $i,j\in \{1,\ldots,k\}$, then
$$\cC_{\geq}({\bf \Delta}_{\Phi}^{\bf m})^+\subseteq \cC_{\geq}({\bf \Delta}_{\Psi}^{\bf m})^+.
$$
\end{enumerate}
\end{proposition}
\begin{proof} Set ${\bf m}:=(m_1,\ldots, m_k)\in \NN^k$ and ${\bf m}':=(m_1-1, m_2,\ldots, m_k)$.
Since ${\bf \Delta}_{\Phi}^{{\bf m}' }(Y)\geq 0$  and $\varphi_1$ is a positive map,  we deduce that
$$
0\leq{\bf \Delta}_{\Phi}^{\bf m}(Y)={\bf \Delta}_{\Phi}^{{\bf m}'}(Y) -\varphi_1({\bf \Delta}_{\Phi}^{{\bf m}'}(Y))\leq {\bf \Delta}_{\Phi}^{{\bf m}'}(Y)
$$
 Using  the fact that $\varphi_i \circ\varphi_j=\varphi_j \circ\varphi_i$ for $i,j\in \{1,\ldots,k\}$, one can continue this process and show that
 $0\leq{\bf \Delta}_{\Phi}^{\bf m}(Y)\leq {\bf \Delta}_{\Phi}^{\bf q}(Y)$ for any ${\bf q}\in \ZZ_+^k$ with ${\bf q}\leq {\bf m}$ and ${\bf q}\neq 0$. Similarly, if $Y\geq 0$ and  ${\bf \Delta}_{\Phi}^{\bf m}(Y)>0$,  we deduce that
  $0<{\bf \Delta}_{\Phi}^{\bf m}(Y)\leq {\bf \Delta}_{\Phi}^{\bf q}(Y)\leq Y$ and  $Y\in \cC_{>}({\bf \Delta}_{\Phi}^{\bf m})^+$.

To prove (ii),   set  ${\bf m}':=(m_1-1, m_2,\ldots, m_k)$ and note that due to the fact that  ${\bf \Delta}_{\Phi}^{\bf m}(Y)\geq 0$ and $\varphi_1$ is a positive linear map, we have
$$0\leq {\bf \Delta}_{\Phi}^{\bf m}(Y)={\bf \Delta}_{\Phi}^{{\bf m}'}(Y)-\varphi_1({\bf \Delta}_{\Phi}^{{\bf m}'}(Y)).
$$
Hence,  we deduce that $\varphi_1^p({\bf \Delta}_{\Phi}^{{\bf m}'}(Y))\leq {\bf \Delta}_{\Phi}^{{\bf m}'}(Y)$ for any $p\in \NN$. Since ${\bf \Delta}_{\Phi}^{{\bf m}'}(Y)$ is a self-adjoint operator, we have
$$
-\|{\bf \Delta}_{\Phi}^{{\bf m}'}(Y)\| \varphi_1^p(I)\leq \varphi_1^p({\bf \Delta}_{\Phi}^{{\bf m}'}(Y))\leq \|{\bf \Delta}_{\Phi}^{{\bf m}'}(Y)\| \varphi_1^p(I).
$$
Now, taking into account that  $\varphi_i^p(I)\to 0$ weakly as $p\to\infty$, we deduce that $\varphi_1^p({\bf \Delta}_{\Phi}^{{\bf m}'}(Y))\to 0$ as $p\to\infty$, which leads to   ${\bf \Delta}_{\Phi}^{{\bf m}'}(Y)\geq 0$. Using the commutativity of $\varphi_1,\ldots, \varphi_k$, one can continue this process and obtain  $Y\in \cC_{\geq}({\bf \Delta}_{\Phi}^{\bf m})^+$. Since $\varphi_i^p(I)\to 0$ weakly as $p\to\infty$ and
$$
-\|Y\|\varphi_i^p(I)\leq \varphi_i^p(Y)\leq \|Y\|\varphi_i^p(I),
$$
we deduce that $\varphi_i^p(Y)\to 0$, as $p\to\infty$.

Now, we prove (iii). Note that if $G\in B(\cH)$, $G\geq 0$, then, for each $i\in \{1,\ldots, k\}$,
 \begin{equation}
 \label{ine2}
 (id- \varphi_i)(G)\geq 0\quad \implies  \quad (id-\psi_i )(G)\geq 0, \quad r\in [0,1].
 \end{equation}
  Assume that $Y\in \cC_{\geq}({\bf \Delta}_{\Phi}^{\bf m})^+$. If ${\bf p}\in \ZZ_+^k$ with  ${\bf p}\geq e_1:=(1,0,\ldots,0)\in \ZZ_+^k$, then
 $(id- \varphi_1)({\bf \Delta}_{\Phi }^{{\bf p}-e_1}(Y))\geq 0$
 for any ${\bf p}\in \ZZ_+^k$  with $e_1\leq {\bf p}\leq {\bf m}$.
 Consequently, due to relation \eqref{ine2}, we have
 \begin{equation}
 \label{ine3}
 (id- \psi_1)({\bf \Delta}_{\Phi }^{{\bf p}-e_1}(Y))\geq 0
 \end{equation}
 for any ${\bf p}\in \ZZ_+^k$  with $e_1\leq {\bf p}\leq {\bf m}$. Due to the commutativity of  the maps $\varphi_1,\ldots, \varphi_k, \psi_1,\ldots, \psi_k$, the latter inequality is equivalent to
 \begin{equation*}
 (id- \varphi_1)({\bf \Delta}_{\Phi }^{{\bf p}-2e_1}\circ(id- \psi_1)(Y))\geq 0
 \end{equation*}
 for   any ${\bf p}\in \ZZ_+^k$  with $2e_1\leq {\bf p}\leq {\bf m}$.
 Due to \eqref{ine3}, we have ${\bf \Delta}_{\Phi }^{{\bf p}-2e_1}\circ(id- \psi_1)(Y)\geq 0$ and, applying again relation \eqref{ine2}, we deduce that
  \begin{equation*}
 (id- \varphi_1)({\bf \Delta}_{\Phi }^{{\bf p}-3e_1}\circ(id-\psi_1 )^2(Y))\geq 0
 \end{equation*}
 for any   ${\bf p}\in \ZZ_+^k$  with $3e_1\leq {\bf p}\leq {\bf m}$. Continuing this process, we obtain the inequality
 $$
 (id- \varphi_2)^{p_2}\circ\cdots \circ(id- \varphi_k)^{p_k}\circ(id-\psi_1 )^{p_1}(Y)\geq 0
 $$
 for any ${\bf p}\in \ZZ_+^k$  with $e_1\leq {\bf p}\leq {\bf m}$.
Similar arguments lead to the inequality  $ {\bf \Delta}_{\Psi }^{\bf p}(Y)\geq 0$ for any ${\bf p}\in \ZZ_+^k$  with $0\leq {\bf p}\leq {\bf m}$ and ${\bf p}\neq 0$. Therefore, $Y\in \cC_{\geq}({\bf \Delta}_{\Psi}^{\bf m})^+$.
The proof is complete.
\end{proof}

 If  $\phi_1,\ldots, \phi_k$ are  positive linear maps on $B(\cH)$, $p\in \NN$, and $i\in \{1,\ldots, k\}$, we define
\begin{equation*}
\begin{split}
\Lambda_i^{[1]}(Y)&:=\sum_{s_i=0}^\infty \phi_i^{s_i}(Y)\quad \text{ and}\\
\Lambda_i^{[p]}(Y)&:=\sum_{s_i=0}^\infty \phi_i^{s_i}(\Lambda_i^{[p-1]}(Y)), \quad p\geq2,
\end{split}
\end{equation*}
for those $Y\in B(\cH)$ for which all the series converge in the weak operator topology.

\begin{theorem}\label{reproducing} Let ${\bf m}=(m_1,\ldots, m_k)\in \NN^k$ and let  $\phi=(\phi_1,\ldots, \phi_k)$ be a $k$-tuple of positive linear maps on $B(\cH)$ such that each $\phi_i$ is pure. Then
${\bf\Delta_\phi^m}:=(id-\phi_1)^{m_1}\circ \cdots \circ (id -\phi_k)^{m_k}$ is a one-to-one map and each $X\in B(\cH)$ has the representation
$$
X= \Lambda_k^{[m_k]}\left(\cdots \left(\Lambda_1^{[m_1]}({\bf \Delta_{\phi}^m}(X))\right)\right),
$$
where the iterated series converge   in the weak operator topology.  If, in addition, ${\bf\Delta_\phi^m}(X)\geq 0$, then

$$
X=\sum_{(s_1,\ldots,s_k)\in \ZZ_+^k}\left(\begin{matrix} s_1+m_1-1\\m_1-1\end{matrix}\right)\cdots \left(\begin{matrix} s_k+m_k-1\\m_k-1\end{matrix}\right)\phi_1^{s_1}\circ \cdots \circ \phi_k^{s_k}({\bf\Delta_\phi^m}(X)),
$$
where
 the convergence of the series is in the weak operator topology.
\end{theorem}
\begin{proof} We use the notation ${\bf \Delta}_\phi^{(m_1,\ldots, m_k)}:={\bf\Delta_\phi^m}$ when we need to emphasize the coordinates of ${\bf m}$. Note that
\begin{equation}
\label{rec}
\sum_{s_1=0}^{q_1} \phi_1^{s_1}({\bf \Delta_{\phi}^m}(X))
= {\bf \Delta}_\phi^{(m_1-1,m_2,\ldots, m_k)}(X)-\phi_1^{q_1+1}({\bf \Delta}_\phi^{(m_1-1,m_2,\ldots, m_k)}(X)).
\end{equation}
If $Z\in B(\cH)$ is a positive operator and $x,y\in \cH$, the Cauchy-Schwarz inequality implies
$$
\left|\left<\phi_i^{q_i}(Z)x, y\right>\right|\leq \|Z\|\left<\phi_i^{q_i}(I)x, x\right>^{1/2} \left<\phi_i^{q_i}(I)y, y\right>^{1/2},\qquad q_i\in \NN.
$$
Since  $\phi_i^{q_i}(I) \to 0$ weakly as $q_i\to \infty$, we deduce that $\left<\phi_i^{q_i}(Z)x, y\right>\to 0$ as $q_i\to\infty$. Taking into account that any bounded linear operator  is a linear combination of positive operators, we conclude that the convergence above holds for any $Z\in B(\cH)$. Passing to the limit in  relation \eqref{rec} as $q_1\to\infty$, we obtain
$$
\sum_{s_1=0}^{\infty} \phi_1^{s_1}({\bf \Delta_{\phi}^m}(X))
= {\bf \Delta}_\phi^{(m_1-1,m_2,\ldots, m_k)}(X).
$$
Similarly, we obtain
$$
\sum_{s_1=0}^{\infty} \phi_1^{s_1}({\bf \Delta}_{\phi}^{(m_1-1,m_2,\ldots, m_k)}(X))
= {\bf \Delta}_\phi^{(m_1-2,m_2\ldots, m_k)}(X)  $$
and, continuing this process,
$$\sum_{s_1=0}^{\infty} \phi_1^{s_1}({\bf \Delta}_{\phi}^{(1,m_2\ldots, m_k)}(X))
= {\bf \Delta}_\phi^{(0,m_2\ldots, m_k)}(X).
$$
Putting together these relations, we deduce that
$
\Lambda_1^{[m_1]}({\bf \Delta_{\phi}^m}(X))={\bf \Delta}_\phi^{(0,m_2\ldots, m_k)}(X).
$
Similar arguments show that
$
\Lambda_2^{[m_2]}({\bf \Delta}_\phi^{(0,m_2\ldots, m_k)}(X))={\bf \Delta}_\phi^{(0,0,m_3\ldots, m_k)}(X)
$
and, eventually, we get
$$
\Lambda_k^{[m_k]}({\bf \Delta}_\phi^{(0,\ldots,0, m_k)}(X))={\bf \Delta}_\phi^{(0,\ldots, 0)}(X)=X.
$$
Putting these relations together we obtain the first  equality in the theorem, which implies  that ${\bf \Delta^m_\phi}$ is a one-to-one map.
 Now, we assume that ${\bf\Delta_\phi^m}(X)\geq 0$. Then the multi-sequence of positive operators
$$
\sum_{s_k^{(1)}=0}^{q_k^{(1)}}\cdots \sum_{s_k^{(m_k)}=0}^{q_k^{(m_k)}}\cdots
\sum_{s_1^{(1)}=0}^{q_1^{(1)}}\cdots \sum_{s_1^{(m_1)}=0}^{q_1^{(m_1)}}
\phi_k^{s_k^{(1)} }\circ \cdots \circ \phi_k^{ s_k^{(m_k)}}\circ\cdots \circ
\phi_1^{s_1^{(1)} }\circ \cdots \circ \phi_1^{ s_1^{(m_1)}} ({\bf\Delta_\phi^m}(X))
$$
is increasing with respect to each of the indexes $q_k^{(1)},\ldots, q_k^{(m_k)},\ldots, q_1^{(1)},\ldots, q_1^{(m_1)}\in \ZZ_+$. Using the first part of the theorem, we  deduce that
$$
\sum \phi_k^{s_k^{(1)}+\cdots +s_k^{(m_k)}}\circ\cdots \circ\phi_1^{s_1^{(1)}+\cdots +s_1^{(m_1)}} ({\bf\Delta_\Phi^m}(X))=X,
$$
where
the summation is taken over all tuples $(s_1^{(1)},\ldots, s_1^{(m_1)},\ldots, s_k^{(1)},\ldots, s_k^{(m_k)})\in \ZZ_+^{m_1+\cdots +m_k}$ and the convergence is in the weak operator topology. Note that, for each $i\in \{1,\ldots,k\}$ and $s_i\in \ZZ_+$, the equation $s_i^{(1)}+\cdots +s_i^{(m_i)}=s_i$  has $\left(\begin{matrix} s_i+m_i-1\\m_i-1\end{matrix}\right)$  distinct solutions in $\ZZ_+^{m_i}$. Combining this fact with the equality above, one can complete the proof.
\end{proof}

\begin{theorem}\label{reproducing2} Let ${\bf m}=(m_1,\ldots, m_k)\in \NN^k$ and let  $\Phi=(\phi_1,\ldots, \phi_k)$ be a $k$-tuple of commuting  positive linear maps on $B(\cH)$ such that each $\phi_i$ is weakly continuous on bounded sets. If $X$ is in the noncommutative cone $\cC_{\geq}({\bf \Delta}_{\Phi}^{\bf m})^+$, then
$$\lim_{q_k\to\infty}\ldots \lim_{q_1\to\infty}  (id-\phi_k^{q_k})\circ\cdots \circ(id-\phi_1^{q_1})(X)
$$
coincides with
 $$\sum_{(s_1,\ldots,s_k)\in \ZZ_+^k}\left(\begin{matrix} s_1+m_1-1\\m_1-1\end{matrix}\right)\cdots \left(\begin{matrix} s_k+m_k-1\\m_k-1\end{matrix}\right)\phi_1^{s_1}\circ \cdots \circ \phi_k^{s_k}({\bf\Delta_\Phi^m}(X)),
$$
where
 the convergence of the series is in the weak operator topology.
\end{theorem}
\begin{proof}
 Let  $Y\in B(\cH)$ be a positive operator such that
   $$
  {\bf \Delta_{\Phi}^p}(Y)  := (id-\phi_1)^{p_1}\circ\cdots \circ(id-\phi_k)^{p_k}(Y)\geq 0
    $$
 for any  ${\bf p}:=(p_1,\ldots, p_k)\in \ZZ_+^k$ with $p_i\in \{0,1,\ldots, m_i\}$ and
       $i\in \{1,\ldots,k\}$.
Fix $i\in \{1,\ldots,k\}$ and assume that $1\leq p_i\leq m_i$.  Then, due to the commutativity of  $\phi_1,\ldots, \phi_k$, we have
$$
(id-\phi_i){\bf \Delta_{\Phi}^{p-e_i}}(Y) ={\bf \Delta_{\Phi}^{p}}(Y)\geq 0,
$$
where $\{{\bf e}_i\}_{i=1}^k$ is the canonical basis in $\CC^k$. Hence, and using Proposition \ref{Delta-ineq} part (i), we have
$$
0\leq \phi_i({\bf \Delta_{\Phi}^{p-e_i}}(Y))\leq {\bf \Delta_{\Phi}^{p-e_i}}(Y)\leq Y,
$$
which proves that $\{\phi_i^s({\bf \Delta_{\Phi}^{p-e_i}}(Y))\}_{s=0}^\infty$ is a decreasing sequence of positive operators which is convergent in the weak operator topology.  Note also that, due to the fact that $0\leq \phi_i(Y)\leq Y$, the sequence $\{\phi_i^s(Y)\}_{s=0}^\infty$ is decreasing and convergent in the weak operator topology.  Since $\phi_i$  is WOT-continuous on bounded sets and  $\phi_1,\ldots, \phi_k$ are commuting,
  we deduce that
\begin{equation}
\label{wot}
\lim_{s\to\infty}\phi_i^s({\bf \Delta_\Phi^{p-e_i}}(Y))
={\bf \Delta_\Phi^{p-e_i}}\left(\lim\limits_{s\to\infty}\phi_i^s(Y)\right).
\end{equation}
 Then we have
 \begin{equation*}
 \begin{split}
 \Lambda_i^{[1]}({\bf \Delta_\Phi^p}(Y))&:=\sum_{s=0}^\infty \phi_{i}^{s}({\bf \Delta_\Phi^p}(Y))
 =\sum_{s=0}^\infty \phi_{i}^{s}\left[ {\bf \Delta_\Phi^{p-e_i}}(Y)
 -\phi_{i}({\bf \Delta_\Phi^{p-e_i}}(Y))\right]\\
 &={\bf \Delta_\Phi^{p-e_i}}(Y)- \lim_{q_i\to \infty}
 \phi_{i}^{q_i}({\bf \Delta_\Phi^{p-e_i}}(Y))\leq {\bf \Delta_\Phi^{p-e_i}}(Y)\leq Y.
 \end{split}
 \end{equation*}
 Due to relation \eqref{wot}, and the WOT-continuity and commutativity of $\phi_1,\ldots, \phi_k$, we deduce that
 $$
 0\leq \Lambda_i^{[1]}({\bf \Delta_\Phi^p}(Y))={\bf \Delta_\Phi^{p-e_i}}\left(Y- \lim_{q_i\to \infty}
 \phi_{i}^{q_i}(Y)\right), \qquad {\bf p}\leq {\bf m}, 1\leq p_i.
 $$
Define
$\Lambda_{i}^{[j]}({\bf \Delta_\Phi^p}(Y))
:=\sum_{s=0}^\infty \phi_{i}^s(\Lambda_{i}^{[j-1]}({\bf \Delta_\Phi^p}(Y)))$,
 where $j=2,\ldots p_i$.
Inductively, we can prove that
\begin{equation}
\label{induct}
0\leq \Lambda_{i}^{[j]}({\bf \Delta_\Phi^p}(Y))={\bf \Delta_\Phi^{ p-{\it j}e_i}}
\left(Y- \lim_{q_j\to \infty}
 \phi_{i}^{q_j}(Y)\right)\leq {\bf \Delta_\Phi^{ p-{\it j}e_i}}(Y)\leq Y,
 \qquad j\leq p_i.
\end{equation}
Indeed, if $j\leq p_i-1$ and  setting $Z:=Y- \lim\limits_{q_j\to \infty}
 \phi_{i}^{q_j}(Y)$, relation
\eqref{induct} implies
\begin{equation*}
\begin{split}
\Lambda_{i}^{[j+1]}({\bf \Delta_\Phi^p}(Y))
&=\lim_{q_{j+1}\to \infty}\sum_{s=0}^{q_{j+1}} \phi_{i}^s\left[{\bf \Delta_\Phi^{ p-{\it j}e_i}}(Z)\right] \\
&=
{\bf \Delta_\Phi^{ p-\text{$(j+ 1)$}e_i}}\left[
Z- \lim_{q_{j+1}\to \infty} \phi_{i}^{q_{j+1}}(Z)\right]\\
&=
{\bf \Delta_\Phi^{ p-\text{$(j+ 1)$}e_i}}(Z)-{\bf \Delta_\Phi^{ p-\text{$(j+ 1)$}e_i}}\left( \lim_{q_{j+1}\to \infty} \phi_{i}^{q_{j+1}}(Z) \right).
\end{split}
\end{equation*}
On the other hand, we have
 \begin{equation*}
 \begin{split}
 \lim_{q_{j+1}\to \infty} \phi_{i}^{q_{j+1}}(Z)
 &=\lim_{q_{j+1}\to \infty} \phi_{i}^{q_{j+1}}
 \left(Y- \lim\limits_{q_j\to \infty}
 \phi_{i}^{q_j}(Y)\right)\\
 &=
 \lim_{q_{j+1}\to \infty} \phi_{i}^{q_{j+1}}(Y)-
 \lim_{q_{j+1}\to \infty}\lim_{q_{j}\to \infty} \phi_{i}^{q_{j+1}}\left( \phi_{i}^{q_{j}}(Y)\right)=0.
 \end{split}
 \end{equation*}
Combining these results, we obtain
$$
\Lambda_{i}^{[j+1]}({\bf \Delta_\Phi^p}(Y)) =
{\bf \Delta_\Phi^{ p-\text{$(j+ 1)$}e_i}}\left( Y- \lim\limits_{q_j\to \infty}
 \phi_{i}^{q_j}(Y)\right)\leq {\bf \Delta_\Phi^{ p-\text{$(j+ 1)$}e_i}}(Y)\leq Y,
$$
for any  ${\bf p}:=(p_1,\ldots, p_k)\in \ZZ_+^k$ with ${\bf p}\leq {\bf m}$ and  $p_i\geq 1$,
which proves our assertion. When $j=p_i$, relation \eqref{induct} becomes
\begin{equation*}
\label{pi}
0\leq \Lambda_{i}^{[p_i]}({\bf \Delta_\Phi^p}(Y)) =
{\bf \Delta_\Phi^{ p-\text{$p_i$}e_i}}\left( Y- \lim\limits_{q_i\to \infty}
 \phi_{i}^{q_i}(Y)\right)\leq Y.
\end{equation*}
Due to the  results above,  we have
\begin{equation}
\label{Psi}
\begin{split}
0\leq \Lambda_{i}^{[m_i]}({\bf \Delta_\Phi^p}(Y))
&={\bf \Delta_\Phi^{ p-\text{$m_i$}e_i}}
\left(Y-\lim_{q_{i}\to \infty} \phi_{i}^{q_{i}} (Y)\right)\\
&
\leq {\bf \Delta_\Phi^{ p-\text{$m_i$}e_i}}(Y)\leq Y,
\end{split}
\end{equation}
for any  ${\bf p}:=(p_1,\ldots, p_k)\in \ZZ_+^k$ with ${\bf p}\leq {\bf m}$ and
 $p_i=m_i$.
 Applying relation \eqref{Psi} in the  particular case when $i=1$, $p_1=m_1$, and $Y=X$, we have
$$
0\leq \Lambda_{1}^{[m_1]}({\bf \Delta_\Phi^{p'}}(X))
={\bf \Delta_\Phi^{ p'-\text{$m_1$}e_1}}\left(X-\lim_{q_{1}\to \infty} \phi_{1}^{q_{1}} (X)\right)
\leq {\bf \Delta_\Phi^{ p'-\text{$m_1$}e_1}}(X)\leq X
$$
for any ${\bf p}'=(m_1,p_2,\ldots, p_k)$ with ${\bf p}'\leq {\bf m}$. Hence and
using again relation \eqref{Psi}, when $i=2$, ${\bf p}=(0,m_2, p_3\ldots, p_k)$, and
$Y=X-\lim_{q_{1}\to \infty} \phi_{1}^{q_{1}} (X)\geq 0$,  we obtain
\begin{equation*}
\begin{split}
0\leq \Lambda_2^{[m_2]}\left({\bf \Delta_\Phi^{ p''-\text{$m_1$}e_1}}\left(X-\lim_{q_{1}\to \infty} \phi_{1}^{q_{1}} (X)\right)\right)
&={\bf \Delta_\Phi^{ p''-\text{$m_1$}e_1-\text{$m_2$}e_2}}
\lim_{q_{2}\to \infty}\lim_{q_{1}\to \infty}\left(id- \phi_{2}^{q_{2}}\right)\circ
\left(id- \phi_{1}^{q_{1}}\right)(X)\\
&\leq {\bf \Delta_\Phi^{ p''-\text{$m_1$}e_1-\text{$m_2$}e_2}}(X)\leq X
\end{split}
\end{equation*}
for any ${\bf p}''=(m_1,m_2, p_3,\ldots, p_k)$. Continuing  this process, a repeated application of \eqref{Psi}, leads to the relation
\begin{equation*}
\label{PsiPsi}
0\leq \Lambda_k^{[m_k]}\left(\cdots \left(\Lambda_1^{[m_1]}({\bf \Delta_{\Phi}^m}(X))\right)\right)=
\lim_{q_k\to\infty}\ldots \lim_{q_1\to\infty}  (id-\phi_{k}^{q_k})\circ\cdots \circ(id-\phi_{1}^{q_1})(X)\leq X,
\end{equation*}
where ${\bf m}=(m_1,\ldots,m_k)$.  Since ${\bf \Delta_{\Phi}^m}(X)\geq 0$, we can  easily see that
$$
\sum \phi_k^{s_k^{(1)}+\cdots +s_k^{(m_k)}}\circ\cdots \circ\phi_1^{s_1^{(1)}+\cdots +s_1^{(m_1)}} ({\bf\Delta_\Phi^m}(X))
=\Lambda_k^{[m_k]}\left(\cdots \left(\Lambda_1^{[m_1]}({\bf \Delta_{\Phi}^m}(X))\right)\right),
$$
where
the summation is taken over all $(s_1^{(1)},\ldots, s_1^{(m_1)},\ldots, s_k^{(1)},\ldots, s_k^{(m_k)})\in \ZZ_+^{m_1+\cdots +m_k}$ and the convergence is in the weak operator topology. As in the proof of Theorem \ref{reproducing}, one can show that the left-hand side of the equality above coincides with
$$
\sum_{(s_1,\ldots,s_k)\in \ZZ_+^k}\left(\begin{matrix} s_1+m_1-1\\m_1-1\end{matrix}\right)\cdots \left(\begin{matrix} s_k+m_k-1\\m_k-1\end{matrix}\right)\phi_1^{s_1}\circ \cdots \circ \phi_k^{s_k}({\bf\Delta_\phi^m}(X)).
$$
Now, one can easily complete  the proof.
\end{proof}

 We recall that  $Y\in \cC_{\geq}({\bf \Delta}_{\Phi}^{\bf m})^{+}$ is called pure if, for each $i\in \{1,\ldots, k\}$, $\varphi_i^s(Y)\to 0$ weakly as $s\to \infty$.

\begin{proposition}\label{pure2} Let ${\bf m}=(m_1,\ldots, m_k)\in \NN^k$ and let  $\Phi=(\phi_1,\ldots, \phi_k)$ be a $k$-tuple of commuting  positive linear maps on $B(\cH)$ such that each $\phi_i$ is weakly continuous on bounded sets, and  let $Y\in B(\cH)$ be a positive operator.
\begin{enumerate}
  \item[(i)] If
 $Y\in \cC_{\geq}({\bf \Delta}_{\Phi}^{\bf m})^+$, then
       \begin{equation*}
0\leq
\lim_{q_k\to\infty}\ldots \lim_{q_1\to\infty}  (id-\phi_{k}^{q_k})\circ\cdots \circ(id-\phi_{1}^{q_1})(Y)\leq Y.
\end{equation*}
\item[(ii)] The operator $Y\in \cC_{\geq}({\bf \Delta}_{\Phi}^{\bf m})^+$ is  pure if and only if
\begin{equation*}
\lim_{{\bf q}=(q_1,\ldots, q_k)\in \ZZ_+^k}(id-\phi_{k}^{q_k})\circ\cdots \circ(id-\phi_{1}^{q_1})(Y)=Y.
\end{equation*}
\end{enumerate}
\end{proposition}
\begin{proof}
 Since  $(id-\phi_{k})\circ\cdots \circ(id-\phi_{1})(Y)\geq 0$ and  taking into account that $\phi_{1}, \ldots, \phi_{k}$ are commuting, we have
 $$
 0\leq (id-\phi_{k}^{q_k})\circ\cdots \circ(id-\phi_{1}^{q_1})(Y)
 =\sum_{s_k=0}^{q_k-1}\phi_{k}^{s_k}\circ\cdots \circ\sum_{s_1=0}^{q_1-1}\phi_{1}^{s_1}\circ
(id-\phi_{k})\circ\cdots \circ(id-\phi_{1})(Y).
$$
Therefore, $\{(id-\phi_{k}^{q_k})\cdots (id-\phi_{1}^{q_1})(Y)\}_{{\bf q}=(q_1,\ldots, q_k)\in \ZZ_+^k}$ is an  increasing sequence of positive operators.
Note also that
$$
 0\leq (id-\phi_{k}^{q_k})\circ\cdots \circ(id-\phi_{1}^{q_1})(Y)\leq  (id-\phi_{{k-1}}^{q_{k-1}})\circ\cdots \circ(id-\phi_{1}^{q_1})(Y)\leq\cdots \leq   (id-\phi_{1}^{q_1})(Y)\leq Y
$$
and, similarly, we have
$$
0\leq (id-\phi_{k}^{q_k})\circ\cdots \circ(id-\phi_{1}^{q_1})(Y)\leq (id-\phi_{i}^{q_i})(Y)\leq Y
$$
for each $i\in \{1,\ldots, k\}$.
Hence, we can deduce that  an  operator $Y\in \cC_{\geq}({\bf \Delta}_{\Phi}^{\bf m})^+$ satisfies relation in (ii) if and only if
$\phi_{i}^s(Y)\to 0$ weakly as $s\to\infty$    for each $i\in \{1,\ldots,k\}$.
\end{proof}

\begin{proposition}\label{exemp} Let
$\Phi=(\varphi_1,\ldots, \varphi_k)$ be a $k$-tuple of commuting, power bounded,   positive linear maps on $B(\cH)$ and let
${\bf m}=(m_1,\ldots, m_k)\in \NN^k$.
If $Y\in B(\cH)$ is a positive operator such that $\phi_{i}(Y)\leq Y$ for $i\in \{1,\ldots, k\}$, then the following statements hold.
\begin{enumerate}
  \item[(i)] If  $
m_1 \phi_{1}(Y)+\cdots + m_k\phi_{k}(Y)\leq Y,
$
then  $Y\in \cC_{\geq}({\bf \Delta}_{\Phi}^{\bf m})^+$.
\item[(ii)]
 If  $Y\in \cC_{\geq}({\bf \Delta}_{\Phi}^{\bf m})^+$, then   ${\bf \Delta}_{\Phi}^{\bf m}(Y)= 0$ if and only if
$$(id-\phi_{1})\circ\cdots \circ(id-\phi_{k})(Y) =0.
$$
\end{enumerate}
\end{proposition}

\begin{proof}
  To prove part (i), we recall (see \cite{Po-Berezin-poly}) that if
$\Phi=(\varphi_1,\ldots, \varphi_k)$ is a $k$-tuple of commuting,   power bounded, positive linear maps on $B(\cH)$,
   $Y\in B(\cH)$ is self-adjoint,  and ${\bf m}:=(m_1,\ldots, m_k)\in \NN^k$, then
  $$Y\in \cC_{\geq}({\bf \Delta}_{\Phi}^{\bf m})^{sa}\quad \text{ \rm
   if and only if }\quad
  (id-\varphi_1)^{\epsilon_1 m_1}\circ\cdots \circ(id-\varphi_k)^{\epsilon_k m_k}(Y)\geq 0 $$
   for all $\epsilon_i\in \{0,1\}$ with $(\epsilon_1,\ldots, \epsilon_k)\neq 0$.

  Let $p:=m_1+\cdots +m_k$ and   set $i_j:=1$ if $1\leq j\leq m_1$, $i_j:=2$ if $m_1+1\leq j\leq m_1+m_2,\ldots$, and $i_j:=k$ if $m_1+\cdots +m_{k-1}+1\leq j\leq m_1+\cdots +m_k$. Due to the remarks above, to prove (ii) is equivalent to showing that if $\sum_{j=1}^p \phi_{{i_j}}(Y)\leq Y$, then
$$
(id-\phi_{{i_1}})\cdots (id-\phi_{{i_p}})(Y)\geq 0.
$$
Set $Y_{i_0}:=Y$ and $Y_{i_j}:=(id -\phi_{{i_j}})(Y_{i_{j-1}})$ if $j\in \{1,\ldots, p\}$. We proceed inductively.
Note that
$Y=Y_{i_0}\geq Y_{i_1}=(id -\phi_{{i_1}})(Y)\geq 0$.   Let $n<p$ and assume that
$$
Y\geq Y_{i_n}\geq (id-\phi_{i_1}-\cdots -\phi_{{i_n}})(Y)\geq 0.
$$
Hence, we deduce that
\begin{equation*}
\begin{split}
Y&\geq Y_{i_n}\geq Y_{i_{n+1}}= Y_{i_n}-\phi_{{i_{n+1}}}(Y_{i_n})\\
&\geq (id-\phi_{{i_1}}-\cdots -\phi_{{i_n}})(Y)
-\phi_{{i_{n+1}}}(Y),
\end{split}
\end{equation*}
which proves item (i).
Now, we prove part (ii).
If  $Y\in \cC_{\geq}({\bf \Delta}_{\Phi}^{\bf m})^+$, then
   $$
   (id-\phi_{1})^{p_1}\circ\cdots \circ(id-\phi_{k})^{p_k}(Y)\geq 0
    $$
for any $p_i\in \{0,1,\ldots, m_i\}$ and
       $i\in \{1,\ldots,k\}$. Due to   Lemma 6.2 from  \cite{Po-Berezin}, if $\varphi:B(\cH)\to B(\cH)$ is a power bounded positive linear map such that $D\in B(\cH)$ is a positive operator with $(id-\varphi)(D)\geq 0$, and $\gamma\geq 1$,  then
$$
(id-\varphi)^\gamma(D)=0 \quad \text{if and only if} \quad (id-\varphi)(D)=0.
$$
Applying this result  in our setting when
$\varphi=\phi_{1}$, $\gamma=m_1$, and $D= (id-\phi_{2})^{m_2}\circ\cdots \circ(id-\phi_{k})^{m_k}(Y)\geq 0$, we deduce that   relation ${\bf \Delta}_{\Phi}^{\bf m}(Y)= 0$ is equivalent to
$(id-\phi_{1})(D)=0$. Due to the commutativity of $\phi_{1},\cdots, \phi_{k}$, the latter equality is equivalent to
$(id-\phi_{2})^{m_2}(\Lambda)=0$, where
 $$\Lambda:=(id-\phi_{3})^{m_3}\circ\cdots \circ (id-\phi_{k})^{m_k}\circ(id-\phi_{1})(Y)\geq 0.
 $$
 Applying again the result mentioned above,  we deduce that the latter equality is equivalent to $(id-\phi_{2})(\Lambda)=0$. Continuing this process, we can complete the proof of part (ii).
\end{proof}

\bigskip

\section{ Generalized Noncommutative  Berezin transforms  }

In this section, we introduce  a class of generalized (constrained) noncommutative  Berezin kernels ${\bf K}_\omega$ associated with certain compatible tuples $\omega:=({\bf f,m, A},R,\cQ)$.
These kernels will play an important role in proving most of the similarity results.
We also prove  that
  the elements of the noncommutative cones  $C_{\geq}({\bf \Delta_{{\bf \Phi}_{\bf A}}^m})^+$ and   $C_{\geq}^{pure}({\bf \Delta_{{\bf \Phi}_{\bf A}}^m})^+$
are in one-to-one correspondence
 with the elements of certain  classes
of extended noncommutative Berezin transforms.

Let ${\bf n}:=(n_1,\ldots, n_k)$, where $n_i\in\NN$.
For each $i\in \{1,\ldots, k\}$,
let $\FF_{n_i}^+$ be the free monoid on $n_i$ generators
$g_{1}^i,\ldots, g_{n_i}^i$ and the identity $g_{0}^i$.  The length of $\alpha\in
\FF_{n_i}^+$ is defined by $|\alpha|:=0$ if $\alpha=g_0^i$  and
$|\alpha|:=p$ if
 $\alpha=g_{j_1}^i\cdots g_{j_p}^i$, where $j_1,\ldots, j_p\in \{1,\ldots, n_i\}$.
If $Z_{i,1},\ldots,Z_{i,n_i}$  are  noncommuting indeterminates,   we
denote $Z_{i,\alpha}:= Z_{i,j_1}\cdots Z_{i,j_p}$  and $Z_{i,g_0^i}:=1$.
 Let  $f_i:= \sum_{\alpha\in
\FF_{n_i}^+} a_{i,\alpha} Z_{i,\alpha}$, \ $a_{i,\alpha}\in \CC$,  be a formal power series in $n_i$ noncommuting indeterminates $Z_{i,1},\ldots,Z_{i,n_i}$. We say that $f_i$ is
a {\it
positive regular free holomorphic function} if the following conditions hold:
$a_{i,\alpha}\geq 0$ for
any $\alpha\in \FF_{n_i}^+$, \ $a_{i,g_{0}^i}=0$,
   $a_{i,g_{j}^i}>0$ for  $j=1,\ldots, n_i$, and
 $$
\limsup_{p\to\infty} \left( \sum_{\alpha\in \FF_{n_i}^+,|\alpha|=p}
|a_{i,\alpha}|^2\right)^{1/2p}<\infty.
 $$
 For each $i\in \{1,\ldots,k\}$,  let $f_i:=\sum\limits_{\alpha\in \FF_{n_i}^+, |\alpha|\geq 1} a_{i,\alpha} Z_{i,\alpha}$ be
a positive regular free holomorphic function in $n_i$ variables
 and let $A_i:=(A_{i,1},\ldots, A_{i,n_i})\in B(\cH)^{n_i}$ be an $n_i$-tuple of
 operators such that
$\sum_{|\alpha|\geq 1} a_{i,\alpha} A_{i,\alpha} A_{i,\alpha}^* $ is convergent
in the weak operator topology.
 One can easily prove that the map
$\Phi_{f_i,A_i}:B(\cH)\to B(\cH)$, defined by
$$\Phi_{f_i,A_i}(X)=\sum_{|\alpha|\geq 1} a_{i,\alpha} A_{i,\alpha}
XA_{i,\alpha}^*,\qquad X\in B(\cH),
$$
where the convergence is in the weak operator topology, is a
 completely positive linear map which is WOT-continuous on bounded
 sets. Moreover, if $0<r<1$, then
$$\Phi_{f_i,A_i}(X)=\text{\rm WOT-}\lim_{r\to1} \Phi_{f_i,rA_i}(X), \qquad
X\in B(\cH).
$$
We denote by $B(\cH)^{n_1}\times_c\cdots \times_c B(\cH)^{n_k}$
   the set of all tuples  ${\bf X}=({ X}_1,\ldots, { X}_k)\in B(\cH)^{n_1}\times\cdots \times B(\cH)^{n_k}$, where ${ X}_i:=(X_{i,1},\ldots, X_{i,n_i})\in B(\cH)^{n_i}$, $i\in \{1,\ldots, k\}$,
     with the property that, for any $p,q\in \{1,\ldots, k\}$, $p\neq q$, the entries of ${ X}_p$ are commuting with the entries of ${ X}_q$. In this case we say that ${ X}_p$ and ${ X}_q$ are commuting tuples of operators.

 Let
${\bf A}:=({ A}_1,\ldots, {A}_k)\in  B(\cH)^{n_1}\times_c\cdots \times_c B(\cH)^{n_k}$, where ${ A}_i:=(A_{i,1},\ldots, A_{i,n_i})\in B(\cH)^{n_i} $ for all   $i=1,\ldots, k$,  be such that $\Phi_{f_i, A_i}(I)$ is well-defined in the weak operator topology.  If ${\bf p}:=(p_1,\ldots, p_k)\in \ZZ_+^k$ and ${\bf f}:=(f_1,\ldots,f_k)$, we denote  $\Phi_{\bf A}:=(\Phi_{f_1,A_1},\ldots, \Phi_{f_k,A_k})$ and    define the {\it defect   mapping}
${\bf \Delta_{f,A}^p}: B(\cH)\to  B(\cH)$ by setting
$$
{\bf \Delta_{f,A}^p} ={\bf \Delta}_{\bf \Phi_A}^{\bf p}:=(id-\Phi_{f_1,A_1})^{p_1}\circ\cdots \circ(id-\Phi_{f_k,A_k})^{p_k}.
$$

Let ${\bf n}:=(n_1,\ldots, n_k)$ and ${\bf m}:=(m_1,\ldots, m_k)$, where $n_i,m_i\in\NN$  and let ${\bf f}:=(f_1,\ldots,f_k)$ be a $k$-tuple of positive regular free holomorphic functions.
In \cite{Po-Berezin-poly}, we developed an operator model theory for the  noncommutative polydomain
$$
{\bf D_f^m}(\cH):=\left\{ {\bf X}=(X_1,\ldots, X_k)\in B(\cH)^{n_1}\times_c\cdots \times_c B(\cH)^{n_k}: \ {\bf \Delta_{f,X}^p}(I)\geq 0 \ \text{ for }\ {\bf 0}\leq {\bf p}\leq {\bf m}\right\}.
$$
We refer to ${\bf D_f^m}:=\{{\bf D_f^m}(\cH): \ \cH \text{ is a Hilbert space}\}$ as the {\it abstract noncommutative (regular) polydomain}, and ${\bf D_f^m}(\cH)$ as its representation on the Hilbert space $\cH$.
Let $H_{n_i}$ be
an $n_i$-dimensional complex  Hilbert space with orthonormal basis
$e_1^i,\dots,e_{n_i}^i$.
  We consider the full Fock space  of $H_{n_i}$ defined by
$$F^2(H_{n_i}):=\CC1\oplus \bigoplus_{p\geq 1} H_{n_i}^{\otimes p},$$
where   $H_{n_i}^{\otimes p}$ is the
(Hilbert) tensor product of $p$ copies of $H_{n_i}$. Set $e_\alpha^i :=
e^i_{j_1}\otimes \cdots \otimes e^i_{j_p}$ if
$\alpha=g^i_{j_1}\cdots g^i_{j_p}\in \FF_{n_i}^+$
 and $e^i_{g^i_0}:= 1\in \CC$.
It is  clear that $\{e^i_\alpha:\alpha\in\FF_{n_i}^+\}$ is an orthonormal
basis of $F^2(H_{n_i})$.
 We  define
 the {\it weighted left creation  operators} $W_{i,j}:F^2(H_{n_i})\to
F^2(H_{n_i})$,   associated with the abstract noncommutative
domain  ${\bf D}_{f_i}^{m_i} $    by setting
\begin{equation}\label{Wij}
W_{i,j} e_\alpha^i:=\frac {\sqrt{b_{i,\alpha}^{(m_i)}}}{\sqrt{b_{i,g_j^i
\alpha}^{(m_i)}}} e^i_{g_j^i \alpha}, \qquad \alpha\in \FF_{n_i}^+,
\end{equation}
where
\begin{equation} \label{b-coeff}
  b_{i,g_0^i}^{(m_i)}:=1\quad \text{ and } \quad
 b_{i,\alpha}^{(m_i)}:= \sum_{p=1}^{|\alpha|}
\sum_{{\gamma_1,\ldots,\gamma_p\in \FF_{n_i}^+}\atop{{\gamma_1\cdots \gamma_p=\alpha }\atop {|\gamma_1|\geq
1,\ldots, |\gamma_p|\geq 1}}} a_{i,\gamma_1}\cdots a_{i,\gamma_p}
\left(\begin{matrix} p+m_i-1\\m_i-1
\end{matrix}\right)
\end{equation}
for all  $ \alpha\in \FF_{n_i}^+$ with  $|\alpha|\geq 1$.
For each $i\in \{1,\ldots, k\}$ and $j\in \{1,\ldots, n_i\}$, we
define the operator ${\bf W}_{i,j}$ acting on the tensor Hilbert space
$F^2(H_{n_1})\otimes\cdots\otimes F^2(H_{n_k})$ by setting
$${\bf W}_{i,j}:=\underbrace{I\otimes\cdots\otimes I}_{\text{${i-1}$
times}}\otimes W_{i,j}\otimes \underbrace{I\otimes\cdots\otimes
I}_{\text{${k-i}$ times}},
$$
where the operators $W_{i,j}$ are defined by relation \eqref{Wij}.
 If  ${\bf W}_i:=({\bf W}_{i,1},\ldots,{\bf W}_{i,n_i})$, then   ${\bf W}:=({\bf W}_1,\ldots, {\bf W}_k)$ is  a pure $k$-tuple, i.e. $\phi_{f_i,{\bf W}_i}^s(I)\to 0$ weakly as $s\to \infty$,   in the
noncommutative polydomain $ {\bf D_f^m}(\otimes_{i=1}^kF^2(H_{n_i}))$.
  The $k$-tuple ${\bf W}$
    is the {\it universal model} associated
  with the abstract noncommutative
  polydomain ${\bf D_f^m}$ (see \cite{Po-Berezin-poly}).

In what follows, we introduce a   generalized noncommutative Berezin kernel
associated with any {\it compatible quadruple} $({\bf f,m, A},R)$ satisfying the following
  conditions:
\begin{enumerate}
\item[(i)]  ${\bf f}:=(f_1,\ldots, f_k)$ is a $k$-tuple of positive regular free holomorphic functions with
  $f_i:=\sum_{\alpha_i\in \FF_{n_i}^+} a_{i,\alpha} Z_{i,\alpha}$ and ${\bf m}=(m_1,\ldots,m_k)\in \NN^k$;
\item[(ii)] ${\bf A}:=({ A}_1,\ldots, {A}_k)\in  B(\cH)^{n_1}\times_c\cdots \times_c B(\cH)^{n_k}$, where ${ A}_i:=(A_{i,1},\ldots, A_{i,n_i})\in B(\cH)^{n_i}$, has the property that $\sum_{\alpha_i\in \FF_{n_i}^+} a_{i,\alpha} A_{i,\alpha}A_{i,\alpha}^*$
     is  weakly convergent;
  \item[(iii)] $R\in
B(\cH)$ is a positive operator such that
 $$
\sum_{(s_1,\ldots,s_k)\in \ZZ_+^k}\left(\begin{matrix} s_1+m_1-1\\m_1-1\end{matrix}\right)\cdots \left(\begin{matrix} s_k+m_k-1\\m_k-1\end{matrix}\right)\Phi_{f_1,A_1}^{s_1}\circ \cdots \circ \Phi_{f_k, A_k}^{s_k}(R)\leq b I,
$$
 for some constant $b>0$, where $$\Phi_{f_i, A_i}(X):=\sum_{\alpha\in \FF_{n_i}^+} a_{i,\alpha} A_{i,\alpha}XA_{i,\alpha}^*,\qquad X\in B(\cH).$$
\end{enumerate}
The {\it generalized noncommutative Berezin kernel} associated with a compatible
quadruple $({\bf f,m, A},R)$ is the operator

   $${\bf K}_{\bf f,A}^R: \cH \to F^2(H_{n_1})\otimes \cdots \otimes  F^2(H_{n_k}) \otimes  \overline{R^{1/2} (\cH)}$$
   defined by
   $$
   {\bf K}_{\bf f,A}^Rh:=\sum_{\beta_i\in \FF_{n_i}^+, i=1,\ldots,k}
   \sqrt{b_{1,\beta_1}^{(m_1)}}\cdots \sqrt{b_{k,\beta_k}^{(m_k)}}
   e^1_{\beta_1}\otimes \cdots \otimes  e^k_{\beta_k}\otimes R^{1/2} A_{1,\beta_1}^*\cdots A_{k,\beta_k}^*h,
   $$
where   the coefficients  $b_{1,\beta_1}^{(m_1)}, \ldots, b_{k,\beta_k}^{(m_k)}$
are given by relation \eqref{b-coeff}. The fact that ${\bf K}_{\bf f,A}^R$ is a well-defined bounded operator will be proved  in the next theorem.

\begin{theorem} \label{Berezin-prop}
The generalized Berezin kernel associated with any compatible  quadruple $({\bf f,m, A},R)$
 has the following properties.
\begin{enumerate}
\item[(i)] ${\bf K}_{\bf f,A}^R$ is a bounded operator  and
$$
({\bf K}_{\bf f,A}^R)^*\,{\bf K}_{\bf f,A}^R=
\sum_{(s_1,\ldots,s_k)\in \ZZ_+^k}\left(\begin{matrix} s_1+m_1-1\\m_1-1\end{matrix}\right)\cdots \left(\begin{matrix} s_k+m_k-1\\m_k-1\end{matrix}\right)\Phi_{f_1,A_1}^{s_1}\circ \cdots \circ \Phi_{f_k, A_k}^{s_k}(R),
$$
where the convergence is  in the weak  operator topology.

\item[(ii)]  For any $i\in \{1,\ldots, k\}$ and $j\in \{1,\ldots, n_i\}$,  $$ {\bf K}_{\bf f,A}^R { A}^*_{i,j}= ({\bf W}_{i,j}^*\otimes I_\cR)  {\bf K}_{\bf f,A}^R,
    $$
    where $\cR:=\overline{R^{1/2} \cH}$ and ${\bf W}=\{{\bf W}_{i,j}\}$
    is the  universal model
     associated
  with the abstract noncommutative
  polydomain ${\bf D_f^m}$.
\end{enumerate}
\end{theorem}
\begin{proof}
Rearranging  WOT-convergent series of positive operators, we deduce that, for each $d\in \NN$,
\begin{equation*}
\begin{split}
\Phi_{f_i, A_i}^d(R)
 &=\sum_{\alpha_1\in \FF_{n_i}^+, |\alpha_1|\geq 1} a_{i,\alpha_1} A_{i,\alpha_1}\left( \cdots
\sum_{\alpha_d\in \FF_{n_i}^+, |\alpha_d|\geq 1} a_{i,\alpha_d} A_{i,\alpha_d}RA_{i,\alpha_d}^*\cdots
       \right)A_{i,\alpha_1}^*\\
     &\qquad =\sum_{\gamma\in \FF_{n_i}^+, |\gamma|\geq d}
     \sum_{{\alpha_1,\ldots,\alpha_d\in \FF_{n_i}^+}\atop{{\alpha_1\cdots \alpha_d=\gamma }\atop {|\alpha_1|\geq
1,\ldots, |\alpha_d|\geq 1}}} a_{i,\alpha_1}\cdots a_{i,\alpha_d} A_{i,\gamma}RA_{i,\gamma}^*
\end{split}
\end{equation*}
and
\begin{equation*}
\begin{split}
\Lambda_{i}^{[1]}(R)&:=\sum_{s=0}^\infty
 \Phi_{f_i,A_i}^{s}(R)
=
R+\sum_{\gamma\in \FF_{n_i}^+, |\gamma|\geq 1}
\left(\sum_{d=1}^{|\gamma|}
     \sum_{{\alpha_1,\ldots,\alpha_d\in \FF_{n_i}^+}\atop{{\alpha_1\cdots \alpha_d
     =\gamma }\atop {|\alpha_1|\geq
1,\ldots, |\alpha_d|\geq 1}}} a_{i,\alpha_1}\cdots
 a_{i,\alpha_d}\right) A_{i,\gamma}RA_{i,\gamma}^*.
\end{split}
\end{equation*}
 Since $\Lambda_{i}^{[j]}(R):=\sum_{s=0}^\infty
  \Phi_{f_i,A_i}^{s}(\Lambda_{i}^{[j-1]}(R))$ for
   $j=2,\ldots, p_i$, using a combinatorial argument and rearranging
    WOT-convergent series  of positive operators, one can prove by induction
    over $p_i$ that
\begin{equation*}
\begin{split}
\Lambda_{i}^{[p_i]}(R)&=R+
\sum_{\alpha\in \FF_{n_i}^+, |\alpha|\geq 1}\left(
 \sum_{p=1}^{|\alpha|}
\sum_{{\gamma_1,\ldots,\gamma_p\in \FF_{n_i}^+}\atop{{\gamma_1\cdots \gamma_p=\alpha }\atop {|\gamma_1|\geq
1,\ldots, |\gamma_p|\geq 1}}} a_{i,\gamma_1}\cdots a_{i,\gamma_p}
\left(\begin{matrix} p+p_i-1\\p_i-1
\end{matrix}\right)\right)A_{i,\alpha} R A_{i,\alpha}^*\\
&
=\sum_{\alpha\in \FF_{n_i}^+} b_{i,\alpha}^{(p_i)} A_{i,\alpha} R A_{i,\alpha}^*.
\end{split}
\end{equation*}
 Now, using the results above, we deduce that
\begin{equation*}
\begin{split}
\|{\bf K}_{\bf f,A}^Rh\|^2 &=
\sum_{\beta_k\in \FF_{n_k}}\cdots \sum_{\beta_1\in \FF_{n_1}}  b_{1,\beta_1}^{(m_1)}\cdots b_{k,\beta_k}^{(m_k)} \left<A_{k,\beta_k}\cdots A_{1,\beta_1}R A_{1,\beta_1}^*\cdots A_{k,\beta_k}^*h, h\right>\\
&=\left< (\Lambda_k^{[m_k]}\circ\cdots \circ\Lambda_1^{[m_1]})(R)h,h\right>\\
&=\sum_{(s_k^{(1)},\ldots,s_k^{(m_k)})\in \ZZ_+^{m_k}}\cdots \sum_{(s_1^{(1)},\ldots, s_1^{(m_1)})\in \ZZ_+^{m_1}} \left<\left[\phi_{f_k, A_k}^{s_k^{(1)}+\cdots +s_k^{(m_k)}}\circ\cdots \circ\phi_{f_1,A_1}^{s_1^{(1)}+\cdots +s_1^{(m_1)}}(R)\right]h,h\right>\\
&=\sum_{(s_1,\ldots,s_k)\in \ZZ_+^k}\left(\begin{matrix} s_1+m_1-1\\m_1-1\end{matrix}\right)\cdots \left(\begin{matrix} s_k+m_k-1\\m_k-1\end{matrix}\right)\left<\left(\Phi_{f_1,A_1}^{s_1}\circ \cdots \circ \Phi_{f_k, A_k}^{s_k}(R)\right)h,h\right>\leq b I
\end{split}
\end{equation*}
for any  $h\in \cH$.  This proves item (i).
 To prove part (ii),  note that
\begin{equation}\label{WbWb}
  W_{i,j}^*
e^i_{\beta_i} =\begin{cases} \frac
{\sqrt{b_{i,\gamma_i}^{(m_i)}}}{\sqrt{b_{i,\beta_i}^{(m_i)}}}\,e^i_{\gamma_i}& \text{ if
}
\beta_i=g_j^i\gamma_i, \ \, \gamma_i\in \FF_{n_i}^+ \\
0& \text{ otherwise }
\end{cases}
\end{equation}
 for any   $\beta_i\in \FF_{n_i}^+$ and
 $j\in \{1,\ldots, n_i\}$.
 Hence, and using the definition of the generalized noncommutative Berezin kernel, we have
 \begin{equation*}
 \begin{split}
  &({\bf W}_{i,j}^*\otimes I){\bf K}_{\bf f,A}^Rh\\
  & =\sum_{\beta_p\in \FF_{n_p}^+, p\in\{1,\ldots,k\}}
   \sqrt{b_{1,\beta_1}^{(m_1)}}\cdots \sqrt{b_{k,\beta_k}^{(m_k)}}
   e^1_{\beta_1}\otimes \cdots \otimes e^{i-1}_{\beta_{i-1}}\otimes W_{i,j}^*e^i_{\beta_i}\otimes e^{i+1}_{\beta_{i+1}}\otimes \cdots \otimes  e^k_{\beta_k}\otimes R^{1/2} A_{1,\beta_1}^*\cdots A_{k,\beta_k}^*h\\
   & =\sum_{{\beta_p\in \FF_{n_p}^+, p\in \{1,\ldots,k\}\backslash\{i\}} \atop{\gamma_i\in \FF_{n_i}}}
   \sqrt{b_{1,\beta_1}^{(m_1)}}\cdots \sqrt{b_{i,\gamma_i}^{(m_i)}}\cdots \sqrt{b_{k,\beta_k}^{(m_k)}}
   e^1_{\beta_1}\otimes \cdots \otimes  e^{i-1}_{\beta_{i-1}}\otimes e^i_{\gamma_i}\otimes e^{i+1}_{\beta_{i+1}} \otimes \cdots \otimes e^k_{\beta_k}\\
   &\qquad \qquad\qquad \otimes R^{1/2} A_{1,\beta_1}^*\cdots  A_{i-1,\beta_i}^* A_{i,g_j^i \gamma_i}^* A_{i+1, \beta_{i+1}}^*\cdots A_{k,\beta_k}^*h
 \end{split}
 \end{equation*}
for any $h\in \cH$. Using the commutativity of  the tuples $A_1,\ldots,A_k$,
we deduce that
$$
({\bf W}_{i,j}^*\otimes I){\bf K}_{\bf f,A}^R={\bf K}_{\bf f,A}^R A_{i,j}^*
$$
for any $i\in \{1,\ldots, k\}$ and $j\in \{1,\ldots, n_i\}$.
The proof is complete.
\end{proof}

For each $i\in\{1,\ldots, k\}$, let $Z_i:=(Z_{i,1},\ldots, Z_{i,n_i})$ be
an  $n_i$-tuple of noncommuting indeterminates and assume that, for any
$s,t\in \{1,\ldots, k\}$, $s\neq t$, the entries in $Z_s$ are commuting
 with the entries in $Z_t$.  The algebra of all polynomials  in indeterminates $Z_{i,j}$ is denoted by $\CC\left<Z_{i,j}\right>$.
  If $\cQ$ is a left ideal of     polynomials  in $\CC\left<Z_{i,j}\right>$, we define  the noncommutative variety
$$
\cV_\cQ(\cH):=\{{\bf X}=\{X_{i,j}\}\in {\bf D_f^m}(\cH):\ g({\bf X})=0 \text{ for all } g\in \cQ\}.
$$
Consider  the subspace
 \begin{equation*}
 \cM_\cQ:=\overline{\text{\rm span}}\{{\bf W}_{(\alpha)} q({\bf W}_{i,j}) {\bf W}_{(\beta)}(\CC): \
 (\alpha), (\beta)\in \FF_{n_1}^+\times \cdots\times \FF_{n_k}^+, q\in \cQ\},
 \end{equation*}
 where ${\bf W}_{(\alpha)}:={\bf W}_{1,\alpha_1}\cdots {\bf W}_{k,\alpha_k}$  if $(\alpha)=(\alpha_1,\ldots, \alpha_k)$,
 and let
 $$\cN_\cQ:=[\bigotimes_{i=1}^kF^2(H_{n_i})]\ominus \cM_\cQ.
  $$
  Throughout this paper,  we assume that $\cN_\cQ\neq
 \{0\}$. It is easy to see that $\cN_\cQ$ is invariant under each
 operator ${\bf W}_{i,j}^* $  for $i\in \{1,\ldots, k\}$ and  $j\in \{1,\ldots, n_i\}$.  Define ${\bf S}_{i,j}:=P_{\cN_\cQ}
{\bf W}_{i,j}|_{\cN_\cQ}$, where
 $P_{\cN_Q}$ is the orthogonal projection of $\bigotimes_{i=1}^k F^2(H_{n_i})$ onto $\cN_Q$.
 The $k$-tuple ${\bf S}:=({\bf S}_1,\ldots, {\bf S}_k)$, where ${\bf S}_i:=({\bf S}_{i,1}\ldots, {\bf S}_{i,n_i})$, is in the noncommutative variety
 $ {\bf \cV_\cQ}(\cN_\cQ)$ and  plays  the role of {\it universal model}
   for the {\it abstract  noncommutative variety}
   $${\bf \cV_\cQ}:=\{{\bf \cV_\cQ}(\cH):\ \cH \text { is a Hilbert space}\}.
   $$

Let $({\bf f,m, A},R)$ be a compatible quadruple.    In addition, we assume that the $k$-tuple ${\bf A}:=({ A}_1,\ldots, {A}_k)\in  B(\cH)^{n_1}\times_c\cdots \times_c B(\cH)^{n_k}$, where ${ A}_i:=(A_{i,1},\ldots, A_{i,n_i})\in B(\cH)^{n_i}$, has the
property that
$$
q({\bf A})=0,\qquad q\in \cQ.
$$
 Under these conditions,  the tuple  $\omega:=({\bf f,m, A},R,\cQ)$ is called compatible.
  We define the {\it (constrained) noncommutative  Berezin kernel} associated with the  tuple
  $\omega$ to be
the operator ${\bf K}_\omega:\cH\to \cN_\cQ\otimes \overline{ R^{1/2}\cH}$
given by
$${\bf K}_\omega:=(P_{\cN_{\cQ}}\otimes I_{\overline{ R^{1/2}\cH}}){\bf K}_{\bf f,A}^R,
$$
where ${\bf K}_{\bf f,A}^R$ is the generalized Berezin kernel associated with the
quadruple $({\bf f,m, A},R)$.

\begin{theorem}\label{lemma2} Let   ${\bf K}_\omega$ be  the constrained noncommutative
Berezin kernel  associated with a compatible tuple
$\omega:=({\bf f,m, A},R,\cQ)$.  Then
\begin{equation*}
    {\bf K}_\omega { A}^*_{i,j}= ({\bf S}_{i,j}^*\otimes I_\cR)  {\bf K}_\omega,\qquad i\in \{1,\ldots, k\}, j\in \{1,\ldots, n_i\},
\end{equation*}
where  $\cR:=\overline{ R^{1/2}(\cH)}$ and ${\bf S}=\{{\bf S}_{i,j}\}$
    is the  universal model
     associated
  with the abstract noncommutative
  variety $\cV_\cQ$. Moreover,
$$
{\bf K}_\omega^*\,{\bf K}_\omega=
\sum_{(s_1,\ldots,s_k)\in \ZZ_+^k}\left(\begin{matrix} s_1+m_1-1\\m_1-1\end{matrix}\right)\cdots \left(\begin{matrix} s_k+m_k-1\\m_k-1\end{matrix}\right)\Phi_{f_1,A_1}^{s_1}\circ \cdots \circ \Phi_{f_k, A_k}^{s_k}(R),
$$
where the convergence is  in the weak  operator topology.
\end{theorem}

\begin{proof}
Since     ${\bf K}_{\bf f,A}^R { A}^*_{i,j}= ({\bf W}_{i,j}^*\otimes I)  {\bf K}_{\bf f,A}^R$
for any $i\in \{1,\ldots, k\}$ and $j\in \{1,\ldots, n_i\}$, we deduce that
\begin{equation*}
\left< {\bf K}_{\bf f,A}^R  x,  q({\bf W}_{i,j}){\bf W}_{(\alpha)}
(1)\otimes y\right>=\left<x, q({\bf A})A_{(\alpha)}
({\bf K}_{\bf f,A}^R)^*(1\otimes y)\right>= 0
\end{equation*}
for any $x\in \cH$, $y\in \overline{R^{1/2}\cH}$, $(\alpha) \in \FF_{n_1}^+\otimes \cdots \otimes \FF_{n_k}^+$, and any polynomial $q \in \cQ$.  Consequently,
\begin{equation*}
\text{\rm range}\,{\bf K}_{\bf f,A}^R \subseteq \cN_Q\otimes
\overline{ R^{1/2}\cH}.
\end{equation*}
Taking into account the definition of  the  constrained  Berezin
kernel ${\bf K}_\omega:\cH\to \cN_\cQ\otimes \overline{ R^{1/2}\cH}$,  one can
use Theorem \ref{Berezin-prop}   to complete the
proof.
\end{proof}

\begin{remark} \label{more} If $n_i\in \NN\cup \{\infty\}$ for $i\in \{1,\ldots, k\}$, all the results of this section remain true under the additional assumption that
$$ \sum_{\alpha\in \FF_{n_i}^+,|\alpha|=p}
|a_{i,\alpha}|^2<\infty,\qquad \text{ if } \ n_i=\infty,$$
for any $p\in \NN$.
\end{remark}

\begin{lemma}
\label{radial}
Let
${\bf A}=({ A}_1,\ldots, { A}_k)\in B(\cH)^{n_1}\times_c\cdots \times_c B(\cH)^{n_k}$ be such that $\Phi_{f_i, A_i}$ is power bounded  for any $i\in \{1,\ldots, k\}$,  and let $Y\in B(\cH)$ be a positive operator such that  $\Phi_{f_i, A_i}(Y)\leq Y$. If  ${\bf m}\in \ZZ_+^k$  and ${\bf m}\neq 0$, then the following statements are equivalent:
\begin{enumerate}
\item[(i)] $Y\in \cC_{\geq}({\bf \Delta_{f,A}^m})^+$;
    \item[(ii)]    $ {\bf \Delta}_{{\bf f},r{\bf A}}^{\bf m}(Y)\geq 0$ for any $r\in [0, 1]$;
\item[(iii)] there exists $\delta\in (0,1)$ such that $ {\bf \Delta}_{{\bf f},r{\bf A}}^{\bf m}(Y)\geq 0$ for any $r\in (\delta, 1)$;
\end{enumerate}
\end{lemma}
\begin{proof}
  To prove that (i) implies (ii), we apply Proposition \ref{Delta-ineq} when $\varphi_i=\Phi_{f_i,A_i}$ and $\psi_i=\Phi_{f_i,rA_i}$.
  Since the implication (ii)$\implies$(iii) is obvious, it remains to prove that (iii)$\implies$(i).

Assume that there exists $\delta\in (0,1)$ such that $ {\bf \Delta}_{{\bf f},r{\bf A}}^{\bf m}(Y)\geq 0$ for any $r\in (\delta, 1)$.
Since $\Phi_{f_i, rA_i}$ is power bounded and $\Phi^s_{f_i, rA_i}(I)\leq r^s\Phi^s_{f_i, A_i}(I)$, it is clear that $\Phi_{f_i, rA_i}^s(I)\to 0$ weakly as $s\to\infty$  for each $i\in\{1,\ldots, k\}$. Applying Proposition \ref{Delta-ineq} part (ii), we deduce that
$ {\bf \Delta}_{{\bf f},r{\bf A}}^{\bf p}(Y)\geq 0$ for any  $r\in (\delta, 1)$ and any ${\bf p}\in \ZZ_+^k$ with ${\bf p}\leq {\bf m}$.
 Note that ${\bf \Delta}_{{\bf f},r{\bf A}}^{\bf p}(Y)$ is a linear combination  of products of the form
 $\Phi_{f_1, rA_1}^{q_1}\circ\cdots \circ\Phi_{f_k, rA_k}^{q_k}(Y)$, where $(q_1,\ldots, q_k)\in \ZZ_+^k$, and
 $$
 \Phi_{f_1, A_1}^{q_1}\circ\cdots \circ\Phi_{f_k, A_k}^{q_k}(Y)=\text{\rm WOT-}\lim_{j\to \infty} \sum_{{\alpha_i\in \FF_{n_i}^+}\atop{|\alpha_1|+\cdots +|\alpha_k|\leq j}} c_{\alpha_1,\ldots,\alpha_k} A_{1,\alpha_1}\cdots A_{k,\alpha_k} Y A_{k,\alpha_k}^*\cdots A_{1,\alpha_1}^*\leq Y
 $$
 for some positive constants $c_{\alpha_1,\ldots,\alpha_k}$. If $x\in \cH$ and $\epsilon >0$, then there is $N_0\in \NN$ such that
 $$\sum_{{\alpha_i\in \FF_{n_i}^+}\atop {|\alpha_1|+\cdots +|\alpha_k|\geq j}} c_{\alpha_1,\ldots,\alpha_k} r^{2(|\alpha_1|+\cdots +|\alpha_k|)} \left<A_{1,\alpha_1}\cdots A_{k,\alpha_k}Y A_{k,\alpha_k}^*\cdots A_{1,\alpha_1}^* x,x\right><\epsilon$$ for any $j\geq N_0$
 and $r\in (\delta, 1)$.  One can use this fact  to show that
 $$
 \Phi_{f_1, A_1}^{q_1}\circ\cdots \circ\Phi_{f_k, A_k}^{q_k}(Y)=\text{\rm WOT-}\lim_{r\to 1}
 \Phi_{f_1, rA_1}^{q_1}\circ\cdots \circ\Phi_{f_k, rA_k}^{q_k}(Y).
 $$
 Hence, we deduce that
 $
  {\bf \Delta}_{{\bf f},{\bf A}}^{\bf p}(Y)=\text{\rm WOT-}\lim_{r\to 1}
  {\bf \Delta}_{{\bf f},r{\bf A}}^{\bf p}(Y)\geq 0
  $ \
 for any ${\bf p}\in \ZZ_+^k$ with ${\bf p}\leq {\bf m}$.    This  completes the proof.
 \end{proof}

We introduce now  the  {\it constrained noncommutative Berezin transform} ${\bf
B}_\omega$
 associated with a compatible  tuple $\omega:=({\bf f,m, A},R,\cQ)$ to be the operator
 ${\bf B}_\omega: B(\cN_\cQ)\to B(\cH)$ given by
 $$
 {\bf B}_\omega[\chi]:={\bf K}_\omega^*[\chi\otimes I_\cR] {\bf K}_\omega,\qquad \chi\in B(\cN_\cQ).
 $$
 where ${\bf K}_\omega$ is the  constrained noncommutative Berezin kernel introduced in Section 2, and  $\cR:=\overline{ R^{1/2}(\cH)}$.
 Now, we are ready to   show that
  the elements of the noncommutative cone  $ C_{\geq}({\bf \Delta_{f,A}^m})^+$
are in one-to-one correspondence
 with the elements of a class
of extended noncommutative Berezin transforms.
\begin{theorem}\label{poisson} Let $\cV_\cQ \subset {\bf D_f^m}$ be an abstract noncommutative variety, where  $\cQ$ is a family of
noncommutative homogeneous
 polynomials in indeterminates $\{Z_{i,j}\}$,  and let ${\bf S}=\{{\bf S}_{i,j}\}$
    be its  universal model.
      If
 ${\bf A}:=({ A}_1,\ldots, {A}_k)\in  B(\cH)^{n_1}\times_c\cdots \times_c B(\cH)^{n_k}$, where ${ A}_i:=(A_{i,1},\ldots, A_{i,n_i})\in B(\cH)^{n_i}$ has the property that $ \Phi_{f_i, A_i}$ is well-defined and $q({\bf A})=0$ for any $ q\in \cQ$,
then there is a bijection
$$\Gamma: CP(A,\cV_\cQ)\to C_{\geq}({\bf \Delta_{f,A}^m})^+, \qquad
 \Gamma(\varphi):= \varphi (I),
 $$
  where  $  CP(A,\cV_\cQ)$ is   the set
  of all completely positive linear maps \ $\varphi : \cS
 \to B(\cH)$ such that
$$ \varphi( {\bf S}_{(\alpha)} {\bf S}_{(\beta)}^*) =A_{(\alpha)} \varphi(I) A_{(\beta)}^*,
 \qquad (\alpha), (\beta)\in \FF_{n_1}^+\times \cdots\times \FF_{n_k}^+,
 $$
 where $\cS:=
\overline{\text{\rm span}} \{ {\bf S}_{(\alpha)} {\bf S}_{(\beta)}^*:\ (\alpha), (\beta)\in \FF_{n_1}^+\times \cdots\times \FF_{n_k}^+\}$. Moreover,  if $D\in C_{\geq}({\bf \Delta_{f,A}^m})^+$, then
$\Gamma^{-1}(D)$ coincides
 with the extended noncommutative
Berezin  transform  associated with $\omega:=({\bf f,m, A},R,\cQ)$ which is defined by
$$
  \overline{{\bf B}}_{\omega}[\chi]:=\lim_{r \to 1}
 {\bf K}_{\omega_r }^* (\chi\otimes I) {\bf K}_{\omega_r} ,\qquad \chi\in
   \cS,
$$
where  $\omega_r:=({\bf f,m, {\it r}A},R_r,\cQ)$ and $R_r:={\bf \Delta_{f,{\it r}A}^m}(D)$,
$r\in [0,1]$, and the limit exists in the operator norm topology.
\end{theorem}

\begin{proof}
 Assume that  $\varphi : \cS
 \to B(\cH)$ is a completely positive map such that
$$ \varphi( {\bf S}_{(\alpha)} {\bf S}_{(\beta)}^*) =A_{(\alpha)} \varphi(I) A_{(\beta)}^*,
 \qquad (\alpha), (\beta)\in \FF_{n_1}^+\times \cdots\times \FF_{n_k}^+.
 $$
 Let ${\bf W}:=\{{\bf W}_{i,j}\}$ be the universal model associated with the abstract noncommutative polydomain ${\bf D_f^m}$. Since
 ${\bf \Delta_{f,W}^p}(I)\geq 0$ for any ${\bf p}\in \ZZ_+^k$ with ${\bf p}\leq {\bf m}$, and $\cN_Q$ is invariant under each operator ${\bf W}_{i,j}^*$, we also have ${\bf \Delta_{f,S}^p}(I)\geq 0$.
 Due to Lemma \ref{radial}, we deduce that
  ${\bf \Delta_{f,{\it r}S}^p}(I)\geq 0$ for any $r\in [0,1)$ and ${\bf p}\in \ZZ_+^k$ with ${\bf p}\leq {\bf m}$.
 Let $f_i:=\sum_{\alpha_i\in \FF_{n_i}^+} a_{i,\alpha} Z_{i,\alpha}$ and note that $\Phi_{f_i, r{\bf S}_i}(I)=\sum_{k=1}^\infty
\sum_{\alpha_i\in \FF_{n_i}^+, |\alpha_i|=k} a_{i,\alpha_i} r^{|\alpha_i|} {\bf S}_{i,\alpha_i} {\bf S}_{i,\alpha_i} ^*\leq
 I$, where the convergence is in the operator norm topology. Consequently,
$\Phi_{f_i, r{\bf S}_i}(I)\in \cS$ and ${\bf \Delta_{f,{\it r}S}^p}(I)\in \cS$ for any  ${\bf p}\in \ZZ_+^k$ with ${\bf p}\leq {\bf m}$.

Setting $D:=\varphi(I)$, we deduce that $D\geq 0$ and
$$
{\bf \Delta_{f,{\it r}A}^p}(D)=\varphi \left(
{\bf \Delta_{f,{\it r}S}^p}(I)\right)\geq 0,\qquad r\in [0,1),
$$
for any  ${\bf p}\in \ZZ_+^k$ with ${\bf p}\leq {\bf m}$. Since  the series
$\sum\limits_{\alpha\in \FF_{n_i}^+, |\alpha_i|\geq 1} a_{i,\alpha_i} A_{i,\alpha_i} A_{i,\alpha_i}^* $ \ is
weakly convergent, we  deduce  that
$\Phi^s_{f_i,A_i}(D)=\text{\rm WOT-}\lim_{r\to1}\Phi_{f_i,rA_i}^s(D)$ for
$s\in\NN$ and, moreover,
$${\bf \Delta_{f,A}^p}(D)=\text{\rm
WOT-}\lim_{r\to1}{\bf \Delta_{f,{\it r}A}^p}(D)\geq 0
$$
for any  ${\bf p}\in \ZZ_+^k$ with ${\bf p}\leq {\bf m}$. This shows that $D\in  C_{\geq}({\bf \Delta_{f,A}^m})^+$.

To
prove that $\Gamma$ is one-to-one, let $\varphi_1$ and $\varphi_2$
be completely positive linear maps on the operator system $\cS$ such that
$ \varphi_j( {\bf S}_{(\alpha)} {\bf S}_{(\beta)}^*) =A_{(\alpha)} \varphi_j(I) A_{(\beta)}^*$ for any
 $(\alpha), (\beta)\in \FF_{n_1}^+\times \cdots\times \FF_{n_k}^+$ and $j=1,2$.  Assume that
$\Gamma(\varphi_1)=\Gamma(\varphi_2)$, i.e.,
$\varphi_1(I)=\varphi_2(I)$. Then we have $\varphi_1( {\bf S}_{(\alpha)} {\bf S}_{(\beta)}^*)=\varphi_2( {\bf S}_{(\alpha)} {\bf S}_{(\beta)}^*)$ for $(\alpha), (\beta)\in \FF_{n_1}^+\times \cdots\times \FF_{n_k}^+$. Taking into account  the continuity of $\varphi_1$ and
$\varphi_2$ in the operator norm, we deduce that
$\varphi_1=\varphi_2$.

 To prove surjectivity of the map $\Gamma$, fix
$D\in C_{\geq}({\bf \Delta_{f,A}^m})^+$. Then
 $D\in  B(\cH)$  is a positive operator with the property
  that
 ${\bf \Delta_{f,{\it r}A}^p}(D)\geq 0$ for any  ${\bf p}\in \ZZ_+^k$ with ${\bf p}\leq {\bf m}$ and $r\in [0,1)$. Since the set $\cQ$ consists of homogeneous noncommutative polynomials in indeterminates $Z_{i,j}$ , we have
 $q(\{rA_{i,j}\})=0$ for any $q\in \cQ$ and $r\in [0, 1)$. We show now that, for each $r\in [0, 1)$,
 the tuple $\omega_r:=({\bf f,m, {\it r}A},R_r,\cQ)$, where $R_r:={\bf \Delta_{f,{\it r}A}^m}(D)$,  is compatible. Indeed, we can use
  Theorem \ref{reproducing} to  obtain
\begin{equation}
\label{Kr*}
D=\sum_{(s_1,\ldots,s_k)\in \ZZ_+^k}\left(\begin{matrix} s_1+m_1-1\\m_1-1\end{matrix}\right)\cdots \left(\begin{matrix} s_k+m_k-1\\m_k-1\end{matrix}\right)\phi_{f_1, rA_1}^{s_1}\circ \cdots \circ \phi_{f_k, rA_k}^{s_k}( R_r), \qquad r\in [0, 1).
\end{equation}
where
 the convergence of the series is in the weak operator topology.
According  to Theorem \ref{lemma2},  the constrained noncommutative Berezin kernel   ${\bf K}_{\omega_r}$,
$r\in [0,1)$,
associated with the compatible tuple $\omega_r:=({\bf f,m, {\it r}A},R_r,\cQ)$,  has the property that

\begin{equation}\label{KK}
    {\bf K}_{\omega_r} (r{A}^*_{i,j})= ({\bf S}_{i,j}^*\otimes I_\cH)  {\bf K}_{\omega_r},
\end{equation}
where    ${\bf S}=\{{\bf S}_{i,j}\}$
    is the  universal model
     associated
  with the abstract noncommutative
  variety $\cV_\cQ$. Moreover,
$$
{\bf K}_{\omega_r}^*\,{\bf K}_{\omega_r}=
\sum_{(s_1,\ldots,s_k)\in \ZZ_+^k}\left(\begin{matrix} s_1+m_1-1\\m_1-1\end{matrix}\right)\cdots \left(\begin{matrix} s_k+m_k-1\\m_k-1\end{matrix}\right)\Phi_{f_1,rA_1}^{s_1}\circ \cdots \circ \Phi_{f_k, rA_k}^{s_k}(R_r),
$$
where the convergence is  in the weak  operator topology.
Hence, and using relation \eqref{Kr*},  we obtain
\begin{equation}\label{K*K}
{\bf K}_{\omega_r}^*{\bf  K}_{\omega_r}= D, \qquad r\in [0,, 1).
\end{equation}
For each  $r\in [0, 1)$ define the operator ${\bf B}_{\omega_r} : \cS  \to B(\cH)$
 by setting
 \begin{equation}\label{pois}
 {\bf B}_{\omega_r}(\chi):= {\bf K}_{\omega_r}^* (\chi\otimes I_\cH) {\bf K}_{\omega_r},
 \qquad \chi\in \cS.
 \end{equation}
  Using relation \eqref{KK} and \eqref{K*K},   we have
\begin{equation}\label{KSK}
{\bf K}_{\omega_r}^* ({\bf S}_{(\alpha)} {\bf S}_{(\beta)}^*\otimes I) {\bf K}_{\omega_r}= r^{|(\alpha)|+ |(\beta)|} A_{(\alpha)} DA_{(\beta)}^*,\qquad
(\alpha), (\beta)\in \FF_{n_1}^+\times \cdots\times \FF_{n_k}^+, \ r\in [0, 1),
\end{equation}
where $|(\alpha)|:=|\alpha_1|+\cdots +|\alpha_k|$ if  $(\alpha)=(\alpha_1,\ldots, \alpha_k)$.
Hence,  and due to relations \eqref{K*K} and  \eqref{pois}, we infer
that ${\bf B}_{\omega_r}$   is a completely positive  linear map with
${\bf B}_{\omega_r}(I)=D$ and  $\|{\bf B}_{\omega_r}\|=\|D\|$ for
$r\in [0,1)$.

Now, we show that $\lim_{r \to 1} {\bf B}_{\omega_r}(\chi)$ exists in the operator norm topology for each $\chi\in \cS$.
Given  a polynomial $p({\bf S}_{i,j}):= \sum_{(\alpha), (\beta)\in \FF_{n_1}^+\times \cdots\times \FF_{n_k}^+} a_{(\alpha)
(\beta)}{\bf S}_{(\alpha)} {\bf S}_{(\beta)}^*$  in the operator system $\cS$, we define
\begin{equation*}
p_D(A_{i,j}):=
\sum_{(\alpha), (\beta)\in \FF_{n_1}^+\times \cdots\times \FF_{n_k}^+} a_{(\alpha) (\beta)}A_{(\alpha)} DA_{(\beta)}^*.
\end{equation*}
The definition is correct since, according to  relation
\eqref{KSK}, we have the following von Neumman type inequality
\begin{equation}\label{von2}
\|p_D(A_{i,j})\|\leq \|D\|\|p({\bf S}_{i,j})\|.
\end{equation}
Now, fix $\chi\in   \cS$ and let
 $\{p^{(s)}({\bf S}_{i,j})\}_{s=1}^\infty$ be a sequence of polynomials in
 $ \cS$ convergent to $\chi$ in the operator norm topology.
 Define the operator
 \begin{equation}\label{fd}
\chi_D(A_{i,j}):= \lim_{s\to\infty}p^{(s)}_D(A_{i,j}).
\end{equation}
 Taking into account relation \eqref{von2}, one can see  that the
  operator $\chi_D(A_{i,j})$ is well-defined and
$
\|\chi_D(A_{i,j})\|\leq \|D\|\|\chi\|.
$
Due to relation  \eqref{KSK}, we have
$
\|p^{(s)}_D(rA_{i,j})\|\leq \|D\|\|p^{(s)}({\bf S}_{i,j})\|,
$
for any $r\in [0,, 1)$.
  Taking into account that ${\bf B}_{\omega_r}$ is a bounded linear operator and using again
   relation \eqref{KSK}, we obtain
 \begin{equation}\label{fdr}\begin{split}
  \lim_{s\to\infty} p^{(s)}_D(rA_{i,j})
 =\lim_{s\to\infty}{\bf K}_{\omega_r}^*(p^{(s)}({\bf S}_{i,j})\otimes I){\bf K}_{\omega_r}
 ={\bf B}_{\omega_r}[\chi],
 \end{split}
 \end{equation}
  for any $r\in [0, 1)$.
 Based on   relations \eqref{fd}, \eqref{fdr}, the fact that
  $\|\chi-p^{(s)}({\bf S}_{i,j})\|\to 0$ as $s\to \infty$, and
 $$
 \lim_{r\to 1}p^{(s)}_D(rA_{i,j})= p^{(s)}_D(A_{i,j}),
 $$
  we can deduce  that
 $
 \lim_{r \to 1} {\bf B}_{\omega_r}[\chi]=\chi_D(A_{i,j})
 $
 in the norm topology. Indeed, we have
 \begin{equation*}
 \begin{split}
\|\chi_D(A_{i,j})-{\bf B}_{\omega_r}[\chi]\|
&\leq \|\chi_D(A_{i,j})-p^{(s)}_D(A_{i,j})\|+\|p^{(s)}_D(A_{i,j})- {\bf B}_{\omega_r}(p^{(s)})\|
 + \|{\bf B}_{\omega_r}(p^{(k)})-{\bf B}_{\omega_r}(\chi)\|\\
&\leq \|\chi-p^{(s)}({\bf S}_{i,j})\|\|D\|+\|p^{(s)}_D(A_{i,j})-p^{(s)}_D(rA_{i,j})\|+\|\chi-p^{(s)}({\bf S}_{i,j})\|\|D\|.
 \end{split}
 \end{equation*}
 Taking into account that  ${\bf B}_{\omega_r}$ is a completely positive linear  map for any $r\in [0, 1)$
   and using
  relation \eqref{KSK},
    we infer  that
 $$
  \overline{{\bf B}}_{\omega}[\chi]:=\lim_{r \to 1}
 {\bf K}_{\omega_r }^* (\chi\otimes I) {\bf K}_{\omega_r}, \qquad   \chi\in
   \cS,
$$
 is a completely positive map such that  $  \overline{{\bf B}}_\omega(I)=D$ and
 $\overline{{\bf B}}_\omega({\bf S}_{(\alpha)} {\bf S}_{(\beta)}^*)=A_{(\alpha)} \overline{{\bf B}}_\omega(I)A_{(\beta)}$,
$\alpha, \beta\in \FF_n^+$.
 The proof is
complete.
\end{proof}
We recall that the variety algebra $\cA(\cV_\cQ)$ is the non-self-adjoint  norm closed algebra generated by universal model $\{{\bf S}_{i,j}\}$ and the identity.
 As a consequence of Theorem \ref{poisson}, we can obtain the following  extension of the noncommutative von
Neumann inequality (see \cite{von}, \cite{Po-von},
\cite{Po-poisson}, \cite{Po-Berezin}, \cite{Po-domains}).

\begin{corollary}\label{VN} Under  the hypotheses of Theorem \ref{poisson}, if
$D\in C_{\geq}({\bf \Delta_{f,A}^m})^+$, then
\begin{equation*}
  \left\|\sum_{(\alpha),(\beta)\in \Lambda} A_{(\alpha)} D
A_{(\beta)}^*\otimes C_{(\alpha),(\beta)}\right\| \leq \|D\|
\left\|\sum_{(\alpha),(\beta)\in \Lambda}
 {\bf S}_{(\alpha)} {\bf S}_{(\beta)}^*\otimes C_{(\alpha),(\beta)}\right\|
 \end{equation*}
 for any
  finite  set $\Lambda\subset \FF_{n_1}^+\times \cdots\times \FF_{n_k}^+$ and $C_{(\alpha),(\beta)}\in
  B(\cE)$, where $\cE$ is a Hilbert space. If, in addition, $D$ is
  an invertible operator, then  the map   $u:
  \cA(\cV_\cQ)\to B(\cH)$ defined by
  $$
  u(q({\bf S})):=q({\bf A}),\qquad q\in\CC\left<Z_{i,j}\right>,
  $$
  is completely bounded and $\|u\|_{cb}\leq \|D^{-1/2}\|\|D^{1/2}\|$.
\end{corollary}
\begin{proof} Note that relation \eqref{KSK} implies
\begin{equation*}
({\bf K}_{\omega_r}^* \otimes I_\cE)({\bf S}_{(\alpha)} {\bf S}_{(\beta)}^*\otimes I\otimes
C_{(\alpha),(\beta)}) ({\bf K}_{\omega_r}\otimes I_\cE)= r^{|(\alpha)|+ |(\beta)|}
A_{(\alpha)} DA_{(\beta)}^*\otimes C_{(\alpha),(\beta)}
\end{equation*}
for any  $(\alpha),(\beta)\in
\FF_{n_1}^+\times \cdots\times \FF_{n_k}^+$ and $ \ r\in [0, 1)$.
Taking into account that ${\bf K}_{\omega_r}^* {\bf K}_{\omega_r}= D$ for $ r\in [0, 1)$, one can easily
deduce the von Neumann type inequality. Now, assume that $D$ is invertible. Then the first part of this
corollary implies
\begin{equation*}
\begin{split}
\|q({\bf A})\|^2&\leq \|D^{-1/2}\|^2\|q({\bf A})D^{1/2}\|^2
= \|D^{-1/2}\|^2 \|q({\bf A})Dq({\bf A})^*\|\\
&\leq \|D^{-1/2}\|^2 \|D\| \|q({\bf S})q({\bf S})^*\|
=\|D^{-1/2}\|^2\|D^{1/2}\|^2 \|q({\bf S})\|^2
\end{split}
\end{equation*}
for any noncommutative polynomial $q$ in indeterminates $\{Z_{i,j}\}$. A similar result holds if we
pass to matrices. Therefore, we  deduce that $u$ is completely
bounded with $\|u\|_{cb}\leq \|D^{-1/2}\|\|D^{1/2}\|$. The proof is
complete.
\end{proof}

In what follows, we  study the noncommutative cone $C_{\geq}^{pure}({\bf \Delta_{f,A}^m})^+$
of all  pure solutions of the operator inequalities $
{\bf\Delta_{f,A}^p}(X)\geq 0$ for  any  ${\bf p}\in \ZZ_+^k$ with ${\bf p}\leq {\bf m}$.

\begin{theorem}\label{poisson2} Let $\cV_\cQ \subset {\bf D_f^m}$ be an abstract noncommutative variety, where  $\cQ$ is a family of
noncommutative
 polynomials in indeterminates $\{Z_{i,j}\}$ such that $\cN_\cQ\neq \{0\}$,  and let ${\bf S}=\{{\bf S}_{i,j}\}$
    be its  universal model.
      If
 ${\bf A}:=({ A}_1,\ldots, {A}_k)\in  B(\cH)^{n_1}\times_c\cdots \times_c B(\cH)^{n_k}$, where ${ A}_i:=(A_{i,1},\ldots, A_{i,n_i})\in B(\cH)^{n_i}$ has the property that $ \Phi_{f_i, A_i}$ is well-defined and $q({\bf A})=0$ for any $ q\in \cQ$,
then there is a bijection
$$\Gamma: CP^{w^*}(A,\cV_\cQ)\to C^{pure}_{\geq}({\bf \Delta_{f,A}^m})^+, \qquad
 \Gamma(\varphi):= \varphi (I),
 $$
  where  $  CP^{w^*}(A,\cV_\cQ)$ is   the set
  of all $w^*$-continuous completely positive linear maps \ $\varphi : \cS^{w^*}
 \to B(\cH)$ such that
$$ \varphi( {\bf S}_{(\alpha)} {\bf S}_{(\beta)}^*) =A_{(\alpha)} \varphi(I) A_{(\beta)}^*,
 \qquad (\alpha), (\beta)\in \FF_{n_1}^+\times \cdots\times \FF_{n_k}^+,
 $$
 where $$\cS^{w^*}:=
\overline{\text{\rm span}}^{w^*} \{ {\bf S}_{(\alpha)} {\bf S}_{(\beta)}^*:\ (\alpha), (\beta)\in \FF_{n_1}^+\times \cdots\times \FF_{n_k}^+\}.
 $$
 In addition,  if $D\in C^{pure}_{\geq}({\bf \Delta_{f,A}^m})^+$, then
$\Gamma^{-1}(D)$ coincides
 with the constrained noncommutative
Berezin  transform  associated with $\omega:=({\bf f,m, A},R,\cQ)$ which is defined by
$$
  {\bf B}_{\omega}[\chi]:=
 {\bf K}_{\omega }^* (\chi\otimes I) {\bf K}_{\omega} ,\qquad \chi\in
   \cS,
$$
where  $\omega:=({\bf f,m,A},R,\cQ)$ and $R:={\bf \Delta_{f,A}^m}(D)$.
\end{theorem}

\begin{proof} Let
$\varphi : \cS^{w^*}
 \to B(\cH)$ be a $w^*$-continuous completely positive linear
map such that
$$ \varphi( {\bf S}_{(\alpha)} {\bf S}_{(\beta)}^*) =A_{(\alpha)} \varphi(I) A_{(\beta)}^*,
 \qquad (\alpha), (\beta)\in \FF_{n_1}^+\times \cdots\times \FF_{n_k}^+.
 $$
Setting $D:=\varphi(I)$ and using the fact that
$\Phi_{f_i, r{\bf S}_i}(I)=\sum_{k=1}^\infty
\sum_{\alpha_i\in \FF_{n_i}^+, |\alpha_i|=k} a_{i,\alpha_i}   {\bf S}_{i,\alpha_i} {\bf S}_{i,\alpha_i} ^* $ is
SOT convergent, we deduce that
$$
{\bf \Delta_{f,A}^p}(D)=\varphi \left( {\bf \Delta_{f,S}^m}(I)\right)\geq
0
$$
${\bf p}\in \ZZ_+^k$ with ${\bf p}\leq {\bf m}$.
On the other hand,  $\{\Phi_{f_i,{\bf S}_i}^s(I)\}_{s=1}^\infty$ is
a
 bounded decreasing sequence of positive operators  which converges
 weakly to $0$, as $s\to \infty$. Since
 $\Phi_{f_i,A_i}^s(D)=\varphi(\Phi_{f_i,{\bf S}_i}^s(I))$ for all  $s\in \NN$,
  $\{\Phi_{f_i,A_i}^k(D)\}_{s=1}^\infty$ is also
a
 bounded decreasing sequence of positive operators  which converges
 weakly, as $s\to \infty$.
Taking into account that  $\varphi$ is
  continuous  in the  $w^*$-topology, which coincides with the weak
  operator topology   on
bounded sets, we deduce that $\Phi_{f_i,A_i}^s(D)\to 0$ weakly, as
$s\to \infty$. Therefore, $D\in C^{pure}_{\geq}({\bf \Delta_{f,A}^m})^+$.

 To
prove that $\Gamma$ is one-to-one, let $\varphi_1$ and $\varphi_2$
be $w^*$-continuous completely positive linear maps on
$\cS^{w^*}$ such that $ \varphi_j( {\bf S}_{(\alpha)} {\bf S}_{(\beta)}^*) =A_{(\alpha)} \varphi_j(I) A_{(\beta)}^*$ for any
 $(\alpha), (\beta)\in \FF_{n_1}^+\times \cdots\times \FF_{n_k}^+$ and $j=1,2$.  Assume that
$\Gamma(\varphi_1)=\Gamma(\varphi_2)$, i.e.,
$\varphi_1(I)=\varphi_2(I)$. Then we have $\varphi_1( {\bf S}_{(\alpha)} {\bf S}_{(\beta)}^*)=\varphi_2( {\bf S}_{(\alpha)} {\bf S}_{(\beta)}^*)$ for $(\alpha), (\beta)\in \FF_{n_1}^+\times \cdots\times \FF_{n_k}^+$. Since     $\varphi_1$ and $\varphi_2$ are
$w^*$-continuous, we deduce that $\varphi_1=\varphi_2$.

To prove that $\Gamma$ is a surjective map, let  $D\in
 C^{pure}_{\geq}({\bf \Delta_{f,A}^m})^+$ be fixed. Due to Lemma \ref{lemma2}, the
constrained noncommutative Berezin kernel
  ${\bf K}_\omega$   associated with a compatible tuple
$\omega:=({\bf f,m, A},R,\cQ)$ satisfies the equation
\begin{equation}\label{KABK}
    {\bf K}_\omega { A}^*_{i,j}= ({\bf S}_{i,j}^*\otimes I_\cH)  {\bf K}_\omega,
\end{equation}
 where ${\bf S}=\{{\bf S}_{i,j}\}$
    is the  universal model
     associated
  with the abstract noncommutative
  variety $\cV_\cQ$. Moreover,
$$
{\bf K}_\omega^*\,{\bf K}_\omega=
\sum_{(s_1,\ldots,s_k)\in \ZZ_+^k}\left(\begin{matrix} s_1+m_1-1\\m_1-1\end{matrix}\right)\cdots \left(\begin{matrix} s_k+m_k-1\\m_k-1\end{matrix}\right)\Phi_{f_1,A_1}^{s_1}\circ \cdots \circ \Phi_{f_k, A_k}^{s_k}(R),
$$
where $R:= {\bf \Delta_{f,A}^m}(D)$ and the convergence is  in the weak  operator topology.
Using   Theorem \ref{reproducing},  we   obtain
\begin{equation*}
D=\sum_{(s_1,\ldots,s_k)\in \ZZ_+^k}\left(\begin{matrix} s_1+m_1-1\\m_1-1\end{matrix}\right)\cdots \left(\begin{matrix} s_k+m_k-1\\m_k-1\end{matrix}\right)\phi_{f_1, A_1}^{s_1}\circ \cdots \circ \phi_{f_k, A_k}^{s_k}( R),
\end{equation*}
where
 the convergence of the series is in the weak operator topology.
Consequently, we deduce that ${\bf K}_\omega^*\,{\bf K}_\omega=D$.
Define the operator ${\bf B}_{\omega} : \cS^{w^*}  \to B(\cH)$
 by setting
 \begin{equation*}
 {\bf B}_{\omega}(\chi):= {\bf K}_{\omega}^* (\chi\otimes I_\cH) {\bf K}_{\omega},
 \qquad \chi\in \cS^{w^*}.
 \end{equation*}
   Now,  due to  relation \eqref{KABK} it is easy to see that
\begin{equation*}
{\bf B}_{\omega}({\bf S}_{(\alpha)} {\bf S}_{(\beta)}^*)={\bf K}_{\omega}^* ({\bf S}_{(\alpha)} {\bf S}_{(\beta)}^*\otimes I) {\bf K}_{\omega}=  A_{(\alpha)} DA_{(\beta)}^*,\qquad
(\alpha), (\beta)\in \FF_{n_1}^+\times \cdots\times \FF_{n_k}^+.
\end{equation*}
 Consequently,  ${\bf B}_{\omega}\in CP^{w^*}(A,\cV_\cQ)$    has
 the required properties.
The proof is complete.
\end{proof}

%%%%%%%%%%%

We remark that an operator $D\in B(\cH)$ is in $C^{pure}_{\geq}({\bf \Delta_{f,A}^m})^+$
 if and only if there is a Hilbert space $\cD$ and an operator
 $K:\cH\to \cN_\cQ\otimes \cD$ such that
 \begin{equation*}
 D=K^*K\quad \text{ and } \quad KA_{i,j}^*= ({\bf S}_{i,j}^*\otimes I_\cD)K, \qquad  i\in \{1,\ldots, k\}, j\in \{1,\ldots, n_i\}.
 \end{equation*}
Indeed,  the direct
implication follows if we take $K$ to be the noncommutative Berezin
kernel ${\bf K}_\omega$.
   To prove the converse,  assume that there is a Hilbert space $\cD$ and an operator
 $K:\cH\to \cN_\cQ\otimes \cD$ such that
 \begin{equation*}
 D=K^*K\quad \text{ and } \quad KA_{i,j}^*= ({\bf S}_{i,j}^*\otimes I_\cD)K, \qquad i\in \{1,\ldots, k\}, j\in \{1,\ldots, n_i\}.
 \end{equation*}
 Then
$$
{\bf \Delta_{f,A}^p}(D)=K^* \left[{\bf \Delta_{f,A}^p}(I)\otimes I_\cD
\right]K\geq 0
$$
for ${\bf p}\in \ZZ_+^k$ with ${\bf p}\leq {\bf m}$.
  Since $\Phi_{f_i,A_i}^s(D)=K^*[\Phi_{f_i,{\bf S}_i}^s(I)\otimes I_\cD]K$,
  $\|\Phi_{f_i,{\bf S}_i}^s(I)\|\leq 1$,
and $\Phi_{f_i,{\bf S}_i}^s(I)\to 0$ weakly, as $s\to 0$, we deduce that
$D\in  C^{pure}_{\geq}({\bf \Delta_{f,A}^m})^+$. This proves our assertion.

We should mention that, in Theorem \ref{poisson2}, the set $\cQ$  is of
arbitrary noncommutative polynomials with $\cN_\cQ\neq \{0\}$,
while, in Theorem \ref{poisson}, $\cQ$ consists of homogeneous
polynomials.

The proof of the next result is similar to that of Corollary
\ref{VN}, so we shall omit it. We recall (see \cite{Po-Berezin3}) that $F^\infty(\cV_\cQ)$ is the WOT-closed algebra generated by all polynomials in ${\bf S}_{i,j}$ and the identity.
\begin{corollary}\label{VN2} Under  the hypotheses of Theorem \ref{poisson2}, if
$D\in C^{pure}_{\geq}({\bf \Delta_{f,A}^m})^+$, then
\begin{equation*}
  \left\|\sum_{(\alpha),(\beta)\in \Lambda} A_{(\alpha)} D
A_{(\beta)}^*\otimes C_{(\alpha),(\beta)}\right\| \leq \|D\|
\left\|\sum_{(\alpha),(\beta)\in \Lambda}
 {\bf S}_{(\alpha)} {\bf S}_{(\beta)}^*\otimes C_{(\alpha),(\beta)}\right\|
 \end{equation*}
 for any
  finite  set $\Lambda\subset \FF_{n_1}^+\times \cdots\times \FF_{n_k}^+$ and $C_{(\alpha),(\beta)}\in
  B(\cE)$, where $\cE$ is a Hilbert space. If, in addition, $D$ is
  an invertible operator, then  the map   $u:
  F^\infty(\cV_\cQ)\to B(\cH)$ defined by
  $$
  u( \varphi):= {\bf K}_\omega [\varphi\otimes I_\cH]{\bf K}_\omega D^{-1},\qquad \varphi\in F_n^\infty(\cV_\cQ),
  $$
  where  ${\bf K}_\omega$ is the constrained noncommutative Berezin kernel associated with the compatible tuple $\omega:=({\bf f,m, A},R,\cQ)$ and $R:={\bf \Delta_{f,A}^m}(D)$, is completely bounded and
   $\|u\|_{cb}\leq \|D^{-1/2}\|\|D^{1/2}\|$.
\end{corollary}

Our last result of this section  is a characterization of the noncommutative cone $C_{\geq}({\bf \Delta_{f,A}^m})^+$.

\begin{theorem}\label{pure} Let $\cV_\cQ \subset {\bf D_f^m}$ be an abstract noncommutative variety, where  $\cQ$ is a family of
noncommutative
 polynomials in indeterminates $\{Z_{i,j}\}$ such that $\cN_\cQ\neq \{0\}$,  and let
 ${\bf A}:=({ A}_1,\ldots, {A}_k)\in  B(\cH)^{n_1}\times_c\cdots \times_c B(\cH)^{n_k}$, where ${ A}_i:=(A_{i,1},\ldots, A_{i,n_i})\in B(\cH)^{n_i}$, have the property that $ \Phi_{f_i, A_i}$ is well-defined and $q({\bf A})=0$ for any $ q\in \cQ$.

 Then   a positive operator  $\Gamma\in B(\cH)$  is in $C_{\geq}({\bf \Delta_{f,A}^m})^+$
if and only
  if there   is a tuple  ${\bf T}:=({ T}_1,\ldots, {T}_k)\in  B(\cH)^{n_1}\times_c\cdots \times_c B(\cH)^{n_k}$, with ${ T}_i:=(T_{i,1},\ldots, T_{i,n_i})\in B(\cH)^{n_i}$,  in the noncommutative
 variety $ \cV_\cQ(\cH)$
  such that
 \begin{equation*}
 A_{i,j} \Gamma^{1/2}= \Gamma^{1/2} T_{i,j}, \qquad i\in \{1,\ldots, k\}, j\in \{1,\ldots, n_i\}.
 \end{equation*}
In addition,  $\Gamma\in C^{pure}_{\geq}({\bf \Delta_{f,A}^m})^+$    if and only if  \ $I_\cH \in
C^{pure}_{\geq}({\bf \Delta_{f,T}^m})^+$.
\end{theorem}
\begin{proof}
Assume that  ${\bf T}\in \cV_\cQ(\cH)$
and
  $A_{i,j} \Gamma^{1/2}= \Gamma^{1/2} T_{i,j}$ for $ i\in \{1,\ldots, k\}$ and $j\in \{1,\ldots, n_i\}$.
 Note that
 $$
  {\bf \Delta_{f,A}^p}(\Gamma)= \Gamma^{1/2}\left[{\bf \Delta_{f,T}^m}(I)\right] \Gamma^{1/2}
\geq 0
 $$
 for any ${\bf p}\in \ZZ_+^k$ with ${\bf p}\leq {\bf m}$.
Since $ \Phi_{f_i,A_i}^s(\Gamma)=\Gamma^{1/2}  \Phi_{f_i,T_i}^s(I)
\Gamma^{1/2} $, $s\in \NN$, we deduce that if $ \Phi_{f_i,T_i}^s(I)\to 0$
weakly as $s\to\infty$.  Therefore,  $\Gamma\in C^{pure}_{\geq}(f,A)^+$ .

 Now, we prove the converse.  Assume that $\Gamma\in B(\cH)$  is in $C_{\geq}({\bf \Delta_{f,A}^m})^+$. Let $f_i:=\sum_{\alpha\in \FF_{n_i}^+} a_{i,\alpha} Z_{i,\alpha}$ and note that
\begin{equation*}
\begin{split}
\left< \Phi_{f_i, A_i}(\Gamma)x,x\right>\leq \|\Gamma^{1/2} x\|^2
\end{split}
\end{equation*}
for any $x\in   \cH$. Hence, we deduce that  $a_{i,g_j^i}\|\Gamma^{1/2}A_{i,j}^*
x\|^2\leq \|\Gamma^{1/2} x\|^2$, for any $x\in   \cH$. Recall that
$a_{i,g_j^i}\neq 0$,  so we can define the operator
 $\Lambda_{i,j}:  \Gamma^{1/2}(\cH)\to
   \Gamma^{1/2}(\cH) $  by setting
 \begin{equation}\label{GiC}
 \Lambda_{i,j} \Gamma^{1/2}x:= \Gamma^{1/2} A_{i,j}^*x, \qquad  x\in \cH,
 \end{equation}
 for $ i\in \{1,\ldots, k\}$ and $j\in \{1,\ldots, n_i\}$.
It is obvious that $\Lambda_{i,j}$
   can    be extended to a bounded operator (also
denoted by $\Lambda_{i,j}$) on the subspace $\cM:=\overline{
\Gamma^{1/2}(\cH)}$. Set ${\bf M}=(M_1,\ldots, M_k)$ with $M_i:=(M_{i,1},\ldots, M_{i, n_i})$ and $M_{i,j}:=\Lambda_{i,j}^*$, and note
that
$$
\Gamma^{1/2}\left[{\bf \Delta_{f,M}^p}(I_\cM)\right] \Gamma^{1/2}=
{\bf \Delta_{f,A}^p}(\Gamma)\geq 0
$$
for ${\bf p}\in \ZZ_+^k$ with ${\bf p}\leq {\bf m}$.
An approximation argument shows that
$
{\bf \Delta_{f,M}^p}(I_\cM)  \geq 0.
$
 For each $ i\in \{1,\ldots, k\}$ and $j\in \{1,\ldots, n_i\}$, define $T_{i,j}:= M_{i,j}\oplus 0$
  with respect to the decomposition
 $\cH= \cM\oplus \cM^\perp$,
   and note that
 ${\bf \Delta_{f,T}^p}(I)\geq 0$.
 If $q\in \cQ$, then relation \eqref{GiC} implies
$
q({\bf M})^*\Gamma^{1/2} =\Gamma^{1/2} q({\bf A})^*=0.
$
Hence,  $q({\bf M})=0$ and, consequently,   $q({\bf T})=0$ for all $q\in \cQ$. Therefore,  ${\bf T}:=\{T_{i,j}\}\in
\cV_\cQ(\cH)$ and  $A_{i,j} \Gamma^{1/2}= \Gamma^{1/2} T_{i,j}$ for $ i\in \{1,\ldots, k\}$ and $j\in \{1,\ldots, n_i\}$.

 Assume   that  $\Gamma\in C^{pure}_{\geq}({\bf \Delta_{f,A}^m})^+$. Then, for each $i\in \{1,\ldots, k\}$,  $ \Phi_{f_i,A_i}^s(\Gamma)\to 0$ weakly,
   as $s\to\infty$.
Taking into account that
\begin{equation*}
\begin{split}
\left< \Phi_{f_i, T_i}^s(I) \Gamma^{1/2}x, \Gamma^{1/2} x\right> &= \left<
\Phi_{f_i,
 A_i}^s(\Gamma)x,  x\right>,\qquad x\in   \cH,
\end{split}
\end{equation*}
  we have \ WOT-$\lim_{s\to\infty}\Phi_{f_i,T_i}^s(I)y=0$ for
any $y\in \text{\rm range}~\Gamma^{1/2}$. Since
$\|\Phi_{f_i,T_i}^s(I)\|\leq 1$, $s\in \NN$,  an approximation argument
shows that WOT-$\lim_{s\to\infty}\Phi_{f_i,T_i}^s(I)y=0$ for any $y\in
\overline {\Gamma^{1/2}(\cH)}$. Note also that
$\Phi_{f_i,T_i}^s(I)z=0$ for any $z\in \cM^\perp$. Consequently,
 $I_\cH \in
C^{pure}_{\geq}({\bf \Delta_{f,T}^m})^+$.
 This completes the proof.
\end{proof}

\bigskip

\section{Analogues of  Rota's similarity results
 for noncommutative  polydomains}

Let ${\bf f}:=(f_1,\ldots, f_k)$ be a $k$-tuple of positive regular free holomorphic functions and let ${\bf m}=(m_1,\ldots, m_k)$ be in $ \NN^k$.
Consider ${\bf A}:=({ A}_1,\ldots, {A}_k)\in  B(\cH)^{n_1}\times\cdots \times B(\cH)^{n_k}$, where ${ A}_i:=(A_{i,1},\ldots, A_{i,n_i})\in B(\cH)^{n_i}$,  to be such that $\Phi_{f_i, A_i}(I)$ is well-defined in the weak operator topology, and let $\cQ$ be a set of noncommutative polynomials in indeterminates $\{Z_{i,j}\}$ with $i\in \{1,\ldots, k\}$ and $j\in \{1,\ldots, n_i\}$. Given another tuple ${\bf B}:=({ B}_1,\ldots, {B}_k)\in  B(\cK)^{n_1}\times\cdots \times B(\cK)^{n_k}$, where ${ B}_i:=(B_{i,1},\ldots, B_{i,n_i})\in B(\cK)^{n_i}$, we say the ${\bf A}$ is jointly similar to ${\bf B}$ if there exists an invertible operator $Y:\cK\to \cH$
such that
$$
A_{i,j}=Y B_{i,j} Y^{-1}
$$
for all  $i\in \{1,\ldots,k\}$ and  $j\in \{1,\ldots, n_i\}$.

In this section we provide necessary and sufficient conditions for a tuple ${\bf A}=({ A}_1,\ldots, {A}_k)$ to
be jointly similar  to a  tuple   ${\bf T}:=({ T}_1,\ldots, {T}_k)\in  B(\cH)^{n_1}\times\cdots \times B(\cH)^{n_k}$
satisfying one of the following properties:
\begin{enumerate}
\item[(i)]
${\bf T}\in {\cV}_{\cQ}(\cH):=\left\{ {\bf X}\in {\bf D_f^m}(\cH): \ q({\bf X})=0, q\in \cQ\right\} $;
\item[(ii)] ${\bf T}\in \left\{ {\bf X} \in {\cV}_{\cQ}(\cH): \
 {\bf \Delta_{f,X}^p}(I)>0  \text{ \rm  for } 0\leq {\bf p}\leq {\bf m}, {\bf p}\neq 0\right\}$;
\item[(iii)]  ${\bf T}$ is a pure  tuple in ${\cV}_{\cQ}(\cH)$, i.e. for each $i\in \{1,\ldots,k\}$, $\Phi_{f_i,T_i}^k(I)\to 0$ weakly as $k\to \infty$.
\end{enumerate}
  We show that these similarities are strongly related to the
existence of invertible positive solutions of the operator
inequalities ${\bf \Delta_{f,A}^p}(Y)\geq 0$ and  ${\bf \Delta_{f,A}^p}(Y)> 0$.

Let $f=\sum_{\alpha\in \FF_{n}^+}  a_\alpha X_\alpha$, $a_\alpha\in \CC$,
be a positive regular free holomorphic function. For any $n$-tuple of
operators $C:=(C_1,\ldots, C_n)\in B(\cH)^n$   such that
$\sum_{|\alpha|\geq 1} a_\alpha C_\alpha C_\alpha^* $ is convergent
in the weak operator topology,  define the joint spectral radius
with respect to the noncommutative domain ${\bf D}^m_f$ by setting
$$
r_f(C):=\lim_{k\to\infty}\|\Phi_{f,C}^k(I)\|^{1/2k},
$$
where the positive linear map $\Phi_{f,C}:B(\cH)\to B(\cH)$ is given
by
$$
\Phi_{f,C}(X):=\sum_{\alpha\in \FF_{n}^+} a_\alpha C_\alpha
XC_\alpha^*,\qquad X\in B(\cH),
$$
and  the convergence is in the week operator topology.
In the particular case when $f:=X_1+\cdots +X_n$, we obtain the
usual definition of the joint spectral radius for $n$-tuples of noncommuting
operators.

Our first result provides  necessary conditions for joint similarity
to tuples of operators in noncommutative varieties
$\cV_{\cQ}(\cH)$. Since the proof is straightforward, we leave it to the reader.

\begin{proposition} \label{properties} Let ${\bf f}:=(f_1,\ldots, f_k)$ be a $k$-tuple of positive regular free holomorphic functions with
  $$f_i:=\sum_{\alpha\in \FF_{n_i}^+} a_{i,\alpha} X_{i,\alpha}
  $$
  and let $\cQ$ be a set of noncommutative polynomials in indeterminates $\{Z_{i,j}\}$, where $i\in \{1,\ldots,k\}$ and  $j\in \{1,\ldots, n_i\}$.
   If  ${\bf T}:=({ T}_1,\ldots, {T}_k)\in \cV_{\cQ}(\cH)\subset {\bf D_f^m}(\cH)$ and
 ${\bf A}:=({ A}_1,\ldots, {A}_k)\in  B(\cK)^{n_1}\times\cdots \times B(\cK)^{n_k}$ are  two tuples of operators which are jointly
similar,
 then, for each $i\in \{1,\ldots,k\}$, the following statements hold:
  \begin{enumerate}
  \item[(i)]$({ A}_1,\ldots, {A}_k)\in  B(\cK)^{n_1}\times_c\cdots \times_c B(\cK)^{n_k}$ and $\sum_{\alpha_i\in \FF_{n_i}^+ } a_{i,\alpha_i} A_{i,\alpha_i} A_{i,\alpha_i}^*$ is convergent in the weak operator topology;
       \item[(ii)] $\Phi_{f_i,A_i}$ is a power bounded completely positive linear map;
      \item[(iii)] $r_{f_i}(A_i)\leq 1$;
     \item[(iv)] $q({\bf A})=0$ for all $q\in \cQ$;
     \item[(v)] if $\Phi_{f_i,T_i}^s(I)\to 0$
 weakly as $s\to \infty$, then $\Phi_{f_i,A_i}^s(I)\to 0$
 weakly.
      \end{enumerate}
\end{proposition}

In what  follows, we assume that
  ${\bf f}:=(f_1,\ldots, f_k)$ is a $k$-tuple of positive regular free holomorphic functions with
  $f_i:=\sum_{\alpha\in \FF_{n_i}^+} a_{i,\alpha} Z_{i,\alpha}$ and ${\bf m}=(m_1,\ldots,m_k)\in \NN^k$. Moreover, let $\cQ$ be a set of noncommutative polynomials in indeterminates $\{Z_{i,j}\}$ and  let
 ${\bf A}:=({ A}_1,\ldots, {A}_k)\in  B(\cH)^{n_1}\times_c\cdots \times_c B(\cH)^{n_k}$, where ${ A}_i:=(A_{i,1},\ldots, A_{i,n_i})\in B(\cH)^{n_i}$ has the property that $\sum_{\alpha\in \FF_{n_i}^+} a_{i,\alpha} A_{i,\alpha}A_{i,\alpha}^*$
     is  weakly convergent and $q({\bf A})=0$ for any $ q\in \cQ$.

Now, we are ready to  provide necessary and sufficient conditions for the joint similarity to
  parts of the adjoints
  of  the universal model ${\bf S}:=({\bf S}_1,\ldots, {\bf S}_k)$, where ${\bf S}_i:=({\bf S}_{i,1}\ldots, {\bf S}_{i,n_i})$,  associated with the
abstract noncommutative variety $\cV_Q$.

\begin{theorem}\label{pure-contr}  Let $\cQ$ be a set of noncommutative polynomials in indeterminates $\{Z_{i,j}\}$, where   $i\in \{1,\ldots,k\}$ and  $j\in \{1,\ldots, n_i\}$, and  let
 ${\bf A}:=({ A}_1,\ldots, {A}_k)\in  B(\cH)^{n_1}\times_c\cdots \times_c B(\cH)^{n_k}$, where ${ A}_i:=(A_{i,1},\ldots, A_{i,n_i})\in B(\cH)^{n_i}$ has the property that $\sum_{\alpha\in \FF_{n_i}^+} a_{i,\alpha} A_{i,\alpha}A_{i,\alpha}^*$
     is  weakly convergent and $q({\bf A})=0$ for any $ q\in \cQ$.
Then the  following
statements are equivalent.
\begin{enumerate}
\item[(i)] There exists an invertible operator $Y:
\cH\to \cG$ such that
$$
A_{i,j}^*=Y^{-1}[({\bf S}_{i,j}^*\otimes I_\cH)|_\cG]Y
$$
for all  $i\in \{1,\ldots,k\}$ and  $j\in \{1,\ldots, n_i\}$,
where
 $\cG\subseteq \cN_\cQ\otimes \cH$ is an invariant  subspace under  each operator ${\bf S}_{i,j}^*\otimes
 I_\cH$.
 \item[(ii)] There is   an
 invertible operator $Q\in \cC_{\geq}({\bf \Delta_{f,A}^m})^+$ such that  $ \Phi_{f_i,A_i}^s(Q)\to 0$ weakly, as $s\to \infty$.
\item[(iii)]  There exist
constants $0<a\leq b$ and a positive operator $R\in B(\cH)$ such
that
$$
aI\leq \sum_{(s_1,\ldots,s_k)\in \ZZ_+^k}\left(\begin{matrix} s_1+m_1-1\\m_1-1\end{matrix}\right)\cdots \left(\begin{matrix} s_k+m_k-1\\m_k-1\end{matrix}\right)\Phi_{f_1,A_1}^{s_1}\circ \cdots \circ \Phi_{f_k, A_k}^{s_k}(R)\leq bI.
$$
\end{enumerate}
Moreover, under the condition (iii), one can choose the invertible operator $Y$ such that $\|Y\|\|Y^{-1}\|\leq \sqrt{\frac{b}{a}}$.
\end{theorem}

\begin{proof} We prove that (i) $\Rightarrow$ (ii).  Assume
that (i) holds and let $a,b>0$ be such that $aI\leq Y^*Y\leq bI$.
Setting $Q:=Y^*Y$  simple
 calculations reveal that
$$
(id-\Phi_{f_1,A_1})^{p_1}\circ\cdots \circ(id-\Phi_{f_k,A_k})^{p_k}(Q)=Y^*\left\{P_\cG\left[ (id-\Phi_{f_1,{\bf S}_1})^{p_1}\circ\cdots \circ(id-\Phi_{f_k,{\bf S}_k})^{p_k}(I)\otimes I\right]|_\cG \right\}Y\geq 0
$$
 for any $p_i\in \{0,1,\ldots, m_i\}$ and
       $i\in \{1,\ldots,k\}$.
Therefore, $Q\in \cC_{\geq}({\bf \Delta_{f,A}^m})^+$. Since  ${\bf S}:=({\bf S}_1,\ldots, {\bf S}_k)$ is a pure
tuple, we have
$\Phi_{f_i,{\bf S}_i}^s(I)\to 0$ weakly, as $s\to\infty$. Taking into
account that $\Phi_{f_i,A_i}^s(Q)= Y^*\left[P_\cG(\Phi_{f_i,{\bf S}_i}^s(I)\otimes
I)|_\cG\right]Y$ for $s\in \NN$, we deduce that
 $\Phi_{f_i,A_i}^s(Q)\to 0$
weakly as $s\to\infty$. Therefore   item  (ii) holds.

Now, we prove the implication  (ii) $\Rightarrow$ (iii).   Let $Q\in \cC_{\geq}({\bf \Delta_{f,A}^m})^+$
be an  invertible operator such that
  $\Phi_{f_i,A_i}^s(Q)\to 0$ weakly as $s\to\infty$. Set
$R:={\bf \Delta_{f,A}^m}(Q)$ and note  that, using  Theorem \ref{reproducing2} and Proposition \ref{pure2}, we obtain
\begin{equation*}\begin{split}
 \sum_{(s_1,\ldots,s_k)\in \ZZ_+^k}&\left(\begin{matrix} s_1+m_1-1\\m_1-1\end{matrix}\right)\cdots \left(\begin{matrix} s_k+m_k-1\\m_k-1\end{matrix}\right)\Phi_{f_1,A_1}^{s_1}\circ \cdots \circ \Phi_{f_k, A_k}^{s_k}(R)\\
 &=\lim_{q_k\to\infty}\ldots \lim_{q_1\to\infty}  (id-\Phi_{f_k, A_k}^{q_k})\circ\cdots \circ(id-\Phi_{f_1,A_1}^{q_1})(Q)=Q
\end{split}
\end{equation*}
 where
 the convergence of the series is in the weak operator topology.
 Hence, we deduce item (iii).
It remains to show that    (iii) $\Rightarrow$ (i). Assume that item
(iii) holds.   Let   ${\bf K}_\omega:\cH\to \cN_\cQ\otimes \cH$ be  the constrained
Berezin kernel  associated with a compatible tuple
$\omega:=({\bf f,m, A},R,\cQ)$.  According to Theorem \ref{lemma2}, we have
\begin{equation}\label{KASK}
    {\bf K}_\omega { A}^*_{i,j}= ({\bf S}_{i,j}^*\otimes I_\cH)  {\bf K}_\omega,
\end{equation}
where   ${\bf S}=\{{\bf S}_{i,j}\}$
    is the  universal model
     associated
  with the abstract noncommutative
  variety $\cV_\cQ$. Moreover,
$$
{\bf K}_\omega^*\,{\bf K}_\omega=
\sum_{(s_1,\ldots,s_k)\in \ZZ_+^k}\left(\begin{matrix} s_1+m_1-1\\m_1-1\end{matrix}\right)\cdots \left(\begin{matrix} s_k+m_k-1\\m_k-1\end{matrix}\right)\Phi_{f_1,A_1}^{s_1}\circ \cdots \circ \Phi_{f_k, A_k}^{s_k}(R),
$$
where the convergence is  in the weak  operator topology. Consequently,
we have
\begin{equation*}
 a\|h\|^2\leq \|{\bf K}_\omega h\|^2 \leq b\|h\|^2,\qquad h\in \cH,
\end{equation*}
 and  the range of $ {\bf K}_\omega $
is a closed subspace of $\cN_\cQ\otimes \cH$.
  Since  the operator $Y:\cH\to \text{\rm range}\,{\bf K}_\omega $ defined by $Yh:={\bf K}_\omega h$, $h\in \cH$,  is invertible, relation \eqref{KASK} implies
\begin{equation*}
A_{i,j}^*=Y^{-1}[({\bf S}_{i,j}^*\otimes I_\cH)|_\cG]Y
\end{equation*}
for all  $i\in \{1,\ldots,k\}$ and  $j\in \{1,\ldots, n_i\}$,
where $\cG:=\text{\rm range}\, {\bf K}_\omega $.
   This proves (i).
   The proof is complete.
\end{proof}

We remark that under the conditions of Theorem \ref{pure-contr},
part (iii), one can
show that the mapping  $\Psi:\cA(\cV_\cQ)\to B(\cH)$ defined by
$$\Psi(g({\bf S}_{i,j})):=
 g(A_{i,j}),\qquad g\in \CC\left<Z_{i,j}\right>,$$  is completely bounded  with ~$\|\Psi\|_{cb}\leq
  \sqrt{ \frac {b} {a}}$.

Taking $R=I$ in Theorem \ref{pure-contr}, we can obtain the
following
  analogue of Rota's  model  theorem,  for  similarity to tuples of operators  in the  noncommutative
variety $\cV_\cQ(\cH)$.

\begin{corollary} Let $\cQ$ be a set of noncommutative polynomials in indeterminates $\{Z_{i,j}\}$,  where $i\in \{1,\ldots,k\}$ and  $j\in \{1,\ldots, n_i\}$, and  let
 ${\bf A}:=({ A}_1,\ldots, {A}_k)\in  B(\cH)^{n_1}\times_c\cdots \times_c B(\cH)^{n_k}$, where ${ A}_i:=(A_{i,1},\ldots, A_{i,n_i})\in B(\cH)^{n_i}$ has the property that $\sum_{\alpha\in \FF_{n_i}^+} a_{i,\alpha} A_{i,\alpha}A_{i,\alpha}^*$
     is  weakly convergent and $q({\bf A})=0$ for any $ q\in \cQ$.
 If
$$
 \sum_{(s_1,\ldots,s_k)\in \ZZ_+^k}\left(\begin{matrix} s_1+m_1-1\\m_1-1\end{matrix}\right)\cdots \left(\begin{matrix} s_k+m_k-1\\m_k-1\end{matrix}\right)\Phi_{f_1,A_1}^{s_1}\circ \cdots \circ \Phi_{f_k, A_k}^{s_k}(I)\leq bI
$$
for some constant  $ b>0$, then there exists an invertible
operator $Y:  \cH\to \cG$ such that
$$
A_{i,j}^*=Y^{-1}[({\bf S}_{i,j}^*\otimes I_\cH)|_\cG]Y
$$
for all  $i\in \{1,\ldots,k\}$ and  $j\in \{1,\ldots, n_i\}$,
where
 $\cG\subseteq \cN_\cQ\otimes \cH$ is an invariant  subspace under  each operator ${\bf S}_{i,j}^*\otimes
 I_\cH$,  and  ${\bf S}:=({\bf S}_1,\ldots, {\bf S}_k)$, with ${\bf S}_i:=({\bf S}_{i,1}\ldots, {\bf S}_{i,n_i})$, is the universal model associated with the abstract
noncommutative variety $\cV_Q$.
\end{corollary}

Let ${\bf L}:=({\bf L}_1,\ldots, {\bf L}_k)$, with ${\bf L}_i:=({\bf L}_{i,1},\ldots, {\bf L}_{i,n_i})$,  be  the universal model associated with the closed noncommutative polyball
$[B(\cH)^{n_1}]_1^-\times \cdots \times [B(\cH)^{n_k}]_1^-$.  More precisely,
  the operator
$${\bf L}_{i,j}:=\underbrace{I\otimes\cdots\otimes I}_{\text{${i-1}$
times}}\otimes L_{i,j}\otimes \underbrace{I\otimes\cdots\otimes
I}_{\text{${k-i}$ times}},
$$
 is acting on the tensor Hilbert space
$F^2(H_{n_1})\otimes\cdots\otimes F^2(H_{n_k})$ and
 $L_{i,j}:F^2(H_{n_i})\to F^2(H_{n_i})$  is the left creation operator  defined by
 $L_{i,j} e_\alpha^i:=e^i_j\otimes e_{\alpha}^i$ for  $\alpha \in \FF_{n_i}^+$.
 Let $\pi_{{\bf L}_i}:\FF_{n_i}^+\to B(\cH)$ be the representation defined by
 $\pi_{{\bf L}_i}(\alpha):={\bf L}_{i,\alpha}$ for $\alpha \in \FF_{n_i}^+$, and let
 $\pi_{\bf L}:\FF_{n_1}^+\times \cdots \times \FF_{n_k}^+\to B(\cH)$ be the direct product representation
 defined by $
\sigma(\alpha_1,\ldots, \alpha_k)=\pi_{{\bf L}_1}(\alpha_1)\cdots \pi_{{\bf L}_k}(\alpha_k)$ for  $ (\alpha_1,\ldots, \alpha_k)\in \FF_{n_1}^+\times \cdots \times \FF_{n_k}^+.
$

A consequence  of Theorem \ref{pure-contr} is the following analogue of Rota's model theorem for  noncommutative polyballs.

\begin{corollary}\label{Rota-polyball} Let $\pi_i:\FF_{n_i}^+\to B(\cH)$, $i\in \{1,\ldots,k\}$, be representations with commuting ranges and let $\sigma: \FF_{n_1}^+\times \cdots \times \FF_{n_k}^+\to \cH$ be their direct product representation.
If
$$
\sum_{\alpha\in \FF_{n_1}^+\times\cdots \times \FF_{n_k}^+} \sigma(\alpha)\sigma(\alpha)^*\leq bI,
$$
for some contant $b>0$, then  then there exists an invariant subspace $\cG\subset F^2(H_{n_1})\otimes\cdots\otimes F^2(H_{n_k})\otimes \cH$ under each operator ${\bf L}_{i,j}\otimes I_\cH$, and an invertible operator $Y:\cH\to \cG$ such that
$$
\sigma(\alpha)^*=Y^{-1} \left[(\pi_{\bf L}(\alpha)^*\otimes I_\cH)|_{\cG}\right] Y,\qquad \alpha\in \FF_{n_1}^+\times\cdots \times \FF_{n_k}^+,
$$
and
$$\|Y^{-1}\|\|Y\|\leq \prod_{i=1}^k \left(\sum_{\alpha_i\in \FF_{n_i}}\|\pi_i(\alpha_i)\|^2\right)^{1/2}.
$$
\end{corollary}

A simple consequence of Corollary \ref{Rota-polyball} is the following von Neumann type inequality.  For $i\in \{1,\ldots, k\}$, let
$T_i:=(T_{i,1}, \ldots, T_{i,n_i})$ be such that $\|T_i\|\leq r<1$ and the entries of $T_i$ commute with those of $T_j$ for any $i\neq j$ in $\{1,\ldots, k\}$. Then
$$
\|[q_{s,t}(T_{i,j})]_{m\times m}\|\leq \frac{1}{(1-r^2)^{k/2}}  \|[q_{s,t}({\bf L}_{i,j})]_{m\times m}\|
$$
for any  matrix $[q_{s,t}]_{m\times m}$ of polynomials   in  variables $\{Z_{i,j}\}$ and any $m\in \NN$.

Another consequence  of Corollary \ref{Rota-polyball}  is  the following analogue of Rota's model theorem for the polydisc.

\begin{corollary} \label{rota-poly}Let $(C_1,\ldots, C_k)\in B(\cH)^k$ be a commuting  tuple of operators and let $S_1,\ldots, S_k$ be the unilateral shifts on the Hardy space $H^2(\DD^k)$ of the polydisc. If there is $b>0$ such that
$$
\sum_{(s_1,\ldots, s_k)\in \ZZ_+^k} C_1^{s_1}\cdots C_k^{s_k} (C_k^{s_k})^*\cdots (C_1^{s_1})^*\leq bI,
$$
then there exists an invariant subspace $\cG\subset H^2(\DD^k)\otimes \cH$ under each operator $S_i\otimes I_\cH$, and an invertible operator $Y:\cH\to \cG$ such that
$$
C_i^*=Y^{-1}[(S_i^*\otimes I_\cH)|_\cG]Y,\qquad i\in \{1,\ldots, k\}.
$$
Moreover,
$$
\|[q_{s,t}(C_1,\ldots, C_k)]_{m\times m}\|\leq \sqrt{b} \sup_{|z_i| \leq 1}\|[q_{s,t}(z_1,\ldots, z_k)]_{m\times m}\|
$$
for any  matrix $[q_{s,t}]_{m\times m}$ of polynomials   in $k$ variables and any $m\in \NN$.
\end{corollary}

\begin{corollary} \label{rota-poly2} Let $(C_1,\ldots, C_k)\in B(\cH)^k$ be a commuting  tuple of operators such that the spectral radius $r(C_i)<1$ for each $i\in \{1,\ldots, k\}$. Then the conclusion of Corollary \ref{rota-poly} holds  with $$b=\prod_{i=1}^k \left(\sum_{s_i=0}^\infty \|C_i^{s_i}\|^2\right).
$$
\end{corollary}
We remark if $(C_1,\ldots, C_k)\in B(\cH)^k$ is any commuting  tuple of operators with $\|C_i\|\leq r<1$ for $i\in \{1,\ldots, k\}$, then Corollary
\ref{rota-poly2} implies the inequality
$$
\|[q_{s,t}(C_1,\ldots, C_k)]_{m\times m}\|\leq \frac{1}{(1-r^2)^{k/2}} \sup_{|z_i| \leq 1}\|[q_{s,t}(z_1,\ldots, z_k)]_{m\times m}\|
$$
for any  matrix $[q_{s,t}]_{m\times m}$ of polynomials   in $k$ variables and any $m\in \NN$.

Another consequence of Theorem \ref{pure-contr} is
 the following analogue of Foia\c s \cite{Fo} (see also \cite{SzFBK-book}) and de Branges--Rovnyak \cite{BR}
  model theorem for pure tuples of operators in $\cV_\cQ(\cH)$.
\begin{corollary} \label{Fo-BR}
A tuple  ${\bf T}:=({ T}_1,\ldots, {T}_k)\in  B(\cH)^{n_1}\times_c\cdots \times_c B(\cH)^{n_k}$ with ${ T}_i:=(T_{i,1},\ldots, T_{i,n_i})\in B(\cH)^{n_i}$  is in the noncommutative
 variety $ \cV_\cQ(\cH)$ and it is pure if and only if
there exists a unitary  operator $U:
\cH\to \cG$ such that
$$
T_{i,j}^*=U^*[({\bf S}_{i,j}^*\otimes I_\cD)|_\cG]U
$$
for all  $i\in \{1,\ldots,k\}$ and  $j\in \{1,\ldots, n_i\}$,
where
 $\cD:=\overline{{\bf \Delta_{f,T}^m}(I)^{1/2}(\cH)}$, the subspace
$\cG\subseteq \cN_\cQ\otimes \cH$ is  invariant   under
each operator ${\bf S}_{i,j}^*\otimes
 I_\cD$,  and  ${\bf S}:=({\bf S}_1,\ldots, {\bf S}_k)$, with ${\bf S}_i:=({\bf S}_{i,1}\ldots, {\bf S}_{i,n_i})$, is the universal model associated with the abstract
noncommutative variety $\cV_Q$.
\end{corollary}
\begin{proof}  A
closer look at the proof of Theorem \ref{pure-contr}, when ${\bf A}={\bf T}$ and
$Q=I_\cH$, reveals that
$$
{\bf K}_\omega^*\,{\bf K}_\omega=
\sum_{(s_1,\ldots,s_k)\in \ZZ_+^k}\left(\begin{matrix} s_1+m_1-1\\m_1-1\end{matrix}\right)\cdots \left(\begin{matrix} s_k+m_k-1\\m_k-1\end{matrix}\right)\Phi_{f_1,A_1}^{s_1}\circ \cdots \circ \Phi_{f_k, A_k}^{s_k}(R)=I,
$$
where
$\omega:=({\bf f,m, T},R,\cQ)$ and $R:={\bf \Delta_{f,T}^m}(I)$. Consequently, ${\bf K}_\omega$ is an isometry
and   the operator   $U:\cH\to   {{\bf K}_\omega(\cH)}$, defined by $Uh:={\bf K}_\omega h$,
$h\in \cH$, is unitary. Now, one can use  relation \eqref{KASK}  to
complete the proof.
\end{proof}

A version of Rota's model theorem (see \cite{R}, \cite{H1}) asserts
that any operator with spectral radius less than
 one is similar to a strict contraction. In what follows we present an
 analogue of this result  in our multivariable noncommutative
 setting.

\begin{theorem} \label{simi2}
 Let $\cQ$ be a set of noncommutative polynomials in indeterminates $\{Z_{i,j}\}$ and  let
 ${\bf A}:=({ A}_1,\ldots, {A}_k)\in  B(\cH)^{n_1}\times_c\cdots \times_c B(\cH)^{n_k}$, where ${ A}_i:=(A_{i,1},\ldots, A_{i,n_i})\in B(\cH)^{n_i}$ has the property that $\sum_{\alpha\in \FF_{n_i}^+} a_{i,\alpha} A_{i,\alpha}A_{i,\alpha}^*$
     is  weakly convergent and $q({\bf A})=0$ for any $ q\in \cQ$. If ${\bf m}\in \ZZ_+^k$, then the following
statements are equivalent.
\begin{enumerate}
\item[(i)]
  There  is a tuple  ${\bf T}:=({ T}_1,\ldots, {T}_k)\in  B(\cH)^{n_1}\times_c\cdots \times_c B(\cH)^{n_k}$, with ${ T}_i:=(T_{i,1},\ldots, T_{i,n_i})\in B(\cH)^{n_i}$,  in the noncommutative
 variety $ \cV_\cQ(\cH)$ such that
 ${\bf \Delta_{f,T}^m}(I)>0$
 and an invertible operator $Y\in B(\cH)$ such that
$$
A_{i,j}=Y^{-1}{T}_{i,j}Y
$$
for all  $i\in \{1,\ldots,k\}$ and  $j\in \{1,\ldots, n_i\}$.
 \item[(ii)] There exists a positive operator $Q\in B(\cH)$
  such that ${\bf \Delta_{f,A}^p}(Q)\geq 0$ for any ${\bf p}\in \ZZ_+^k$ with ${\bf p}\leq {\bf m}$, and
 $${\bf \Delta_{f,A}^m}(Q)>0.
 $$
\item[(iii)] $r_{f_i}(A_{i,1},\ldots,A_{i,n_i})<1$ for each $i\in \{1,\ldots, k\}$.
  \item[(iv)] $\lim\limits_{s\to \infty} \|\Phi_{f_i,A_i}^s(I)\|=0$ for each $i\in \{1,\ldots, k\}$.
\item[(v)] For each $i\in \{1,\ldots, k\}$, the completely positive map $\Phi_{f_i,A_i}$ is power bounded and pure, and   there is an
 invertible positive operator $R\in B(\cH)$, such that
  the equation
  \begin{equation*}
  {\bf \Delta_{f,A}^m}(X)=R
  \end{equation*}
  has a positive  solution  $X$ in $B(\cH)$ such that ${\bf \Delta_{f,A}^p}(X)\geq 0$ for any ${\bf p}\in \ZZ_+^k$ with ${\bf p}\leq {\bf m}$.
\end{enumerate}
 Moreover, in this case, for any
 invertible positive operator $R\in B(\cH)$,
  the equation
  $
  {\bf \Delta_{f,A}^m}(X)=R
  $
  has a
  unique  positive solution, namely,
  $$
  X:=\sum_{(s_1,\ldots,s_k)\in \ZZ_+^k}\left(\begin{matrix} s_1+m_1-1\\m_1-1\end{matrix}\right)\cdots \left(\begin{matrix} s_k+m_k-1\\m_k-1\end{matrix}\right)\Phi_{f_1,A_1}^{s_1}\circ \cdots \circ \Phi_{f_k, A_k}^{s_k}(R),
  $$
  where the convergence is in the uniform topology, which is an invertible operator.
\end{theorem}
\begin{proof} First we prove the equivalence  (i) $\Leftrightarrow$ (ii).
Assume that (i) holds and   ${\bf \Delta_{f,T}^m}(I)\geq cI$ for some $c>0$. Then we have
$$
Y\left[{\bf \Delta_{f,A}^m}(Y^{-1} (Y^{-1})^*)\right] Y^*={\bf \Delta_{f,T}^m}(I)\geq cI.
$$
Setting $Q:=Y^{-1} (Y^{-1})^*$ we deduce that ${\bf \Delta_{f,A}^m}(Q)>0$.
 Since ${\bf T}\in \cV_\cQ(\cH)$, we have $${\bf \Delta_{f,A}^p}(Q)=Y^{-1}{\bf \Delta_{f,T}^p}(I)(Y^{-1})^*\geq 0$$ for any ${\bf p}\leq {\bf m}$.
Conversely, assume that  item (ii) holds and let $Q\in B(\cH)$ be  a positive
operator  such that${\bf \Delta_{f,A}^p}(Q)\geq 0$ for any ${\bf p}\in \ZZ_+^k$ with ${\bf p}\leq {\bf m}$, and
 ${\bf \Delta_{f,A}^m}(Q)>0$. Since  $\Phi_{f_i,A_i}$ is a positive linear map, we deduce that, for each $i\in \{1,\ldots, k\}$,
$$
0<{\bf \Delta_{f,A}^m}(Q)\leq (id-\Phi_{f_i,A_i})^{m_i}(Q)\leq \cdots \leq (id-\Phi_{f_i,A_i})(Q)\leq Q.
$$
Therefore, $Q$ is an invertible positive operator.
 Since
 $
{\bf \Delta_{f,A}^m}(Q)\geq b I
$
 for some constant $b>0$, we can choose $c>0$ such that $bI\geq cQ$, and  deduce that
 $$Q^{-1/2}[{\bf \Delta_{f,A}^m}(Q)]Q^{-1/2}\geq cI.$$
 Setting $T_i:=Q^{-1/2} A_iQ^{1/2}$, $i=1,\ldots,n$, the latter inequality implies
 ${\bf \Delta_{f,T}^m}(I)>0$.   As above,  we deduce that ${\bf \Delta_{f,T}^p}(I)\geq 0$,
 for any ${\bf p}\leq {\bf m}$, which shows that ${\bf T}\in {\bf D_f^m}(\cH)$.
 Since  $q({\bf A})=0$,\quad $q\in \cQ$, we deduce that ${\bf T} \in \cV_\cQ(\cH)$.
Therefore, item (i) holds.

Now we prove the  equivalence (iii) $\Leftrightarrow$ (iv).
Assume that item  (iii) holds and let $a>0$ be such that  $ r_{f_i}( A_i)< a<1$. Then there is $m_0\in \NN$ such that $\|\Phi_{f_i,A_i}^s(I)\|\leq
a^s$ for any $s\geq m_0$. This clearly implies condition (iv).  Now, we  assume that (iv) holds.
 Note that, for each $i\in \{1,\ldots, k\}$ and $s\in\NN$, we have
\begin{equation*}
\begin{split}
r_{f_i}( A_i)^s&=\lim_{p\to
\infty}\left[\|\Phi_{f_i,A_i}^{sp}(I)\|^{1/2ps}\right]^s\\
&=\lim_{p\to \infty}
\|\Phi_{f_i,A_i}^{s(p-1)}(\Phi_{f_i,A_i}^s(I))\|^{1/2p}\\
&\leq \lim_{p\to
\infty}\left(\|\Phi_{f_i,A_i}^s(I)\|^p\right)^{1/2p}
=\|\Phi_{f_i,A_i}^s(I)\|^{1/2}< a^{s/2}
\end{split}
\end{equation*}
 for any $s\in \NN$. Consequently,  $r_{f_i}( A_i)<1$, so item (iii) holds.
The implication (v) $\Rightarrow$ (ii) is obvious.

 In what follows
we prove that
  (i) $\Rightarrow$ (iii).
  Assume that  item (i) holds.
    Let ${\bf T}:=({ T}_1,\ldots, {T}_k)\in  B(\cH)^{n_1}\times_c\cdots \times_c B(\cH)^{n_k}$, with ${ T}_i:=(T_{i,1},\ldots, T_{i,n_i})\in B(\cH)^{n_i}$,  be in  the noncommutative
 variety $ \cV_\cQ(\cH)$ such that
 ${\bf \Delta_{f,T}^m}(I)>0$
 and let $Y\in B(\cH)$ be  an invertible operator such that
$$
A_{i,j}=Y^{-1}{T}_{i,j}Y
$$
for all  $i\in \{1,\ldots,k\}$ and  $j\in \{1,\ldots, n_i\}$.
Recall that under these conditions we have $(id-\Phi_{f_i,T_i})(I)>0$, which implies    $\|\Phi_{f_i,T_i}(I)\|<1$ for $i\in \{1,\ldots, k\}$.
  On the other hand, note that
\begin{equation*}
\begin{split} r_{f_i}(T_i)&=
 \lim_{s\to\infty}\|\Phi_{f_i,T_i}^s(I)\|^{1/2s}\\
&\leq \lim_{s\to\infty}\|Y\|^{1/s}\|Y^{-1}\|^{1/s}\|\Phi_{f_i,A_i}^s(I)\|^{1/2s}\\
&=r_{f_i}(A_i).
\end{split}
\end{equation*}
Similarly, we obtain the inequality
   $
r_{f_i}(A_i)\leq r_{f_i}(T_i).
$
Therefore, we have
\begin{equation*}
\begin{split}
r_{f_i}(A_i)=r_{f_i}(T_i) =\lim_{s\to\infty} \|\Phi_{f_i,T_i}^s(I)\|^{1/2s}\leq
\|\Phi_{f_i,T_i}(I)\|^{1/2}<1.
\end{split}
\end{equation*}
Therefore,  item (iii) holds. Now,  we prove the implication  (iii)
  $\Rightarrow$ (v). To this end,  assume that $r_{f_i}(A_i)<1$ for each $i\in \{1,\ldots, k\}$ and let $R\in B(\cH)$ be  an invertible positive operator.
   We have
    \begin{equation*}
    \begin{split}
    \frac {1} {\|R^{-1}\|} \, I &\leq R
    \leq \sum_{(s_1,\ldots,s_k)\in \ZZ_+^k}\left(\begin{matrix} s_1+m_1-1\\m_1-1\end{matrix}\right)\cdots \left(\begin{matrix} s_k+m_k-1\\m_k-1\end{matrix}\right)\Phi_{f_1,A_1}^{s_1}\circ \cdots \circ \Phi_{f_k, A_k}^{s_k}(R)\\
   & \leq \|R\| \left( \sum_{s_1=0}^\infty \left(\begin{matrix} s_1+m_1-1\\ m_1-1
\end{matrix}\right)\|\Phi_{f_1,A_1}^{s_1}(I)\|\right)\cdots \left( \sum_{s_k=0}^\infty \left(\begin{matrix} s_k+m_k-1\\ m_k-1
\end{matrix}\right)\|\Phi_{f_k,A_k}^{s_k}(I)\|\right)\, I.
    \end{split}
    \end{equation*}
 Note that
 $$
 \lim_{s_i\to\infty} \left[\left(\begin{matrix} s_i+m_i-1\\ m_i-1
\end{matrix}\right)\|\Phi_{f_i,A_i}^{s_i}(I)\|\right]^{1/2s_i}
=r_{f_i}(T_i)<1.
$$
Consequently,
   \begin{equation}\label{ab}
 aI \leq \sum_{(s_1,\ldots,s_k)\in \ZZ_+^k}\left(\begin{matrix} s_1+m_1-1\\m_1-1\end{matrix}\right)\cdots \left(\begin{matrix} s_k+m_k-1\\m_k-1\end{matrix}\right)\Phi_{f_1,A_1}^{s_1}\circ \cdots \circ \Phi_{f_k, A_k}^{s_k}(R) \leq bI
 \end{equation}
 for some  constants  $0<a<b$, where the convergence of the series is in the operator norm topology.
Since $r_{f_i}(A_i)<1$, we have $ \lim_{s\to \infty} \|\Phi_{f_i,A_i}^s(I)\|=0$. Therefore, $\Phi_{f_i,A_i}$ is  a power bounded,  pure completely positive map which is WOT-continuous on bounded sets. Now,   we can use Theorem  \ref{reproducing}
 to obtain the equality
 $$
 {\bf \Delta_{f,A}^m}\left[\sum_{(s_1,\ldots,s_k)\in \ZZ_+^k}\left(\begin{matrix} s_1+m_1-1\\m_1-1\end{matrix}\right)\cdots \left(\begin{matrix} s_k+m_k-1\\m_k-1\end{matrix}\right)\Phi_{f_1,A_1}^{s_1}\circ \cdots \circ \Phi_{f_k, A_k}^{s_k}(R)\right]=R.
 $$
 Consequently, and due to  relation \eqref{ab},
  $$
  X:=\sum_{(s_1,\ldots,s_k)\in \ZZ_+^k}\left(\begin{matrix} s_1+m_1-1\\m_1-1\end{matrix}\right)\cdots \left(\begin{matrix} s_k+m_k-1\\m_k-1\end{matrix}\right)\Phi_{f_1,A_1}^{s_1}\circ \cdots \circ \Phi_{f_k, A_k}^{s_k}(R)
  $$
  is an invertible positive solution of the equation  ${\bf \Delta_{f,A}^m}(X)=R$. Since ${\bf \Delta_{f,A}^m}(X)\geq 0$ and $\Phi_{f_i,A_i}$ is pure, we use  Proposition \ref{Delta-ineq} part (ii) to  deduce that  ${\bf \Delta_{f,A}^p}(X)\geq 0$ for any ${\bf p}\in \ZZ_+^k$ with ${\bf p}\leq {\bf m}$.
 Therefore, item (v) holds.

   To prove the last part of the theorem,
   let $X'\geq 0$ be an invertible operator such
   ${\bf \Delta_{f,A}^m}(X')=R$, where $R\geq 0$ is a fixed   arbitrary invertible operator. Then, using again Theorem  \ref{reproducing},  we deduce that
   \begin{equation*}\begin{split}
  \sum_{(s_1,\ldots,s_k)\in \ZZ_+^k}\left(\begin{matrix} s_1+m_1-1\\m_1-1\end{matrix}\right)\cdots \left(\begin{matrix} s_k+m_k-1\\m_k-1\end{matrix}\right)\Phi_{f_1,A_1}^{s_1}\circ \cdots \circ \Phi_{f_k, A_k}^{s_k}(R)  =X'.
\end{split}
\end{equation*}
 Therefore,  there is  unique  positive solution
   of the inequality ${\bf \Delta_{f,A}^m}(X)=R$.
   The proof is complete.
\end{proof}

Now we  can obtain the following
multivariable generalization of Rota's  similarity result  (see Paulsen's book \cite{Pa-book}).

\begin{corollary} \label{rota}
Under the hypotheses of Theorem \ref{simi2}, if the joint spectral
radius $r_{f_i}(A_i)<1$ for each $i\in\{1,\ldots,k\}$, then the tuple
 ${\bf T}:=({ T}_1,\ldots, {T}_k)\in  B(\cH)^{n_1}\times_c\cdots \times_c B(\cH)^{n_k}$, with
 $${ T}_i:=(P^{-1/2}A_{i,1}P^{1/2},\ldots,P^{-1/2} A_{i,n_i}P^{1/2})\in B(\cH)^{n_i},$$
 is in the noncommutative variety $\cV_\cQ(\cH)$
and  ${\bf \Delta_{f,T}^m}(I)>0$, where
$$P:=\sum_{(s_1,\ldots,s_k)\in \ZZ_+^k}\left(\begin{matrix} s_1+m_1-1\\m_1-1\end{matrix}\right)\cdots \left(\begin{matrix} s_k+m_k-1\\m_k-1\end{matrix}\right)\Phi_{f_1,A_1}^{s_1}\circ \cdots \circ \Phi_{f_k, A_k}^{s_k}(I)$$
 is convergent
in the operator  norm topology  and
 $$
 \|P^{1/2}\|^2\|P^{-1/2}\|^2\leq \sum_{(s_1,\ldots,s_k)\in \ZZ_+^k}\left(\begin{matrix} s_1+m_1-1\\m_1-1\end{matrix}\right)\cdots \left(\begin{matrix} s_k+m_k-1\\m_k-1\end{matrix}\right)\|\Phi_{f_1,A_1}^{s_1}(I)\| \cdots  \|\Phi_{f_k, A_k}^{s_k}(I)\|.
$$
In particular, if  each $f_i$ is a positive regular noncommutative
polynomial, then $P$ is in the $C^*$-algebra generated by
$A_{i,j}$ and the identity.
\end{corollary}
\begin{proof}   A closer look at the proof of  Theorem
\ref{simi2} and taking  $R=I$ leads to the desired result. The last
part of this corollary is now obvious.
\end{proof}

We say that  $\pi_i:\FF_{n_i}^+\to B(\cH)$ is  a strictly row contractive  representation if its generators form a strict row contraction, i.e.
$\|[\pi_i(g_1^i)\cdots \pi_i(g_{n_i}^i)]\|<1$. We  denote
$$
r(\pi_i):=r( \pi_i(g_1^i),\ldots, \pi_i(g_{n_i}^i))
$$
and call it the spectral radius of  $\pi_i$.

\begin{corollary} \label{strictly} Let $\pi_i:\FF_{n_i}^+\to B(\cH)$, $i\in \{1,\ldots,k\}$, be representations with commuting ranges and let $\sigma: \FF_{n_1}^+\times \cdots \times \FF_{n_k}^+\to \cH$ be the direct product representation defined by
$$
\sigma(\alpha_1,\ldots, \alpha_k)=\pi_1(\alpha_1)\cdots \pi_k(\alpha_k),\qquad (\alpha_1,\ldots, \alpha_k)\in \FF_{n_1}^+\times \cdots \times \FF_{n_k}^+.
$$
Then the following statements are equivalent:
\begin{enumerate}
\item[(i)] There is an invertible operator $Y \in B(\cH)$ such that $Y^{-1}\sigma(\cdot) Y$ is the direct product of strictly row contractive representations, i.e. $Y^{-1}\pi_i(\cdot) Y$ is a  strictly row contractive representation for each $i\in \{1,\ldots, k\}$.

\item[(ii)] $r(\pi_i)<1$  for each $i\in \{1,\ldots, k\}$.
\end{enumerate}
\end{corollary}

In the  particular case when $n_1=\cdots n_k=1$, Corollary \ref{strictly} shows that
 a $k$-tuple of commuting operators $(C_1,\ldots, C_k)\in B(\cH)^k$ is jointly similar to a $k$-tuple of commuting strict contractions $(G_1,\ldots, G_k)\in B(\cH)$    if and only if
$$ r(C_i)<1,\qquad  i\in \{1,\ldots,k\},
$$
where $r(C_i)$ denotes the spectral radius of $C_i$.

The next result provides necessary and sufficient conditions for  tuples of operators to
 be  similar
 to  a tuple in the noncommutative variety $\cV_\cQ(\cH)$. Since the proof is straightforward, we shall omit it.

\begin{proposition} \label{simi4} Let
 ${\bf A}:=({ A}_1,\ldots, {A}_k)\in  B(\cH)^{n_1}\times_c\cdots \times_c B(\cH)^{n_k}$, where ${ A}_i:=(A_{i,1},\ldots, A_{i,n_i})\in B(\cH)^{n_i}$ has the property that $\sum_{\alpha_i\in \FF_{n_i}^+} a_{i,\alpha} A_{i,\alpha}A_{i,\alpha}^*$
     is  weakly convergent and $q({\bf A})=0$ for any $ q\in \cQ$. Then the following
statements are equivalent.
\begin{enumerate}
\item[(i)]
  There  is a tuple  ${\bf T}:=({ T}_1,\ldots, {T}_k)\in  B(\cH)^{n_1}\times_c\cdots \times_c B(\cH)^{n_k}$, with ${ T}_i:=(T_{i,1},\ldots, T_{i,n_i})\in B(\cH)^{n_i}$,  in the noncommutative
 variety $ \cV_\cQ(\cH)$,
   and an invertible operator $Y\in B(\cH)$ such that
$$
A_{i,j}=Y^{-1}{T}_{i,j}Y
$$
for all  $i\in \{1,\ldots,k\}$ and  $j\in \{1,\ldots, n_i\}$.
\item[(ii)]  There is an
 invertible positive operator $R\in B(\cH)$, such that
 $$
  {\bf \Delta_{f,A}^p}(R)\geq 0
  $$
  for any ${\bf p}\in \ZZ_+$ with ${\bf p}\leq {\bf m}$.
\end{enumerate}
\end{proposition}

\bigskip

\section{Analogue of  Sz.-Nagy's similarity result
 for noncommutative  polydomains}

Let ${\bf f}:=(f_1,\ldots, f_k)$ be a $k$-tuple of positive regular free holomorphic functions and let ${\bf m}=(m_1,\ldots, m_k)$ be in $ \NN^k$.
Consider ${\bf A}:=({ A}_1,\ldots, {A}_k)\in  B(\cH)^{n_1}\times\cdots \times B(\cH)^{n_k}$, where ${ A}_i:=(A_{i,1},\ldots, A_{i,n_i})\in B(\cH)^{n_i}$,  to be such that $\Phi_{f_i, A_i}(I)$ is well-defined in the weak operator topology, and let $\cQ$ be a set of noncommutative polynomials in indeterminates $\{Z_{i,j}\}$ with $i\in \{1,\ldots, k\}$ and $j\in \{1,\ldots, n_i\}$.
In this section we provide necessary and sufficient conditions for a tuple ${\bf A}=({ A}_1,\ldots, {A}_k)$ to
be jointly similar  to a  tuple   ${\bf T}:=({ T}_1,\ldots, {T}_k)\in  B(\cH)^{n_1}\times\cdots \times B(\cH)^{n_k}$
satisfying  the property

$${\bf T}\in  \left\{{\bf  X} \in {\cV}_{\cQ}(\cH): \
 {\bf \Delta_{f,X}^p}(I)= 0  \text{ \rm  for } 0\leq {\bf p}\leq {\bf m}, {\bf p}\neq 0\right\}$$

  We show that this similarity  is strongly related to the
existence of invertible positive solutions of the operator
 equation  $
{\bf \Delta_{f,A}^p}(Y)= 0$.  Here is our analogue of  Sz.-Nagy's similarity result
\cite{SzN} for noncommutative  polydomains.

\begin{theorem}\label{simi} Let
 ${\bf A}:=({ A}_1,\ldots, {A}_k)\in  B(\cH)^{n_1}\times_c\cdots \times_c B(\cH)^{n_k}$, where ${ A}_i:=(A_{i,1},\ldots, A_{i,n_i})\in B(\cH)^{n_i}$ has the property that $\sum_{\alpha\in \FF_{n_i}^+} a_{i,\alpha} A_{i,\alpha}A_{i,\alpha}^*$
     is  weakly convergent and $q({\bf A})=0$ for any $ q\in \cQ$.  Then the following
statements are equivalent.
\begin{enumerate}
 \item[(i)] There  is ${\bf T}:=({ T}_1,\ldots, {T}_k)\in  B(\cH)^{n_1}\times_c\cdots \times_c B(\cH)^{n_k}$, with ${ T}_i:=(T_{i,1},\ldots, T_{i,n_i})\in B(\cH)^{n_i}$,  in the noncommutative
 variety $ \cV_\cQ(\cH)$, such that
 $$\Phi_{f_i, T_i}(I)=I, \quad i\in \{1,\ldots, k\},
 $$
   and an invertible operator $Y\in B(\cH)$ such that
$$
A_{i,j}=Y^{-1}{T}_{i,j}Y
$$
for all  $i\in \{1,\ldots,k\}$ and  $j\in \{1,\ldots, n_i\}$.
\item[(ii)]
 There
 exist
 positive constants
 $0<c\leq d$ such that
 $$
 cI\leq  \Phi_{f_1,A_1}^{s_1}\circ \cdots \circ \Phi_{f_k,A_k}^{s_k} (I)\leq dI, \qquad  s_1,\ldots, s_k\in \ZZ_+.
 $$
 \item[(iii)] There
 exist
 positive constants
 $0<c\leq d$ such that
 $$
 cI\leq \frac {1} {p^{(1)}\cdots p^{(k)}} \sum_{s_k=0}^{p^{(k)}-1}\cdots \sum_{s_1=0}^{p^{(1)}-1} \Phi_{f_k,A_k}^{s_k}\circ \cdots \circ \Phi_{f_1,A_1}^{s_1}  (I)\leq dI
 $$
 for any $p^{(1)},\ldots, p^{(k)}\in \NN$.
 \item[(iv)] There is a positive invertible operator $Q\in B(\cH)$ such that  $\Phi_{f_i, A_i}(Q)=Q$ for any $i\in \{1,\ldots,k\}$. Moreover, the operator $Q$ can be chosen  in the von Neumann algebra generated by $\{A_{i,j}\}$  and the identity such that $cI\leq Q\leq dI$.
\end{enumerate}
\end{theorem}

\begin{proof}
We prove that (i)$\implies$ (ii). Assume that item (i) holds. Then we have
\begin{equation*}
\begin{split}
\Phi_{f_1,A_1}^{s_1}\circ \cdots \circ \Phi_{f_k,A_k}^{s_k} (I)
&=Y^{-1}\left[\Phi_{f_1,T_1}^{s_1}\circ \cdots \circ \Phi_{f_k,T_k}^{s_k} (YY^*)\right]{Y^*}^{-1}\\
&\leq \|YY^*\|Y^{-1}\left[\Phi_{f_1,T_1}^{s_1}\circ \cdots \circ \Phi_{f_k,T_k}^{s_k} (I)\right]{Y^*}^{-1}\\
&\leq \|Y\|^2 \|Y^{-1}\|^2 I.
\end{split}
\end{equation*}
On the other hand, since $\Phi_{f_i, T_i}(I)=I$ for $i\in \{1,\ldots, k\}$, we deduce that
\begin{equation*}
\begin{split}
I&=\Phi_{f_1,T_1}^{s_1}\circ \cdots \circ \Phi_{f_k,T_k}^{s_k} (I)
=Y\left[\Phi_{f_1,A_1}^{s_1}\circ \cdots \circ \Phi_{f_k,A_k}^{s_k}(Y^{-1} {Y^*}^{-1})\right]Y^*\\
&\leq \|Y^{-1}{Y^*}^{-1}\|Y\left[\Phi_{f_1,A_1}^{s_1}\circ \cdots \circ \Phi_{f_k,A_k}^{s_k}(I)\right]Y^*.
\end{split}
\end{equation*}
Hence, we have
$$
Y^{-1} {Y^*}^{-1}\leq  \|Y^{-1}{Y^*}^{-1}\| \Phi_{f_1,A_1}^{s_1}\circ \cdots \circ \Phi_{f_k,A_k}^{s_k}(I)
$$
which implies
$$
\Phi_{f_1,A_1}^{s_1}\circ \cdots \circ \Phi_{f_k,A_k}^{s_k}(I)\geq \frac{1}{\|Y^{-1}\|} Y^{-1} {Y^*}^{-1}\geq \frac{1}{\|Y\|^2 \|Y^{-1}\|} I.
$$
Note that the implication (ii)$\implies$(iii) is obvious.
Now, we prove that (iii)$\implies$ (iv). Assume that item (iii) holds.
For each $(p^{(1)},\ldots, p^{(k)})\in \NN^k$, we define the operator
$$
Q_{(p^{(1)},\ldots, p^{(k)})}:=  \frac {1} {p^{(1)}\cdots p^{(k)}} \sum_{s_k=0}^{p^{(k)}-1}\cdots \sum_{s_1=0}^{p^{(1)}-1} \Phi_{f_k,A_k}^{s_k}\circ \cdots \circ \Phi_{f_1,A_1}^{s_1}  (I).
$$
 In what follows, we show that there are subsequences  $\{p_{j_1}^{(1)}\}_{j_1=1}^\infty,\ldots, \{p_{j_k}^{(k)}\}_{j_k=1}^\infty$ such that
$$
Q:=\lim_{j_k\to\infty}\ldots \lim_{j_1\to\infty} Q_{(p^{(1)}_{j_1},\ldots, p^{(k)}_{j_k})}
$$
exists, where the limits are taken in the weak operator topology, and $Q$ is a positive invertible operator with the property that $\Phi_{f_i, A_i}(Q)=Q$ for any $i\in \{1,\ldots, k\}$.

Define  the sequence of operators $\{Q_{p^{(1)}, A_1}\}_{p^{(1)}=1}^\infty$ by setting
$$
Q_{p^{(1)}, A_1}:=\frac{1}{p^{(1)}}\sum_{s_1=0}^{p^{(1)}-1} \Phi_{f_1,A_1}^{s_1}(I).
$$
Note that $cI\leq Q_{p^{(1)}, A_1}\leq dI$ for any $p^{(1)}\in \NN$.
Since the closed unit ball of $B(\cH)$ is weakly compact, there is a subsequence
  $\{Q_{p^{(1)}_{j_1}, A_1}\}_{j_1=1}^\infty $
 weakly
 convergent   to an operator $Q_{A_1}\in B(\cH)$.
 It is clear that $Q_{A_1}$ is an invertible positive
 operator and
 $cI\leq Q_{A_1}\leq dI$.

 Let $P\in B(\cH)$ be an invertible positive operator with the property that
 \begin{equation*}
 %\label{P-ineq}
 \frac{1}{j+1} \sum_{s_1=0}^j \Phi_{f_1, A_1}^{s_1}(P)\leq bI,\qquad j\in \ZZ_+.
 \end{equation*}
 Note that this inequality is satisfied when $P=I$. Using the fact that $\Phi_{f_1, A_1}$ is a positive linear map, for any $t\in\ZZ_+$, we have
 \begin{equation*}
 \begin{split}
 \Phi_{f_1, A_1}^t(P)\left(\sum_{j=0}^t \frac{1}{j+1}\right)
 &\leq \sum_{j=0}^t \frac{1}{j+1} \|\Phi_{f_1, A_1}^j(P)\|\Phi_{f_1, A_1}^{t-j}(I)\\
 &\leq b\sum_{j=0}^t \Phi_{f_1, A_1}^{t-j}(I)= b\sum_{j=0}^t \Phi_{f_1, A_1}^{j}(I)\\
 &\leq b\|P^{-1}\|\sum_{j=0}^t  \Phi_{f_1, A_1}^j(P)\leq b^2(t+1)I.
 \end{split}
 \end{equation*}
 Hence, we deduce that
 $$
 \frac{1}{t}  \Phi_{f_1, A_1}^t(P) \leq \frac{b^2\frac{t+1}{t}}{\sum_{j=0}^t \frac{1}{j+1}}I,\qquad t\in \NN,
 $$
 which implies $\frac{1}{t}  \Phi_{f_1, A_1}^t(P)\to 0$ in norm as $t\to\infty$. In particular, this convergence holds when $P=I$.

 On the other hand,
 since
 $$
 Q_{p^{(1)}_{j_1}, A_1}-\Phi_{f_1, A_1}(Q_{p^{(1)}_{j_1}, A_1})=\frac{1}{p_{j_1}^{(1)}}I-\frac{1}{p_{j_1}^{(1)}}\Phi_{f_1, A_1}^{p_{j_1}^{(1)}}(I)
 $$
 and $\frac{1}{p_{j_1}^{(1)}}\Phi_{f_1, A_1}^{p_{j_1}^{(1)}}(I)\to 0$ in norm as $j_1\to\infty$, we deduce that $Q_{p^{(1)}_{j_1}, A_1}-\Phi_{f_1, A_1}(Q_{p^{(1)}_{j_1}, A_1})\to 0$ in norm as $j_1\to\infty$.
 Since $\Phi_{f_1, A_1}$ is weakly continuous on bounded sets and
 $Q_{p^{(1)}_{j_1}, A_1}\to Q_{A_1}$ weakly, we deduce that
 $$
 \Phi_{f_1, A_1}(Q_{A_1})=Q_{A_1}.
 $$
Define  the sequence of operators $\{Q_{p^{(2)}, A_2}\}_{p^{(2)}=1}^\infty$ by setting
$$
Q_{p^{(2)}, A_2}:=\frac{1}{p^{(2)}}\sum_{s_2=0}^{p^{(2)}-1} \Phi_{f_2,A_2}^{s_2}(Q_{A_1}).
$$
Note that $cI\leq Q_{p^{(2)}, A_2}\leq dI$ for any $p^{(2)}\in \NN$.
As above, one can prove that $\frac{1}{t}  \Phi_{f_2, A_2}^t(Q_{A_1})\to 0$ in norm as $t\to\infty$.
Since the closed unit ball of $B(\cH)$ is weakly compact, there is a subsequence
  $\{Q_{p^{(2)}_{j_2}, A_2}\}_{j_2=1}^\infty $
 weakly
 convergent   to an operator $Q_{A_2,A_1}\in B(\cH)$.
 It is clear that $Q_{A_2,A_1}$ is an invertible positive
 operator and
 $cI\leq Q_{A_2,A_1}\leq dI$. Since
 $$
 Q_{p^{(2)}_{j_2}, A_2}-\Phi_{f_2, A_2}(Q_{p^{(2)}_{j_2}, A_2})=\frac{1}{p_{j_2}^{(2)}}Q_{A_1}-\frac{1}{p_{j_2}^{(2)}}\Phi_{f_2, A_2}^{p_{j_2}^{(2)}}(Q_{A_1})
 $$
 and $\frac{1}{p_{j_2}^{(1)}}\Phi_{f_2, A_2}^{p_{j_1}^{(1)}}(Q_{A_1})\to 0$ in norm as $j_2\to\infty$, we deduce that $Q_{p^{(2)}_{j_2}, A_2}-\Phi_{f_2, A_2}(Q_{p^{(2)}_{j_2}, A_2})\to 0$ in norm as $j_2\to\infty$.
 Since $\Phi_{f_2, A_2}$ is weakly continuous on bounded sets and
 $Q_{p^{(2)}_{j_2}, A_2}\to Q_{A_2,A_1}$ weakly, we deduce that
 $$
 \Phi_{f_2, A_2}(Q_{A_2,A_1})=Q_{A_2,A_1}.
 $$
Since $\Phi_{f_1, A_1}$ is WOT-continuous on bounded sets, $\Phi_{f_1, A_1}$ commutes with $\Phi_{f_2, A_2}$, and $\Phi_{f_1, A_1}(Q_{A_1})=Q_{A_1}$, we deduce  that
\begin{equation*}
\begin{split}
\Phi_{f_1, A_1}(Q_{A_2,A_1})&=\text{\rm WOT-}\lim_{j_2\to\infty} \Phi_{f_1, A_1}(Q_{A_2,p_{j_2}^{(2)}})\\
&=
\text{\rm WOT-}\lim_{j_2\to\infty}\left(\frac{1}{p^{(2)}_{j_2}}\sum_{s_2=0}^{p^{(2)}_{j_2}-1} \Phi_{f_2,A_2}^{s_2}(\Phi_{f_1, A_1}(Q_{A_1}))\right)\\
&=
\text{\rm WOT-}\lim_{j_2\to\infty}\left(\frac{1}{p^{(2)}_{j_2}}\sum_{s_2=0}^{p^{(2)}_{j_2}-1} \Phi_{f_2,A_2}^{s_2} (Q_{A_1})\right)\\
&=Q_{A_2, A_1}.
\end{split}
\end{equation*}

Continuing this process, we find   an invertible positive
 operator   $Q_{A_k,\ldots, A_1}$  with the property that
 $cI\leq Q_{A_k,\ldots, A_1}\leq dI$ and
 $$
 \Phi_{f_i, A_i}(Q_{A_k,\ldots, A_1})=Q_{A_k,\ldots, A_1},\qquad i\in\{1,\ldots, k\}.
 $$
Therefore, item (iv) holds. To prove that (iv)$\implies$(i) we assume that
  there is a positive invertible operator $Q\in B(\cH)$ such that  $\Phi_{f_i, A_i}(Q)=Q$ for any $i\in \{1,\ldots,k\}$.
Set $T_{i,j}:=Q^{-1/2}A_{i,j} Q^{1/2}$ for all  $i\in \{1,\ldots,k\}$ and  $j\in \{1,\ldots, n_i\}$, and note that
$$
\Phi_{f_i, T_i}(I)=Q^{-1/2}\Phi_{f_i, A_i}(Q) Q^{-1/2}= I, \qquad i\in \{1,\ldots, k\}.
$$
  The proof is complete.
\end{proof}

We say that  $\pi_i:\FF_{n_i}^+\to B(\cH)$ is row contractive (resp. coisometric, Cuntz) representation if its generators form a row contraction (resp. coisometry, unitary), i.e. the operator matrix
$[\pi_i(g_1^i)\cdots \pi_i(g_{n_i}^i)]$  is contractive (resp. coisometric, unitary) from the direct sum $\cH^{(n_i)}:=\cH\oplus\cdots \oplus \cH$ to $\cH$.

\begin{corollary} \label{iso} Let $\pi_i:\FF_{n_i}^+\to B(\cH)$, $i\in \{1,\ldots,k\}$, be representations with commuting ranges and let $\sigma: \FF_{n_1}^+\times \cdots \times \FF_{n_k}^+\to \cH$ be the direct product representation defined by
$$
\sigma(\alpha_1,\ldots, \alpha_k)=\pi_1(\alpha_1)\cdots \pi_k(\alpha_k),\qquad (\alpha_1,\ldots, \alpha_k)\in \FF_{n_1}^+\times \cdots \times \FF_{n_k}^+.
$$
Then the following statements are equivalent:
\begin{enumerate}
\item[(i)] There is an invertible operator $Y \in B(\cH)$ such that $Y^{-1}\sigma(\cdot) Y$ is the direct product of row coisometric representations, i.e. $Y^{-1}\pi_i(\cdot) Y$ is a row coisometric representation for each $i\in \{1,\ldots, k\}$.

\item[(ii)]  There exist constants $0<c\leq d$ such
$$
c\|h\|^2\leq \|\sigma(\alpha_1,\ldots, \alpha_k)h\|^2\leq d\|h\|^2,\qquad h\in \cH,
$$
for any $(\alpha_1,\ldots, \alpha_k)\in \FF_{n_1}^+\times \cdots \times \FF_{n_k}^+$.
\end{enumerate}
\end{corollary}
We should remark that another consequence of Theorem \ref{simi}  regarding the similarity to  a direct product of   Cuntz representations was mentioned in the introduction.
Note also that  in the particular case when $n_1=\cdots= n_k=1$  Corollary \ref{iso} implies  the following result for the polydisc.

\begin{corollary}\label{uniform-polydisc} A $k$-tuple of commuting operators $(C_1,\ldots, C_k)\in B(\cH)^k$ is jointly similar to a $k$-tuple of commuting isometries $(V_1,\ldots, V_k)\in B(\cH)$ if and only if there are constants $0<c\leq d$ such that
$$
c\|h\|^2\leq \|C_1^{s_1}\cdots C_k^{s_k}h\|^2\leq d\|h\|^2, \quad h\in \cH,
$$
for any $s_1,\ldots, s_k\in \ZZ^+$. Moreover, there is an  invertible operator $\xi:\cH\to \cH$  such that $V_i=\xi C_i \xi^{-1}$  for $i\in \{1,\ldots, k\}$ and $\xi$ is in the von Neumann algebra generated by $C_1,\ldots, C_n$ and the identity.
\end{corollary}

We remark that under the conditions of Corollary \ref{uniform-polydisc},
 we have  the  inequality
$$
\|[q_{s,t}(C_1,\ldots, C_k)]_{m\times m}\|\leq \sqrt{\frac{d}{c}} \sup_{|z_i| \leq 1}\|[q_{s,t}(z_1,\ldots, z_k)]_{m\times m}\|
$$
for any  matrix $[q_{s,t}]_{m\times m}$ of polynomials   in $k$ variables and any $m\in \NN$.

As a consequence of Corollary \ref{uniform-polydisc}, we deduce the well-known result (see \cite{Di}, \cite{Da}) that any uniformly bounded representation $u:\ZZ^k\to B(\cH)$ is similar to a unitary representation.
More precisely there is an invertible operator $\xi:\cH\to \cH$ such that $\xi u(\cdot) \xi^{-1}$ is a unitary representation, and $\xi$ can be chosen in the von Neumann  algebra generated by $u(\ZZ^k)$. In the particular case  when $k=1$, we recover Sz-Nagy similarity result \cite{SzN}.

\bigskip

\section{Joint similarity of positive linear maps}

In what follows,  we provide analogues  of all the similarity results presented in the previous sections in the context of joint similarity of commuting tuples  of positive linear maps on the algebra of bounded linear operators on a separable Hilbert space.

We say that a commuting $k$-tuple $\Lambda:=(\lambda_1,\ldots, \lambda_k)$ of   positive linear maps on $B(\cH)$ is pure if, for each $i\in \{1,\ldots, k\}$,   $ \lambda_i^s(I)\to 0$ weakly as $s\to \infty$.
 Let
$\Phi:=(\varphi_1,\ldots, \varphi_k)$   be another $k$-tuples of commuting  positive linear maps on $B(\cK)$.
We say that  $\Phi$ is jointly similar to $\Lambda$ if there is an invertible operator $R\in B(\cH,\cK)$ such that
$$
\varphi_i(RXR^*)=R\lambda_i(X)R^*,\qquad X\in B(\cH),
$$
for any $i\in \{1,\ldots, k\}$. This relation is equivalent to $\varphi_i=\psi_R\circ \lambda_i\circ \psi_R^{-1}$ for $i\in \{1,\ldots, k\}$, where
$\psi_R(X):=RXR^*$. Note that the relation above shows that the discrete semigroups of positive linear maps $\{\varphi_1^{p_1}\circ\cdots \circ \varphi_k^{p_k}\}_{(p_1,\ldots, p_k)\in \ZZ_+^k}$ and
 $\{\lambda_1^{p_1}\circ\cdots \circ \lambda_k^{p_k}\}_{(p_1,\ldots, p_k)\in \ZZ_+^k}$ are also similar. We also remark that
 ${\bf \Delta}_{\bf \Phi}^{\bf p}(RXR^*)=R{\bf \Delta}_{\bf \Lambda}^{\bf p}(X)R^*$ for any ${\bf p}\in \ZZ_+^k$ and $X\in B(\cH)$. Consequently,
 $D\in \cC_{\geq}({\bf \Delta}_{\Lambda}^{\bf m})^+$ if and only if
 $RDR^*\in \cC_{\geq}({\bf \Delta}_{\Phi}^{\bf m})^+$. In particular, we have
 $I\in \cC_{\geq}({\bf \Delta}_{\Lambda}^{\bf m})^+$ if and only if
 $RR^*\in \cC_{\geq}({\bf \Delta}_{\Phi}^{\bf m})^+$.

 We recall (see e.g. \cite{EL}) that any $w^*$-continuous completely positive map $\varphi$ on $B(\cH)$ is determined by a sequence $\{C_\kappa\}_{\kappa=1}^n$ ($n\in \NN$ or $n=\infty$) of bounded operators on $\cH$, in the sense that
$$
\varphi(X)=\sum_{j=1}^n C_jXC_j^*,\qquad X\in B(\cH),
$$
where, if $n=\infty$, the convergence is in the $w^*$-topology.
The next result is an analogue of Theorem \ref{pure-contr}  for commuting $k$-tuples of $w^*$-continuous completely positive linear maps.

\begin{theorem} \label{w*-pure}  Let $\Phi:=(\varphi_1,\ldots, \varphi_k)$
be a commuting $k$-tuple of $w^*$-continuous completely positive linear maps on $B(\cH)$ and let ${\bf m}\in \NN_+^k$. Then the following statements are equivalent.
\begin{enumerate}
\item[(i)] $\Phi$ is jointly similar to a commuting $k$-tuple $\Lambda:=(\lambda_1,\ldots, \lambda_k)$ of pure $w^*$-continuous positive linear maps on $B(\cG)$, where $\cG$ is a Hilbert space, such that  $I\in \cC_{\geq}({\bf \Delta}_{\Lambda}^{\bf m})^+$.
\item[(ii)] There is   an
 invertible operator $Q\in \cC_{\geq}({\bf \Delta_{\Phi}^m})^+$ such that, for each $i\in \{1,\ldots, k\}$,   $ \varphi_i^s(Q)\to 0$ weakly as $s\to \infty$.
\item[(iii)]  There exist
constants $0<a\leq b$ and a positive operator $R\in B(\cH)$ such
that
$$
aI\leq \sum_{(s_1,\ldots,s_k)\in \ZZ_+^k}\left(\begin{matrix} s_1+m_1-1\\m_1-1\end{matrix}\right)\cdots \left(\begin{matrix} s_k+m_k-1\\m_k-1\end{matrix}\right)\varphi_1^{s_1}\circ \cdots \circ \varphi_k^{s_k}(R)\leq bI.
$$
  \end{enumerate}
\end{theorem}
\begin{proof} Assume that condition (i) holds. Then there is an invertible operator $ Y\in B(\cG, \cH)$ such that
$$
\varphi_i(YXY^*)=Y\lambda_i(X)Y^*,\qquad X\in B(\cG),
$$
for any $i\in \{1,\ldots, k\}$, and ${\bf \Delta}_{\bf \Lambda}^{\bf p}(I)\geq 0$ for any ${\bf p}\in \ZZ_+^k$ with ${\bf p}\leq {\bf m}$. Since
${\bf \Delta}_{\bf \Phi}^{\bf p}(YY^*)=Y{\bf \Delta}_{\bf \Lambda}^{\bf p}(I)Y^*$, we deduce that ${\bf \Delta}_{\bf \Phi}^{\bf p}(Q)\geq 0$, where $Q:=YY^*$ is an invertible positive operator. On the other hand, since $\varphi_i^s(Q)=Y\lambda_i^s(I)Y^*$, $s\in \NN$, we conclude that item (ii) holds. Now, we prove that (ii)$\implies$ (iii). Let $Q\in \cC_{\geq}({\bf \Delta_{\Phi}^m})^+$ be an invertible operator such that, for each $i\in \{1,\ldots, k\}$,   $ \varphi_i^s(Q)\to 0$ weakly as $s\to \infty$.
Setting
$R:={\bf \Delta_{\Phi}^m}(Q)$ and using  Theorem \ref{reproducing2} and Proposition \ref{pure2}, we obtain
\begin{equation*}\begin{split}
 \sum_{(s_1,\ldots,s_k)\in \ZZ_+^k}&\left(\begin{matrix} s_1+m_1-1\\m_1-1\end{matrix}\right)\cdots \left(\begin{matrix} s_k+m_k-1\\m_k-1\end{matrix}\right)\varphi_1^{s_1}\circ \cdots \circ \varphi_k^{s_k}(R) =Q
\end{split}
\end{equation*}
 where
 the convergence of the series is in the weak operator topology.
 Hence, we deduce item (iii). To prove the implication (iii)$\implies$ (i) we assume that item (iii) holds. Since each $\varphi_i$ is a $w^*$-continuous completely positive linear map on $B(\cH)$, there is a sequence $\{A_{i,j}\}_{j=1}^{n_i}$ $(n_i\in \cH$ or $n_i=\infty$) of bounded operators on $B(\cH)$ such that
 $\varphi_i(X)=\sum_{j=1}^{n_i} A_{i,j} X A_{i,j}^*$ for any $X\in B(\cH)$.
According to Remark \ref{more}, Theorem \ref{Berezin-prop} holds true when $f_i=q_i:=Z_{i,1}+\cdots +Z_{i,n_i}$ (even when $n_i=\infty$) and ${\bf q}:=(q_1,\ldots, q_k)$. In this
case, the generalized Berezin kernel associated with the compatible quadruple $({\bf q},{\bf m}, {\bf A}, R)$ has the property that
 for any $i\in \{1,\ldots, k\}$ and $j\in \{1,\ldots, n_i\}$,
   \begin{equation}
   \label{KASK2}{\bf K}_{\bf q,A}^R { A}^*_{i,j}= ({\bf W}_{i,j}^*\otimes I_\cR)  {\bf K}_{\bf q,A}^R,
    \end{equation}
    where $\cR:=\overline{R^{1/2} \cH}\subseteq\cH$ and ${\bf W}=\{{\bf W}_{i,j}\}$
    is the  universal model
     associated
  with the abstract noncommutative
  polydomain ${\bf D_q^m}$. Moreover, we have
  $$
({\bf K}_{\bf q,A}^R)^*\,{\bf K}_{\bf q,A}^R=
\sum_{(s_1,\ldots,s_k)\in \ZZ_+^k}\left(\begin{matrix} s_1+m_1-1\\m_1-1\end{matrix}\right)\cdots \left(\begin{matrix} s_k+m_k-1\\m_k-1\end{matrix}\right)\varphi_1^{s_1}\circ \cdots \circ \varphi_k^{s_k}(R),
$$
where the convergence is  in the weak  operator topology, which implies
\begin{equation*}
 a\|h\|^2\leq \|{\bf K}_{\bf q,A}^R h\|^2 \leq b\|h\|^2,\qquad h\in \cH.
\end{equation*}
Then $\cG:=\text{\rm range}\, {\bf K}_{\bf q,A}^R$ is a closed subspace of
$F^2(H_{n_1})\otimes \cdots F^2(H_{n_k})\otimes \cH$ and invariant under each operator
${\bf W}_{i,j}^*\otimes I_\cH$.
 Since  the operator $Y:\cH\to\cG$ defined by $Yh:={\bf K}_{\bf q,A}^R h$, $h\in \cH$,  is invertible, relation \eqref{KASK2} implies
\begin{equation} \label{AY}
A_{i,j}^*=Y^{-1}[({\bf S}_{i,j}^*\otimes I_\cH)|_\cG]Y.
\end{equation}
For all  $i\in \{1,\ldots,k\}$ and  $j\in \{1,\ldots, n_i\}$, set $T_{i,j}:=P_\cG({\bf W}_{i,j}\otimes I_\cH)|_\cG$ and define
$\lambda_i(X):=\sum_{j=1}^{n_i} T_{i,j} XT_{i,j}^*$ for any $X\in B(\cG)$.
Note that
${\bf \Delta}_{\bf \Lambda}^{\bf p}(I)=P_\cG({\bf \Delta}_{\bf q,W}^{\bf p}(I)\otimes I_\cH)|_\cG\geq  0$ for any ${\bf p}\in \ZZ_+^k$ with ${\bf p}\leq {\bf m}$.
Since   $\lambda_i^s(I)=P_\cG(\Phi_i^s(I)\otimes I_\cH)|_\cG$, $s\in \NN$,
and $\Phi_i^s(I)\to 0$ weakly  as $s\to \infty$, we deduce that
$\Lambda:=(\lambda_1,\ldots, \lambda_k)$ is a  tuple of  pure $w^*$-continuous completely positive linear maps on $B(\cG)$. On the other hand, due to relation \eqref{AY}, we have
$\varphi_i(Y^*XY)=Y^*\lambda_i(X)Y$ for any $X\in B(\cG)$ and $i\in \{1,\ldots, k\}$. Therefore, item (i) holds and the proof is complete.
\end{proof}

We remark that there is an  analogue of Proposition \ref{simi4} for  commuting $k$-tuple of   positive linear maps. Indeed, one can easily see that  if
$\Phi:=(\varphi_1,\ldots, \varphi_k)$
 is a commuting $k$-tuple of   positive linear maps on $B(\cH)$, then  $\Phi$ is jointly similar to a commuting $k$-tuple $\Lambda:=(\lambda_1,\ldots, \lambda_k)$ of   positive linear maps on $B(\cH)$ such  that
 $$ {\bf \Delta}_{\Lambda}^{\bf p}(I)\geq 0,\qquad {\bf p}\in \ZZ_+, {\bf p}\leq {\bf m},$$ if and only if there is an invertible positive  operator $R\in B(\cH)$ such that
 ${\bf \Delta}_{\bf \Phi}^{\bf p}(R)\geq 0$  for any ${\bf p}\in \ZZ_+$ with ${\bf p}\leq {\bf m}$.

We recall that the  spectral radius of a positive linear map $\varphi$ on $B(\cH)$ is defined by
$r(\varphi):=\lim_{s\to\infty}\|\varphi^k\|^{1/k}$. The analogue of Theorem \ref{simi2} for  commuting $k$-tuple of   positive linear maps is the following.

\begin{theorem} \label{w*-strict}  Let $\Phi:=(\varphi_1,\ldots, \varphi_k)$
be a commuting $k$-tuple of   positive linear maps on $B(\cH)$. Then the following statements are equivalent.
\begin{enumerate}
\item[(i)]
 $r(\varphi_i)<1$ for each $i\in \{1,\ldots,k\}$.
 \item[(ii)] $\Phi$ is jointly similar to a commuting $k$-tuple $\Lambda:=(\lambda_1,\ldots, \lambda_k)$ of  positive linear maps on $B(\cH)$, with $\lambda_i(I)<I$ for any $i\in \{1,\ldots, k\}$.
     \item[(iii)] For each ${\bf m}\in \NN_+^k$, $\Phi$ is jointly similar to a commuting $k$-tuple $\Lambda:=(\lambda_1,\ldots, \lambda_k)$ of  positive linear maps on $B(\cH)$ with  $I\in \cC_{>}({\bf \Delta}_{\Lambda}^{\bf m})^+$.
\end{enumerate}
\end{theorem}
\begin{proof} First we prove that (iii)$\implies$ (ii)$\implies$ (i). Assume that (iii) holds and fix ${\bf m}\in \NN_+^k$. Then there is an invertible operator $R\in B(\cH)$ such that $$
\varphi_i(RXR^*)=R\lambda_i(X)R^*,\qquad X\in B(\cH),
$$
for any $i\in \{1,\ldots, k\}$, and ${\bf \Delta}_{\bf \Lambda}^{\bf p}(I)> 0$ for any ${\bf p}\in \ZZ_+^k$ with ${\bf p}\leq {\bf m}$.  Consequently, $\lambda_i(I)<I$ for $i\in \{1,\ldots, k\}$. Therefore item (ii) holds. Now, we prove that (ii)$\implies$ (i).
Note that $\varphi_i^s(RR^*)=R\lambda_i^s(I)R^*$, $s\in \NN$, and
\begin{equation*}
\begin{split}
r(\lambda_i)&=\lim_{s\to\infty}\|\lambda_i^s(I)\|^{1/2s}\\
&\leq \lim_{s\to \infty}\left(\|R^{-1}\|^2 \|R\|^2 \|\varphi_i^s(I)\|\right)^{1/2s}\leq r(\varphi_i).
\end{split}
\end{equation*}
Similarly, we obtain the inequality $r(\varphi_i)\leq r(\lambda_i)$. Therefore,
$$ r(\varphi_i)= r(\lambda_i)=\lim_{s\to\infty}\|\lambda_i^s(I)\|^{1/2s}\leq \|\lambda_i(I)\|^{1/2}<1
$$
for $i\in \{1,\ldots, k\}$, which proves our assertion. Now, we prove that (i)$\implies$ (iii). Assume that $r(\varphi_i)<1$ for each $i\in \{1,\ldots,k\}$ and let $R\in B(\cH)$ be an arbitrary invertible operator. As in the proof of Theorem \ref{simi2} (implication (iii)$\implies$ (v)), we can deduce that
$$
aI\leq Q:=\sum_{(s_1,\ldots,s_k)\in \ZZ_+^k}\left(\begin{matrix} s_1+m_1-1\\m_1-1\end{matrix}\right)\cdots \left(\begin{matrix} s_k+m_k-1\\m_k-1\end{matrix}\right)\varphi_1^{s_1}\circ \cdots \circ \varphi_{k}^{s_k}(R)\leq bI
$$
for some constants $0<a\leq b$, where the convergence is in the operator norm.
Since $r(\varphi_i)<1$, we also have $\lim_{s\to\infty} \|\varphi^s(I)\|=0$, which shows that $\varphi_i$ is pure. Using Theorem  \ref{reproducing} and the continuity in norm of $\varphi_i$,
 we obtain
 $$
 {\bf \Delta_{\Phi}^m}\left[Q\right]=R>0.
 $$
 Since $\varphi_i$ is pure, Proposition \ref{Delta-ineq}  part (ii) implies
${\bf \Delta}_{\bf \Phi}^{\bf p}(Q)\geq 0$ for any ${\bf p}\in \ZZ_+^k$ with ${\bf p}\leq {\bf m}$. Consequently and using the fact that ${\bf \Delta_{\Phi}^m}\left[Q\right]>0$, we deduce that ${\bf \Delta}_{\bf \Phi}^{\bf p}(Q)>0$ for any ${\bf p}\in \ZZ_+^k$ with ${\bf p}\leq {\bf m}$.
For each $i\in \{1,\ldots, k\}$, set
$$\lambda_i:=Q^{-1/2} \varphi_i(Q^{1/2} XQ^{1/2})Q^{-1/2}
$$
and ${\bf \Lambda}=(\lambda_1,\ldots, \lambda_k)$. Now it is clear that
${\bf \Delta}_{\bf \Lambda}^{\bf p}(I)=Q^{-1/2} {\bf \Delta}_{\bf \Phi}^{\bf p}(Q)Q^{-1/2}> 0$ for any ${\bf p}\in \ZZ_+^k$ with ${\bf p}\leq {\bf m}$.
Therefore, ${\bf \Phi}$ is jointly similar to ${\bf \Lambda}$ and
$I\in \cC_{>}({\bf \Delta}_{\Lambda}^{\bf m})^+$. This completes the proof.
\end{proof}

We remark that the condition  $I\in \cC_{>}({\bf \Delta}_{\Lambda}^{\bf m})^+$ implies $\lambda_i(I)<I$ for each $i\in \{1,\ldots, k\}$, but the converse is not true. On the other hand, if $r(\varphi_i)<1$ for each $i\in \{1,\ldots,k\}$, then
the equation
  $
  {\bf \Delta_\Phi^m}(X)=R,
  $
  where
  $R\in B(\cH)$ is an invertible positive operator,
  has a
  unique  positive solution, namely,
  $$
  X:=\sum_{(s_1,\ldots,s_k)\in \ZZ_+^k}\left(\begin{matrix} s_1+m_1-1\\m_1-1\end{matrix}\right)\cdots \left(\begin{matrix} s_k+m_k-1\\m_k-1\end{matrix}\right)\varphi_1^{s_1}\circ \cdots \circ \varphi_{k}^{s_k}(R),
  $$
  where the convergence is in the uniform topology. Moreover, $X$  is an invertible operator in   $\cC_{>}({\bf \Delta}_{\Lambda}^{\bf m})^+$.

The next result is an analogue of Theorem \ref{simi} for commuting $k$-tuple of $w^*$-continuous positive linear maps. Since the proof is similar, we shall omit it.

\begin{theorem} \label{w*}  Let $\Phi:=(\varphi_1,\ldots, \varphi_k)$
be a commuting $k$-tuple of $w^*$-continuous positive linear maps on $B(\cH)$. Then the following statements are equivalent.
\begin{enumerate}
\item[(i)] $\Phi$ is jointly similar to a commuting $k$-tuple $\Lambda:=(\lambda_1,\ldots, \lambda_k)$ of $w^*$-continuous positive linear maps on $B(\cH)$, with $\lambda_i(I)=I$ for $i\in \{1,\ldots, k\}$.
\item[(ii)]
There exist constants $0<c\leq d$ such that
$$cI\leq \varphi_1^{s_1}\circ \cdots \circ  \varphi_k^{s_k}(I)\leq dI,\qquad (s_1,\ldots, s_k)\in \ZZ_+^k.
$$

\item[(iii)] There
 exist
 positive constants
 $0<c\leq d$ such that
 $$
 cI\leq \frac {1} {p^{(1)}\cdots p^{(k)}} \sum_{s_k=0}^{p^{(k)}-1}\cdots \sum_{s_1=0}^{p^{(1)}-1} \varphi_k^{s_k}\circ \cdots \circ \varphi_1^{s_1}  (I)\leq dI
 $$
 for any $p^{(1)},\ldots, p^{(k)}\in \NN$.
 \item[(iv)] There is a positive invertible operator $Q\in B(\cH)$ such that  $\varphi_i(Q)=Q$ for any $i\in \{1,\ldots,k\}$.
\end{enumerate}
\end{theorem}

\bigskip

      %\Refs
      %\widestnumber\key{BFPQR}
      %\def\n{\key}
       %

      \end{document}